\documentclass[11pt,reqno]{amsart}
\usepackage[T1]{fontenc}
\usepackage{amsmath}
\usepackage{amsfonts}
\usepackage{amssymb}
\usepackage{float}
\usepackage{graphicx}
\usepackage{color}
\usepackage[normalem]{ulem}
\usepackage{color}
\definecolor{MyLinkColor}{rgb}{0,0,0.4}
\newcommand{\R}{{\mathbb R}}
\newcommand{\E}{{\mathcal E}}
\newcommand{\N}{{\mathbb N}}

\newcommand{\cK}{\mathcal{K}}

\newcommand{\cI}{\mathcal{I}}
\newcommand{\cJ}{\mathcal{J}}
\newcommand{\cM }{\mathcal{M}}

\newcommand{\cP}{\mathcal{P}}

\newcommand{\h}{\mathfrak{h}}
\newcommand{\wt}{\widetilde}

\newcommand{\p}{\partial}

\newcommand{\e}{\varepsilon}

\newtheorem{thm}{Theorem}[section]
\newtheorem{prop}[thm]{Proposition}
\newtheorem{lemma}[thm]{Lemma}

\theoremstyle{remark} 
\newtheorem{rem}[thm]{Remark}
\numberwithin{equation}{section}
\title[Self-similarity in a thin film Muskat problem]{Self-similarity in a thin film Muskat problem}

\thanks{ }

\author[Ph. Lauren\c cot]{Philippe Lauren\c cot}
\address{Institut  de Math\'ematiques de Toulouse, UMR~5219, Universit\'e de Toulouse, CNRS, F-31062 Toulouse cedex~9, France}
\email{laurenco@math.univ-toulouse.fr}

\author[B.--V. Matioc]{Bogdan--Vasile Matioc}
\address{Institut f{\"u}r Angewandte Mathematik, Leibniz Universit{\"a}t Hannover, Welfengarten~1, 30167 Hannover, Deutschland}
\email{matioc@ifam.uni-hannover.de}

\date{\today}


\usepackage[colorlinks=true,linkcolor=MyLinkColor,citecolor=MyLinkColor]{hyperref} %dvips

\begin{document}

%%%%%%%%%%%%%%%%%%%%%%%%%%%%%%%%%%%%%%%%
\begin{abstract}
The large time behavior of non-negative weak solutions to a thin film approximation of the two-phase Muskat problem is studied. 
A classification of self-similar solutions is first provided: there is always a unique even self-similar solution while a continuum of non-symmetric 
self-similar solutions exist for certain fluid configurations. 
Despite this non-uniqueness, convergence of all non-negative weak solutions towards a self-similar solution is proved.
\end{abstract}
%%%%%%%%%%%%%%%%%%%%%%%%%%%%%%%%%%%%%%%%

\maketitle

%%%%%%%%%%%%%%%%%%%%%%%%%%%%%%%%%%%%%%%%
%%%%%%%%%%%%%%%%%%%%%%%%%%%%%%%%%%%%%%%%
\section{Introduction}\label{s:intro}
%%%%%%%%%%%%%%%%%%%%%%%%%%%%%%%%%%%%%%%%
%%%%%%%%%%%%%%%%%%%%%%%%%%%%%%%%%%%%%%%%

The purpose of this work is to investigate the large time asymptotics of a thin film approximation to the Muskat problem derived recently in \cite{EMM12}. 
It is a mathematical model describing the evolution of two immiscible and vertically superposed thin fluid layers, of different densities and viscosities,  
on a flat surface when gravity is the sole driving force.
More precisely, in a two-dimensional setting, we assume that  
the impermeable bottom of the porous medium is located at $y=0$, and we denote the thickness of the lower and upper fluids by $f=f(t,x)\ge 0$ and $g=g(t,x)\geq0$, 
respectively. The thin film Muskat problem then reads
\begin{equation}\label{eq:S2}
\left\{
\begin{array}{rcl}
\p_t f & = & \p_x\left( f \p_x\left(  (1+R) f +R g\right)\right),\\[1ex]
\p_t g & = & R_\mu\p_x\left( g \p_x\left(  f + g \right)\right),
\end{array}
\right.
\qquad (t,x)\in (0,\infty)\times \R,
\end{equation}
and appears as the singular limit of the two-phase Muskat problem when the thickness of the fluid layers vanishes.

%%%%%%%%%%%%%%%%%%%%%%%%%%%%%%%%%%%%%%%%
%%%%%%%%%%%%%%%%%%%%%%%%%%%%%%%%%%%%%%%%
\paragraph{\bf The thin film Muskat problem} 
%%%%%%%%%%%%%%%%%%%%%%%%%%%%%%%%%%%%%%%%
%%%%%%%%%%%%%%%%%%%%%%%%%%%%%%%%%%%%%%%%

The Muskat problem was proposed in \cite{Mu34} as a model for the motion of two immiscible fluids with different densities and viscosities in a porous medium, 
the intrusion of water into oil for instance. It describes the time evolution of the domains occupied by the two fluids and of the potential distributions of the fluids.
More precisely,  the space and time evolution of the thickness $f$ and $g$ of the two fluids
($h:=f+g$ being then the total height of the fluid system) and of   the potential distributions is described by the following system of equations
\begin{equation}\label{MP1}
\left\{
\begin{array}{rllllll}
\Delta u_+&=&0&\text{in} \ [f<y<h]\ , \\[1ex]
\Delta u_-&=&0&\text{in} \  [0<y<f]\ , \\[1ex]
\p_th&=&-\mu_+^{-1}\langle\nabla u_+|(-\partial_x h,1)\rangle &\text{on} \ [y=h]\ ,\\[1ex]
u_+&=&G\rho_+ h-\gamma_{h}\kappa_{h}&\text{on} \ [y=h]\ ,\\[1ex]
 { \p_yu_{-}}&=&0&\text{on} \ [y=0]\ ,\\[1ex]
u_+-u_-&=&G(\rho_+-\rho_-)f+\gamma_{f}\kappa_{f}&\text{on} \ [y=f]\ ,\\[1ex]
\p_tf&=&-\mu_\pm^{-1}\langle\nabla u_\pm|(-\partial_x f,1)\rangle &\text{on} \ [y=f]\ ,
\end{array}
\right.
\end{equation} 
with given initial data $(f,h)(0)=(f_0,h_0),$ cf. \cite{EMM12a, EMM12}. 
The interface $[y=f]$ separates the two fluids (we use the subscript $-$ for the fluid below and we refer to the fluid located 
above this interface by using the subscript $+$),
and we assume a uniform pressure, normalized to be zero, on the interface  $[y=h]$ which separates the fluid system from the air.
Moreover,
\begin{itemize}
\item $\rho_\pm$ and $\mu_\pm$ are the densities and viscosities of the fluids $\pm$, 
\item $G$ is the gravity constant,
\item $u_\pm := p_\pm + G\rho_\pm y$ are the velocity potentials,  the velocity $\mathbf{v}_\pm$ of the fluids being given by
Darcy's law $\mathbf{v}_\pm :=-\mu_\pm^{-1}\ \nabla u_\pm$,
\item $\gamma_f$ and $\kappa_{f}$ are the surface tension coefficient and curvature of the interface $[y=f]$,
\item $\gamma_h$ and $\kappa_{h}$ are the surface tension coefficient and curvature of the interface $[y=h]$.
\end{itemize}
This complex moving boundary value problem was studied in \cite{EMM12a} where it was shown to be of parabolic type for small initial data. 
This property is used to prove the well-posedness and to study the stability properties of the equilibria of \eqref{MP1} (see \cite{CG10} for a related problem).

For thin fluid layers, the full Muskat problem \eqref{MP1} is approximated in \cite{EMM12} by a strongly coupled parabolic system of
equations having only the functions $f$ and $g$ as unknowns, 
see also \cite{JM14} for a similar derivation in the context of seawater intrusion modeling.
More precisely, a new parameter $0<\e \ll 1$ is introduced in the system \eqref{MP1} to scale the thickness of the layers: the variables and the unknowns in \eqref{MP1}
are then scaled as follows
$$\begin{array}{ccc}
x=\tilde{x}, & y=\varepsilon \tilde{y}, & t=\tilde{t}/\varepsilon, \\ 
 & & \\
f(t,x)=\varepsilon \tilde{f}(\tilde{t},\wt x), & h(t,x)=\varepsilon \tilde{h}(\tilde{t},\wt x), & u_\pm(t,x,y) = \tilde{u}_\pm(\tilde{t},\wt x,\wt y) .
\end{array}
$$
Then, using formal expansions for $\wt u_\pm$ in $\e$ and omitting the tildes, one retains, at the lowest order in $\e$,  the following thin film Muskat problem
\begin{equation}\label{MP2}
\left\{
\begin{array}{rcl}
\p_t f & = & \displaystyle\p_x\left(   f\p_x\Big(\frac{G\rho_-}{\mu_-}f
+\frac{G\rho_+}{\mu_-}g-\frac{\gamma_f+\gamma_h}{\mu_-}\p_x^2f-\frac{\gamma_h}{\mu_-}\p_x^2g\Big)\right),\\[2ex]
\p_t g & = & \displaystyle\p_x\left(   g\p_x\Big(\frac{G\rho_+}{\mu_+}f+\frac{G\rho_+}{\mu_+}g-\frac{\gamma_h}{\mu_+}\p_x^2f-\frac{\gamma_h}{\mu_+}\p_x^2g\Big)\right),
\end{array}
\right.
\end{equation}
with initial data $(f,g)(0)=(f_0,g_0)$, where $g:=h-f.$ We emphasize that the cross-diffusion terms are nonlinear and have highest order. 

The existence, uniqueness, and life span of classical solutions to this limit system are studied  in \cite{EM12x} when 
considering surface tension effects at both interfaces, and in \cite{EMM12} when allowing only for gravity effects 
(which corresponds to setting $\gamma_f=\gamma_h=0$ in \eqref{MP2}). Non-negative global weak solutions on a bounded 
interval and with no-flux boundary conditions were constructed in \cite{ELM11} for $\gamma_f=\gamma_h=0$, 
and in \cite{BM12} when assuming only capillary forces. 
Weak solutions to a class of systems including \eqref{MP2} 
with $\gamma_f=\gamma_g=0$, $\mu_-=\mu_+$, and with periodic boundary conditions are also constructed in \cite{AIJMxx}.
 We subsequently uncover that the system \eqref{MP2} can be interpreted as the gradient flow of a certain energy functional with respect to the $2-$Wasserstein metric. 
This gradient flow structure allowed us  to use  tools from the  calculus of variations and to 
implement a discrete time scheme  to obtain, in the limit when the time step goes to zero, non-negative and globally defined  weak solutions of \eqref{MP2}, 
cf. \cite{LM12xx, LM12x}. While in \cite{LM12xx} we assumed $\gamma_f\gamma_h\neq 0,$ and the weak solutions are defined in $\R$ or $\R^2$, 
the solutions found in \cite{LM12x} are only subject to gravity effects and the analysis is one-dimensional. 
The uniqueness of these weak solutions is still an open question.

The above mentioned gradient flow structure is actually reminiscent from the porous medium equation (PME) 
\begin{equation}\label{PME}
\p_t f=\p_x(f\p_xf)
\end{equation}
and the thin film equation (TFE) 
\begin{equation}\label{TFE}
\p_t f=\p_x(f\p_x^3f)
\end{equation}
to which \eqref{MP2} reduces (up to a multiplicative constant) when $g=0$ and either $\gamma_f=0$ or gravity is neglected. 
Indeed, both equations are gradient flows associated to a suitable functional for the $2-$Wasserstein distance,
see \cite{GO01, MMS09, Ot98, Ot01} and the references therein. Such a gradient flow structure is rather seldom in the context
of parabolic systems and, apart from \eqref{MP2}, we are only aware of the model for diffusion of 
multiple species presented in \cite{BMPB10} and the parabolic-parabolic chemotaxis Keller-Segel system and its variants \cite{BCKKLLxx, BL13, Zixx}.

According to the discussion above, the thin film Muskat problem \eqref{MP2} can be interpreted as a 
two-phase generalization of the PME \eqref{PME} when capillary is neglected and of the TFE \eqref{TFE} when gravity is neglected. 
The large time behavior of non-negative solutions to these two equations in $\mathbb{R}^n$,  $n\ge 1,$ has been 
thoroughly investigated, see \cite{CT00, K76, K75, Ne84, Ot01, Ra84, Va07} for the PME and \cite{BPW92, CU07, CT02, MMS09} for the TFE and the references therein.
It is actually given by self-similar solutions and is a typical example of asymptotic simplification, 
in the sense that any non-negative solution converges towards the unique 
non-negative self-similar solution having the same $L_1$-norm as its initial condition. 
It is then tempting to figure out whether such a behavior is also enjoyed by \eqref{MP2} and the purpose of
this paper is to investigate thoroughly this issue when capillary forces are neglected.

More precisely, we focus on the system \eqref{eq:S2} endowed with the initial conditions
\begin{equation*} 
f(0)=f_0,\qquad g(0)=g_0.
\end{equation*}
which is obtained from \eqref{MP2} after introducing the parameters
\begin{equation}
R:=\frac{\rho_+}{\rho_- - \rho_+} , \qquad \mu:=\frac{\mu_-}{\mu_+},\qquad R_\mu:=\mu R , \label{banff1}
\end{equation}
neglecting capillary effects ($\gamma_f=\gamma_h=0$), and rescaling the space variable suitably. In the remainder of this paper
the parameters $R$ and $R_\mu$ are assumed to be positive. 
Physically, this means that the denser fluid layer is located beneath the less dense one. 

%%%%%%%%%%%%%%%%%%%%%%%%%%%%%%%%%%%%%%%%
%%%%%%%%%%%%%%%%%%%%%%%%%%%%%%%%%%%%%%%%
\paragraph{\bf Self-similar solutions} 
%%%%%%%%%%%%%%%%%%%%%%%%%%%%%%%%%%%%%%%%
%%%%%%%%%%%%%%%%%%%%%%%%%%%%%%%%%%%%%%%%

The first contribution of this paper is a classification of non-negative self-similar solutions to \eqref{eq:S2}. Let us first recall that, given $M>0$, the PME \eqref{PME} 
possesses a unique self-similar solution $f_M(t,x) = t^{-1/3} F_M(x t^{-1/3})$ which is given by the Barenblatt-Pattle profile
$$
F_M(x)=\left(a_M- \frac{x^2}{6} \right)_+,
$$
the positive constant $a_M$ being uniquely determined by the volume contraint $\|F_M\|_1=M$, see \cite{Va07} for instance. 
We note that the self-similar solution $f_M$ satisfies $\|f_M(t)\|_1=M$ for all $t\ge 0$ and that the self-similar profile $F_M$ is even and has a connected positivity set. 

Concerning \eqref{eq:S2}, a simple computation reveals that it enjoys the same scaling property as the PME \eqref{PME} and that
volume-preserving self-similar solutions shall be of the form 
\begin{equation}\label{sesi}
(f,g)(t,x) = t^{-1/3} (f_s,g_s)(x t^{-1/3})\ , \qquad (t,x)\in (0,\infty)\times \R\ .
\end{equation}
As we shall see below, the presence of a second fluid changes drastically the shape of the self-similar profiles $(f_s,g_s)$ and 
complicates the analysis a lot. 
Namely, we first show that, as the PME \eqref{PME}, the gravity driven thin film Muskat problem \eqref{eq:S2} 
has for each configuration of the fluids --viscosity, 
density and volumes-- a unique even self-similar solution. 
This solution is described in Proposition~\ref{P:-1} and illustrated in Figure~\ref{F:1}. 
It has the interesting property that, if the ratio of the viscosities is very large or very small (see Proposition~\ref{P:-1}~$(iii)$ and~$(iv)$),  
the less viscous fluid layer consists of two disconnected blobs while the other fluid  forms a blob which fills the space between the  two 
blobs of the less viscous fluid. 
Moreover, in this regime, there are other self-similar solutions which are determined by non-symmetric profiles. 
We show that there is actually a continuum of self-similar profiles parametrized by a real-analytic curve which contains the even
self-similar profile as an interior point, and all other points on this curve are non-symmetric self-similar profiles of the thin film 
Muskat problem \eqref{eq:S2}, see  Theorem~\ref{T:MT1} and Figures~\ref{F:2} and~\ref{F:3}. On the other hand, in the complement of 
this small/large viscosities ratio regime the existence of self-similar profiles, other than the even one, is excluded, see Theorem~\ref{T:MT1}.

%%%%%%%%%%%%%%%%%%%%%%%%%%%%%%%%%%%%%%%%
%%%%%%%%%%%%%%%%%%%%%%%%%%%%%%%%%%%%%%%%
\paragraph{\bf Large time behavior} 
%%%%%%%%%%%%%%%%%%%%%%%%%%%%%%%%%%%%%%%%
%%%%%%%%%%%%%%%%%%%%%%%%%%%%%%%%%%%%%%%%

The existence of a plethora of non-symmetric self-similar solutions makes the study of the asymptotic behavior of the weak solutions of \eqref{eq:S2} much more involved. 
Moreover, compared to the PME \eqref{PME}, we have a further unknown that corresponds to the height of the second fluid layer.
Due to this fact, the  problem  \eqref{eq:S2} has a higher degree of nonlinearity than the PME, being additionally doubly degenerate as 
all coefficients of the highest order spatial derivatives of \eqref{eq:S2} vanish on sets where $f=g=0$. Therefore, many techniques used when studying 
the asymptotic behavior of solutions of the  PME, e.g. the entropy method and the comparison principle fail in the context of \eqref{eq:S2}. 
Nevertheless, relaying on compactness arguments, we can still prove the convergence of the global non-negative weak solutions towards a 
self-similar solution, see Theorem~\ref{T:MT3}. 
A key observation here is that the energy computed on the continuum of self-similar profiles has some monotonicity properties.

%%%%%%%%%%%%%%%%%%%%%%%%%%%%%%%%%%%%%%%%
%%%%%%%%%%%%%%%%%%%%%%%%%%%%%%%%%%%%%%%%
\paragraph{\bf Film rupture} 
%%%%%%%%%%%%%%%%%%%%%%%%%%%%%%%%%%%%%%%%
%%%%%%%%%%%%%%%%%%%%%%%%%%%%%%%%%%%%%%%%

We emphasize that a particular feature of the gravity driven thin film Muskat problem is that it models the rupture of thin films.
This interesting phenomenon was studied by several authors in connection with model equations related to the TFE~\eqref{TFE},
see \cite{Const93, Or97, ZL99} and the references therein. In our setting, the film rupture occurs, for example, in the small/large viscosities ratio regime.
According to Theorem~\ref{T:MT3}, weak solutions corresponding to even initial configurations with both fluid layers having a connected set of positive  thickness converge 
towards the even self-similar  solution which has the property that the less viscous layer consists of two disjoint blobs.
We thus observe rupture of the less viscous fluid at least in infinite time, see the numerical simulation in Figure~\ref{F:4}. 
In fact, our simulations suggest that  the film rupture occurs in finite time.

%%%%%%%%%%%%%%%%%%%%%%%%%%%%%%%%%%%%%%%%
%%%%%%%%%%%%%%%%%%%%%%%%%%%%%%%%%%%%%%%%
\paragraph{\bf Outline} 
%%%%%%%%%%%%%%%%%%%%%%%%%%%%%%%%%%%%%%%%
%%%%%%%%%%%%%%%%%%%%%%%%%%%%%%%%%%%%%%%%

The outline of the paper is as follows. The next section is devoted to a detailed statement of the main results of this paper. 
As a preliminary step, we introduce a rescaled version \eqref{RS} of the thin film Muskat system \eqref{eq:S2} which relies in particular 
on the classical transformation to self-similar variables. The advantage of this alternative formulation is twofold: the profiles of 
non-negative self-similar solutions to \eqref{eq:S2} are non-negative stationary solutions to \eqref{RS} and it also allows us to reduce 
the study to non-negative self-similar solutions having both an $L_1$-norm equal to one. We then give a complete classification of
non-negative stationary solutions to \eqref{RS} in Theorem~\ref{T:MT1}. In particular, we identify a range of the parameters for 
which a continuum of stationary solutions exists. The convergence of any non-negative weak solution to \eqref{RS} to one of 
these stationary solutions is stated in Theorem~\ref{T:MT3}. Section~\ref{S:2} is devoted to the 
classification of self-similar profiles and the proof of Theorem~\ref{T:MT1}. After deriving some basic properties of the self-similar
profiles in Section~\ref{S:21}, we split the analysis in three parts and study first even profiles in Section~\ref{S:22} after
turning to non-symmetric profiles with either connected supports in Section~\ref{S:23} or disconnected supports in Section~\ref{S:24}. 
Identifying the supports of the profiles is at the heart of this classification and requires to solve nonlinear algebraic systems of equations in $\R^5$, 
their detailed analysis being partly postponed to the Appendix. Section~\ref{S:3} is devoted to the study of the asymptotic behavior of the weak 
solutions of the rescaled system \eqref{RS}. After recalling the existence of solutions to \eqref{RS} and their properties in Section~\ref{S:31}, 
the convergence to a stationary solution is established in Section~\ref{S:32}. The proof relies on the availability of a Liapunov functional which takes distinct values for 
different stationary solutions. In Section~\ref{S:5} we present numerical simulations which indicate that the even self-similar profile is
not the unique attractor of the system.

%%%%%%%%%%%%%%%%%%%%%%%%%%%%%%%%%%%%%%%%
%%%%%%%%%%%%%%%%%%%%%%%%%%%%%%%%%%%%%%%%
\section{Main results}\label{s:main}
%%%%%%%%%%%%%%%%%%%%%%%%%%%%%%%%%%%%%%%%
%%%%%%%%%%%%%%%%%%%%%%%%%%%%%%%%%%%%%%%%

%%%%%%%%%%%%%%%%%%%%%%%%%%%%%%%%%%%%%%%%
%%%%%%%%%%%%%%%%%%%%%%%%%%%%%%%%%%%%%%%%
\subsection{Alternative formulations}\label{S:main1}
%%%%%%%%%%%%%%%%%%%%%%%%%%%%%%%%%%%%%%%%
%%%%%%%%%%%%%%%%%%%%%%%%%%%%%%%%%%%%%%%%

The system \eqref{eq:S2} is a parabolic system with a double degeneracy: the eigenvalues of the matrix associated to the right-hand side of \eqref{eq:S2} are 
non-negative and they vanish both if $f=g=0$. 
A natural framework to work with is thus that of weak solutions and the analysis performed in \cite{LM12x} is dedicated to proving existence of non-negative global 
weak solutions to \eqref{eq:S2} corresponding 
to initial data $(f_0,g_0)$ which are probability densities in $\R$ and belong to $L_2(\R)$. However, as mentioned in the discussion following \cite[Remark~1.2]{LM12x},
one may consider arbitrary non-negative
initial  data by  simply introducing an additional scaling factor in \eqref{eq:S2}. More precisely, given non-negative initial data $(f_0,g_0) $ satisfying 
$ f_0, g_0 \in L_1(\R, (1+x^2) dx)\cap L_2(\R)$ and $f_0, g_0\not\equiv 0$, 
we define $\eta\in (0,\infty)$ by  $\eta^2 := \|f_0\|_1/\|g_0\|_1.$  Then, 
if $(f,g)$ is a global weak solution to \eqref{eq:S2} corresponding to  $(f_0,g_0)$, then setting
\begin{equation}\label{tr1}
\phi( t,x):=\frac{f(t \|g_0\|_1^{-1},x)}{\|f_0\|_1} \quad\text{ and }\quad \psi( t,x):=\frac{g( t \|g_0\|_1^{-1},x)}{\|g_0\|_1}
\end{equation}
for $(t,x)\in [0,\infty)\times\R,$ we see that $\left( \phi, \psi \right)$ solves the system
\begin{equation}\label{Pb}
\left\{
\begin{array}{rcl}
\p_{t}  \phi & = & \p_x\left(  \phi\ \p_x \left( (1+R) \eta^2  \phi + R \psi \right) \right),  \\[1ex]
\p_{t}  \psi & = & R_\mu \p_x\left(  \psi\ \p_x\left(\eta^{2}  \phi + \psi \right) \right),
\end{array}
\right.
\qquad (t,x)\in (0,\infty)\times \R \ ,
\end{equation}
with initial data
\begin{equation*}\label{PbIn}
\left( \phi, \psi \right)(0) = (\phi_0,\psi_0) := \left( \frac{f_0}{\|f_0\|_1} , \frac{g_0}{\|g_0\|_1} \right)\ .
\end{equation*}
Introducing the set
\begin{equation*}
\cK:=\left\{w\in L_1(\R, (1+x^2) dx) \cap L_2(\R)\,:\, w\ge 0 \text{ a.e. and } \|w\|_1 = 1 \right\}\ ,
\end{equation*}
it follows from \cite{LM12x} that, given $(\phi_0,\psi_0)\in\cK^2$, there is a global weak solution $\left( \phi, \psi \right)$ of \eqref{Pb} 
with initial data $(\phi_0,\psi_0)$ such that $\left( \phi(t), \psi(t) \right) \in \cK^2$ for all $t\geq 0,$ and the mapping  $t \mapsto \E\left( \phi(t), \psi(t) \right)$
is non-increasing a.e. in $(0,\infty)$.  Here,  $\E$ denotes the energy functional
\begin{equation} 
\E(u,v):=\frac{\eta^2}{2}\ \| u \|_2^2 + \frac{R}{2}\ \left\| \eta\ u + \eta^{-1}\ v \right\|_2^2\ , \quad (u,v)\in \cK^2 \ . \label{energy} 
\end{equation}
In fact, the system \eqref{Pb} is the gradient flow of the energy functional $\E$ with respect to the $2-$Wasserstein metric \cite{LM12x}. 

A further transformation of \eqref{Pb} involves the so-called self-similar variables and reads
\begin{equation}\label{eq:Cha2}
 (\bar{f},\bar{g})(t,x):=e^{t/3} \left( \phi , \psi \right) \left(e^t-1,xe^{t/3}\right), \qquad (t,x)\in [0,\infty)\times\R \ .
\end{equation}
Then, setting $\left( \bar{f}_0, \bar{g}_0 \right) := \left( \phi_0,\psi_0 \right)$ and dropping the bars to simplify the notation, we end up with the following 
rescaled system
\begin{equation}\label{RS}
\left\{
\begin{array}{lll}
\p_t f&=&\p_x\left( f \p_x(\eta^2 (1+R) f + R g + x^2/6)\right)\ ,\\[1ex]
\p_t g&=&\p_x\left( g \p_x(\eta^{2}R_\mu f + R_\mu   g + x^2/6)\right)\ ,
\end{array}
\right.\qquad (t,x)\in(0,\infty)\times\R\ ,
\end{equation}
with initial data $\left( f_0, g_0 \right)\in \cK^2.$ 
In addition, it clearly follows from the properties of $\left( \phi , \psi \right)$ and \eqref{eq:Cha2} that $\left( f(t) , g(t) \right)$ belong to $\cK^2$ for all times $t\ge 0$ and that $t\mapsto \E_*\left( f(t) , g(t) \right)$ is a non-increasing function a.e. in $(0,\infty)$. We have introduced here the rescaled energy $\E_*$ through
\begin{equation} \label{eq:sup3}
\E_{*}(u,v) := \E(u,v) + \frac{1}{6} \cM_{2}(u,v)\ , \quad (u,v)\in\cK^2\ , 
\end{equation}
with
\begin{equation}
\cM_{2}(u,v) := \int_\R \left( u + \Theta v \right)(x)\ x^2\, dx \quad \text{ and }\quad \Theta := \frac{R}{\eta^2R_\mu}=\frac{1}{\mu\eta^2} \ . \label{eq:sup1c}
\end{equation}
The main feature of \eqref{RS} is that, if $\left( \phi , \psi \right)$ is a self-similar solution of \eqref{Pb} of the form \eqref{sesi}, that is,
\begin{equation}
(\phi,\psi)(t,x) = t^{-1/3} (F,G)\left( x t^{-1/3} \right)\ , \qquad (t,x)\in (0,\infty)\times\R\ , \label{duboismortier}
\end{equation}
then the corresponding self-similar profile $(F , G)$ is a stationary solution to \eqref{RS}. Such a property is also useful when studying the attracting properties
of the self-similar solutions to \eqref{Pb}. Indeed, it amounts to the stability of steady-state solutions to \eqref{RS}.

%%%%%%%%%%%%%%%%%%%%%%%%%%%%%%%%%%%%%%%%
%%%%%%%%%%%%%%%%%%%%%%%%%%%%%%%%%%%%%%%%
\subsection{Main results}\label{S:main2}
%%%%%%%%%%%%%%%%%%%%%%%%%%%%%%%%%%%%%%%%
%%%%%%%%%%%%%%%%%%%%%%%%%%%%%%%%%%%%%%%%

We enhance that the value of the ratio $\mu$ of the viscosities of the fluids was not important when proving the existence of weak solutions for \eqref{Pb} on the real 
line or on a bounded interval. Also, when studying the asymptotic properties of weak and strong solutions defined on a bounded interval, the viscosities influence just 
the rate at which the solutions converge towards the (flat) equilibria.
In this setting though, it  turns out that, for fixed densities, $\mu$ is the parameter which determines the shape of the self-similar solutions of \eqref{Pb}. 
In other words, once $R$ and $\eta$ are fixed, the structure of the steady-state solutions to \eqref{RS} varies according to the values of $R_\mu$ and is described in 
the next theorem. 
For further use we set
\begin{equation}
\begin{split}
R_\mu^0(R,\eta) & := R + \frac{\eta^2}{1+\eta^2}\ , \quad
R_\mu^+(R,\eta) := R + \left( \frac{1+\eta^2}{\eta^2} \right)^2\ , \\ 
R_\mu^-(R,\eta) & := \frac{R^3(1+R)}{R^3+(\eta^2(1+R)+R)^2}\ .
\end{split} \label{spip}
\end{equation}

%%%%%%%%%%%%%%%%%%%%%%%%%%%%%%%%%%%%%%%%
\begin{thm}[Classification of self-similar profiles]\label{T:MT1}
Let $R$, $R_\mu$, and $\eta$ be positive constants. 
Then, the following hold.
\begin{itemize}

\item[$(i)$] There exists a unique even stationary solution $(F_0,G_0) \in \cK^2\cap H^1(\R,\R^2)$ of \eqref{RS}.

\item[$(ii)$] If $R_\mu\not \in \big[ R_\mu^-(R,\eta),R_\mu^+(R,\eta) \big]$, then there are a bounded interval $\Lambda := [\ell_-,\ell_+]$ containing zero 
and a one-parameter family $(F_\ell,G_\ell)_{\ell\in\Lambda}\subset \cK^2\cap H^1(\R,\R^2)$ of stationary solutions of \eqref{RS} which are non-symmetric if $\ell\ne 0$. 
In addition, $(F_\ell,G_\ell)$ depends continuously on $\ell\in\Lambda$ and even analytically on $\ell\in (\ell_-,\ell_+)$. 

\item[$(iii)$] Setting $\Lambda:=\{0\}$ and $\ell_-=\ell_+=0$ for $R_\mu\in \big[ R_\mu^-(R,\eta),R_\mu^+(R,\eta) \big]$, any steady-state solution of \eqref{RS} 
belongs to the family $(F_\ell,G_\ell)_{\ell\in\Lambda}$.

\item[$(iv)$] The map $\ell\mapsto \mathcal{E}_*(F_\ell,G_\ell)$ is decreasing on $[\ell_-,0]$ and increasing on $[0,\ell_+]$.

\end{itemize}

Furthermore, there are $R_\mu^M(R,\eta)>R_\mu^+(R,\eta)$ and $R_\mu^m(R,\eta)\in (0,R_\mu^-(R,\eta))$ such that 

\begin{itemize}

\item[$(v)$] If $R_\mu\not \in \big[ R_\mu^-(R,\eta),R_\mu^+(R,\eta) \big]$ and $\ell\not\in \{\ell_-,\ell_+\}$, then either $F_\ell$ or $G_\ell$ has a disconnected support.

\item[$(vi)$] If $R_\mu\not \in \big( R_\mu^m(R,\eta),R_\mu^M(R,\eta) \big)$ and $\ell \in\{\ell_-,\ell_+\}$, then both $F_\ell$ and $G_\ell$ have connected supports.

\item[$(vii)$] If $R_\mu\in (R_\mu^m(R,\eta),R_\mu^-(R,\eta))\cap (R_\mu^+(R,\eta),R_\mu^M(R,\eta))$ and $\ell \in\{\ell_-,\ell_+\},$ then either 
$F_\ell$ or $G_\ell$ has a disconnected support.

\end{itemize}

The threshold value $R_\mu^M(R,\eta)$ is actually the unique solution in $(R+1,\infty)$ of \eqref{RmuMc} while $R_\mu^m(R,\eta)$ is the 
unique solution in $(0,R)$ of \eqref{Rmums}.

\end{thm}
%%%%%%%%%%%%%%%%%%%%%%%%%%%%%%%%%%%%%%%%

The analysis performed below actually gives more information on the continuum $(F_\ell,G_\ell)_{\ell\in\Lambda}$ of stationary solutions of \eqref{RS}.
In particular, explicit formulas are available, see Proposition~\ref{P:-1} for the even solutions and Propositions~\ref{P:0} and~\ref{P:1} 
for the non-symmetric solutions with connected supports and disconnected supports, respectively. 
In addition, if $\ell\in \Lambda$, the reflection $x\mapsto (F_\ell(-x),G_\ell(-x))$ of $(F_\ell,G_\ell)$ is also a stationary 
solution to \eqref{RS} owing to the invariance of \eqref{RS} by reflection, so that there is $\ell'\in \Lambda$ such 
that $(F_\ell(-x),G_\ell(-x))=(F_{\ell'}(x),G_{\ell'}(x))$ for $x\in\R$. 
It is also worth pointing out that the interval $\Lambda$ depends on $R$, $\eta$, and $R_\mu$.

The proof of Theorem~\ref{T:MT1} is rather involved and relies on a detailed study of the connected components of the positivity sets of $F$ and $G$. 
The first step is to identify the number and location of these connected components. 
In doing so, we end up with systems of three to five algebraic equations. 
Each solution of one of these systems satisfying suitable constraints corresponds to a stationary solution of \eqref{RS} and 
the second step is to figure out for which values of the parameters $(R,R_\mu,\eta)$ these systems have solutions satisfying the constraints already mentioned. 
In particular, one of these systems turns out to be underdetermined and is the reason for getting a continuum of steady-solutions in some cases.

An important feature revealed by Theorem~\ref{T:MT1} is that the value of the energy selects at most two stationary solutions in the 
continuum $(F_\ell,G_\ell)_{\ell\in\Lambda}$ (when $\Lambda\ne \{0\}$). This property is the cornerstone of the proof of the next result dealing with the large time
behavior of the solutions to \eqref{RS}.

%%%%%%%%%%%%%%%%%%%%%%%%%%%%%%%%%%%%%%%%
\begin{thm}[Convergence towards a steady state]\label{T:MT3}
Let $R$, $R_\mu$, and $\eta$ be given positive constants and consider $(f_0,g_0)\in\cK^2$.
There are $\ell\in\Lambda$ and a stationary solution $(F_\ell,G_\ell)$ of \eqref{RS} 
such that the weak solution $(f,g)$ of \eqref{RS} with initial data $(f_0,g_0)$ constructed in Theorem~\ref{T:WP} satisfies
\begin{equation}\label{CV}
\lim_{t\to\infty} (f(t),g(t))=(F_\ell, G_\ell) \;\;\text{ in }\;\; L_1(\R, (1+x^2) dx,\R^2) \cap L_2(\R,\R^2)\ . 
\end{equation}
Additionally, if $f_0$ and $g_0$ are both even, then $\ell=0$.
\end{thm}
%%%%%%%%%%%%%%%%%%%%%%%%%%%%%%%%%%%%%%%%

Owing to the gradient flow structure of \eqref{RS}, the outcome is quite obvious when $R_\mu\in \big[ R_\mu^-(R,\eta),R_\mu^+(R,\eta) \big]$ since \eqref{RS} has a 
unique stationary solution by Theorem~\ref{T:MT1}. 
This contrasts markedly with the situation for $R_\mu\not\in [R_\mu^-(R,\eta),R_\mu^+(R,\eta)]$ where there is a continuum of stationary solutions of \eqref{RS}. 
However, thanks to Theorem~\ref{T:MT1}~$(iv)$, there are at most two steady states having the same energy, a property which allows us to exclude the non-convergence of 
the trajectory with the help of the connectedness of the $\omega$-limit set. 

Theorem~\ref{T:MT3} guarantees the convergence of any trajectory of \eqref{RS} to a steady state but provides no information on the speed of convergence. 
At this point there is a major difference between the system \eqref{RS} and the porous medium equation written in self-similar variables
\begin{equation}
\partial_t f = \partial_x \left( f \partial_x (f + x^2/6) \right)\ , \qquad (t,x)\in (0,\infty)\times \R\ . \label{rsPME}
\end{equation}
The exponential convergence of weak solutions of \eqref{rsPME} towards the corresponding Barenblatt-Pattle profile is obtained by showing the exponential
decay of the relative entropy, the 
latter being a consequence of the exponential decay of the entropy dissipation, see \cite{CJMTA01, CT00, Va07} and the references therein.
Coming back to the system \eqref{RS}, if 
a weak solution $(f,g)$ of \eqref{RS} converges to some steady state $(F_\infty, G_\infty)$, the 
relative entropy $\E_*((f,g)|(F_\infty, G_\infty))$ is 
\begin{equation}\label{R:E}
\E_*((f,g)|(F_\infty, G_\infty)):=\E_*(f,g)-\E_*(F_\infty, G_\infty)\geq 0,
\end{equation}
and the entropy dissipation $\cI(f,g)$ is 
\begin{eqnarray}\label{EP}
2 \cI(f,g) & := & \int_\R f\left( \eta^2(1+R)\p_x f+R\p_x g+\frac{x }{3}\right)^2\ dx \nonumber\\
& & \ + \Theta \int_\R g \left( \eta^2 R_\mu \p_x f + R_\mu \p_x g+\frac{x }{3}\right)^2 \  dx\ ,
\end{eqnarray}
see Theorem~\ref{T:WP}~$(iv)$. However, the entropy/entropy dissipation approach which proves successful for \eqref{rsPME} does not seem to extend 
easily to the system \eqref{RS}. 
One reason is likely to be that, since there may exist several steady-state solutions, the choice of the relative entropy becomes unclear. 
Moreover, it is not   clear whether $\cI(f,g)$ is a decreasing function of time.

%%%%%%%%%%%%%%%%%%%%%%%%%%%%%%%%%%%%%%%%
%%%%%%%%%%%%%%%%%%%%%%%%%%%%%%%%%%%%%%%%
\section{Self-similar profiles}\label{S:2}
%%%%%%%%%%%%%%%%%%%%%%%%%%%%%%%%%%%%%%%%
%%%%%%%%%%%%%%%%%%%%%%%%%%%%%%%%%%%%%%%%

According to the discussion in Section~\ref{S:main1} the profiles $(F,G)$ of self-similar solutions of \eqref{Pb} defined in \eqref{duboismortier} are
steady-state solutions of \eqref{RS} and thereby satisfy the equations
\begin{equation}\label{ststa}
\left\{
\begin{array}{rllll}
 F\p_x\left( (1+R) \eta^2  F + R  G+x^2/6 \right)&=&0\ ,\\[1ex]
 G\p_x\left( R_\mu \eta^{2} F + R_\mu   G+x^2/6 \right)&=&0\ ,
\end{array}
\right. \qquad\text{a.e. in }\; \R\ ,
\end{equation}
with $(F,G)\in \cK^2\cap H^1(\R,\R^2)$. Note that these properties guarantee in particular that neither $F$ nor $G$ vanishes identically. 
The aim of this section is to classify all solutions of \eqref{ststa}.

To this end, let $(F,G)\in \cK^2\cap H^1(\R,\R^2)$ be a solution of \eqref{ststa} and define the positivity sets $\cP_F$ and $\cP_G$ of $F$ and $G$ by
$$
\cP_F := \left\{ x\in\R\ :\ F(x)>0 \right\} \ , \quad \cP_G := \left\{ x\in\R\ :\ G(x)>0 \right\} \ .
$$
We notice that $\cP_F$ and $\cP_G$ are both non-empty as $\|F\|_1=\|G\|_1=1$ and open as $F$ and $G$ are continuous on $\R$. It can be easily seen from \eqref{ststa} that:
\begin{itemize}
\item If $\cI$ is an interval in $\cP_F\cap \cP_G$, then there are $(a, b)\in\R^2$ such that
\begin{equation}\label{F1}
\begin{aligned}
 \eta^2 F(x) & = a-b-\frac{R_\mu-R}{6R_\mu}x^2\ ,\qquad x\in \cI\ ,\\
G(x) & = \frac{(1+R)b-aR}{R}-\frac{1+R-R_\mu}{6R_\mu}x^2\ ,\qquad x\in \cI\ .\\
\end{aligned}
\end{equation}
\item If $\cI$ is an interval in $\cP_F\setminus\cP_G,$ then there is $a\in\R$ such that
\begin{equation}\label{F2}
\eta^2 F(x)=\frac{a}{1+R}-\frac{x^2}{6(1+R)}\ , \quad G(x) = 0\ , \qquad x\in \cI\ .
\end{equation}
\item If $\cI$ is an interval in $\cP_G\setminus\cP_F,$ then there is $b\in\R$ such that
\begin{equation}\label{F3}
\eta^2 F(x) = 0\ , \quad G(x) = \frac{b}{R} - \frac{x^2}{6R_\mu}\ , \qquad x\in \cI\ .
\end{equation}
\end{itemize}

We emphasize here that the parameters $a$ and $b$ are likely to depend upon the interval $\cI$.

%%%%%%%%%%%%%%%%%%%%%%%%%%%%%%%%%%%%%%%%
%%%%%%%%%%%%%%%%%%%%%%%%%%%%%%%%%%%%%%%%
\subsection{First properties}\label{S:21}
%%%%%%%%%%%%%%%%%%%%%%%%%%%%%%%%%%%%%%%%
%%%%%%%%%%%%%%%%%%%%%%%%%%%%%%%%%%%%%%%%

We collect in this section several basic properties of solutions of \eqref{ststa}.

%%%%%%%%%%%%%%%%%%%%%%%%%%%%%%%%%%%%%%%%
\begin{lemma}\label{L:T1} Let  $(F,G)$ be a solution to \eqref{ststa}. Then:
\begin{itemize}
\item[$(i)$] $\cP_F\cap \cP_G \neq\emptyset$.
\item[$(ii)$]  Every connected component of  $\cP_F$ and $ \cP_G$ is bounded.
\item[$(iii)$] If $R_\mu>R$ and $\cI$ is a connected component of $\cP_F$, then $0\in\cI$ and $\cP_F$ is an interval of $\R$.
\item[$(iv)$] If $R_\mu<1+R$ and $\cI$ is a connected component of $\cP_G$, then $0\in\cI$ and $\cP_G$ is an interval of $\R$.
\item[$(v)$] If $\cI$ is a connected component of $\cP_F$ (resp. $\cP_G$) with $0\not\in\cI,$ then $\cI\cap\cP_G \ne \emptyset$ (resp. $\cI\cap\cP_F \ne \emptyset$).
\end{itemize}
\end{lemma}
%%%%%%%%%%%%%%%%%%%%%%%%%%%%%%%%%%%%%%%%

\begin{proof}
$(i)$: Assume for contradiction that $\cP_F\cap \cP_G = \emptyset$ and  let $\cI=(\alpha,\beta)$ be a connected component of $\cP_F.$ 
Then $F$ is given by \eqref{F2} on $\cI$ for some $a\in\R$ and satisfies $F(\alpha)=F(\beta)=0$. Consequently, $a>0$ and $\beta=-\alpha=\sqrt{6a}$. 
Similarly, recalling that $G\not\equiv 0$, if $\cJ$ is a connected component of $\cP_G,$ it follows from \eqref{F3} that there is $b>0$ 
such that $\cJ = (-\sqrt{6bR_\mu/R},\sqrt{6bR_\mu/R})$. 
Thus, $0\in \cI\cap\cJ\neq\emptyset$ and this contradicts $\cI\cap\cJ \subset \cP_F\cap \cP_G$.

\smallskip

\noindent $(ii)$: Assume first for contradiction that $\cP_F\cap \cP_G$ has an unbounded connected component $\cI$. 
Then $(F,G)$ are given by \eqref{F1} in $\cI$ for some $(a,b)\in\R^2$ and their non-negativity implies that $R_\mu\le R$ and $1+R\le R_\mu$, whence a contradiction.
Assume next for contradiction that $\cP_F$ has an unbounded connected component $\cJ$.
Owing to the just established boundedness of  the connected components of $\cP_F\cap \cP_G$, there is $r>0$ such that $G(x)=0$ for all $x\in\cJ$ such that $|x|>r$.
Then, $F$ is given by \eqref{F2} on that set which contradicts its non-negativity. This proves the claim for  $\cP_F$, the assertion for  $\cP_G$ following  by a
similar argument.

\smallskip

\noindent $(iii)$: Let $\cI$ be a connected component of $\cP_F$ and recall that it is a bounded interval by~$(ii)$. 
Assume for contradiction that $\cI\subset (0,\infty)$. 
According to \eqref{F1} and \eqref{F2}, for $x\in\cI$, $\eta^2\partial_x F(x)$ is given either by $-(R_\mu-R)x/3 R_\mu$ if $G(x)>0$ or by $-x/3(1+R)$ if $G(x)=0$.
Therefore, $\partial_x F<0$ in $\cI$ which, together with the continuity of $F$, entails that $F$ is decreasing in $\cI$ and contradicts the fact
that $F$ vanishes at both ends of $\cI$. 
A similar argument rules out the possibility that $\cI \subset (-\infty,0)$ and completes the proof.

\smallskip

\noindent $(iv)$: The proof is similar to that of~$(iii)$.

\smallskip

\noindent $(v)$: Consider a connected component $\cI$ of $\cP_F$ and assume that $\cI\subset (0,\infty)$. 
Assuming for contradiction that $\cI\cap \cP_G = \emptyset$, we readily infer from \eqref{F2} 
that $F$ is decreasing in $\cI$ which contradicts that $F$ vanishes at both ends of $\cI$ (recall that $\cI$ is bounded by~$(ii)$). 
We argue in a similar way if $\cI \subset (-\infty,0)$. 
\end{proof}

We next notice some invariance properties of \eqref{ststa} which can be checked by direct computations and allow us to reduce the range of the
parameters $R$, $R_\mu$, and $\eta$ to study.

%%%%%%%%%%%%%%%%%%%%%%%%%%%%%%%%%%%%%%%%
\begin{lemma}\label{L:2.4c}
Let $(F,G)$ be a solution of \eqref{ststa} with parameters $(R,R_\mu,\eta)$. Then
\begin{itemize}
\item[$(i)$] $x\mapsto (F(-x),G(-x))$ is also a solution of \eqref{ststa} with parameters $(R,R_\mu,\eta)$.
\item[$(ii)$] Introducing
\begin{equation}
R_{\mu,1} := \frac{R(1+R)}{R_\mu}\ , \quad \eta_1 := \frac{1}{\eta} \sqrt{\frac{R}{1+R}} \ , \label{spirou}
\end{equation}
and
\begin{equation}
\lambda := \left( \frac{\eta^2 R_\mu}{R} \right)^{1/3}\ , \quad F_1(x) := \lambda G(\lambda x)\ , G_1(x) := \lambda F(\lambda x)\ , \quad x\in \R\ , \label{fantasio}
\end{equation}
the pair $(F_1,G_1)$ belongs to $\cK^2\cap H^1(\R,\R^2)$ and is a solution of \eqref{ststa} with parameters $(R,R_{\mu,1},\eta_1)$ instead of $(R,R_\mu,\eta)$.
\end{itemize}
\end{lemma}
%%%%%%%%%%%%%%%%%%%%%%%%%%%%%%%%%%%%%%%%

%%%%%%%%%%%%%%%%%%%%%%%%%%%%%%%%%%%%%%%%
%%%%%%%%%%%%%%%%%%%%%%%%%%%%%%%%%%%%%%%%
\subsection{Even self-similar profiles}\label{S:22}
%%%%%%%%%%%%%%%%%%%%%%%%%%%%%%%%%%%%%%%%
%%%%%%%%%%%%%%%%%%%%%%%%%%%%%%%%%%%%%%%%

The observation~$(i)$ in Lemma~\ref{L:T1} is the starting point of the classification of even solutions of \eqref{ststa}.

%%%%%%%%%%%%%%%%%%%%%%%%%%%%%%%%%%%%%%%%
\begin{prop}[Classification of  even self-similar profiles] \label{P:-1} Let $R,R_\mu$, and $\eta$ be given positive parameters. 
There is a unique even solution $(F,G)$ of \eqref{ststa} with parameters $(R,R_\mu,\eta)$ which is given by:
\begin{itemize}
\item[$(i)$] If $R_\mu=R_\mu^0(R,\eta),$ then $\cP_F=\cP_G=(-\beta,\beta)$ and
$$
F(x)=G(x)=\frac{R_\mu-R}{6\eta^2R_\mu}(\beta^2-x^2)\ ,\qquad x\in \cP_F=\cP_G=(-\beta,\beta)\ ,
$$
where $\beta>0$ is defined in \eqref{c2:V}.

\item[$(ii)$] If $R_\mu\in (R_\mu^0(R,\eta),R_\mu^+(R,\eta)),$ then $\cP_F=(-\beta,\beta)$, $\cP_G=(-\gamma,\gamma)$,
and
\begin{align*}
\eta^2 F(x) = & \frac{R_\mu-R}{6R_\mu}(\beta^2-x^2)\ ,\qquad x\in \cP_F=(-\beta,\beta)\ ,\\
 & \\
G(x) = &
\left\{
\begin{array}{ll}
\displaystyle\frac{\gamma^2}{6R_\mu}+ \frac{R-R_\mu}{6R_\mu}\beta^2-\frac{1+R-R_\mu}{6R_\mu}x^2\ ,& |x|\leq\beta\ ,\\
 & \\
\displaystyle \frac{1}{6R_\mu}(\gamma^2-x^2)\ , & \beta\le |x|\le \gamma\ ,
\end{array}
\right. 
\end{align*} 
where $0<\beta<\gamma$ are defined by \eqref{c2:V}.

\item[$(iii)$] if $R_\mu \ge R_\mu^+(R,\eta)$, then $\cP_F=(-\beta,\beta)$, $\cP_G=(-\gamma,-\alpha)\cup (\alpha,\gamma)$, and
\begin{align*}
&\eta^2 F(x)=
\left\{
\begin{array}{ll}
\displaystyle\frac{R_\mu-R}{6 R_\mu}\beta^2+\frac{R(1+R-R_\mu)}{6 R_\mu(1+R)}\alpha^2-\frac{x^2}{6 (1+R)}\ ,& |x|\leq\alpha,\\
 & \\
\displaystyle \frac{R_\mu-R}{6 R_\mu}(\beta^2-x^2), & \alpha\le |x|\le \beta\ ,
\end{array}
\right.  \\
& G(x)=
\left\{
\begin{array}{ll}
\displaystyle \frac{1+R-R_\mu}{6R_\mu}(\alpha^2-x^2)\ , & \alpha\le |x|\leq\beta,\\
& \\
\displaystyle\frac{1}{6R_\mu}(\gamma^2-x^2)\ ,& \beta\le |x|\le \gamma\ ,\\
\end{array}
\right. 
\end{align*} 
where $0\le \alpha<\beta<\gamma$ is the solution of \eqref{eq:41}-\eqref{eq:43}.  

\item[$(iv)$] if $R_\mu\le R_\mu^-(R,\eta)$ then $\cP_F=(-\gamma,-\alpha)\cup (\alpha,\gamma)$, $\cP_G = (-\beta,\beta)$, and
\begin{align*}
&\eta^2 F(x)=
\left\{
\begin{array}{ll}
\displaystyle \frac{R_\mu-R}{6 R_\mu}(\alpha^2-x^2), & \alpha\le |x|\leq\beta\ ,\\
 & \\
\displaystyle\frac{1}{6 (1+R)}(\gamma^2-x^2)\ ,& \beta\le |x|\le \gamma,
\end{array}
\right.  \\
&G(x)=
\left\{
\begin{array}{ll}
\displaystyle\frac{R_\mu-R}{6R_\mu}\alpha^2+\frac{1+R-R_\mu}{6R_\mu}\beta^2-\frac{x^2}{6R_\mu} \ ,& |x|\leq\alpha\ , \\
& \\
\displaystyle \frac{1+R-R_\mu}{6R_\mu}(\beta^2-x^2)\ , & \alpha\le |x|\le \beta\ ,
\end{array}
\right. 
\end{align*} 
where $0\le \alpha<\beta<\gamma$ is the solution of \eqref{eq:51}-\eqref{eq:53} . 

\item[$(v)$] if $R_\mu\in (R_\mu^-(R,\eta),R_\mu^0(R,\eta)),$ then $\cP_F=(-\gamma,\gamma)$, $\cP_G=(-\beta,\beta)$, and
\begin{align*}
&\eta^2 F(x)=\left\{
\begin{array}{ll}
\displaystyle \frac{\gamma^2}{6 (1+R)}-\frac{R(1+R-R_\mu)}{6 (1+R)R_\mu}\beta^2-\frac{R_\mu-R}{6 R_\mu}x^2\ , & |x|\leq\beta\ ,\\
 & \\
\displaystyle\frac{1}{6 (1+R)}(\gamma^2-x^2)\ ,& \beta\le |x|\le \gamma\ ,
\end{array}
\right. \\
&G(x)=\frac{1+R-R_\mu}{6R_\mu}(\beta^2-x^2)\quad\text{ in }\;\; \cP_G=(-\beta,\beta)\ ,
\end{align*} 
where $0<\beta<\gamma$ is the solution of \eqref{c8:V}.
\end{itemize}
\end{prop}
%%%%%%%%%%%%%%%%%%%%%%%%%%%%%%%%%%%%%%%%

%%%%%%%%%%%%%%%%%%%%%%%%%%%%%%%%%%%%%%%%
\begin{figure}
\includegraphics[scale=0.71]{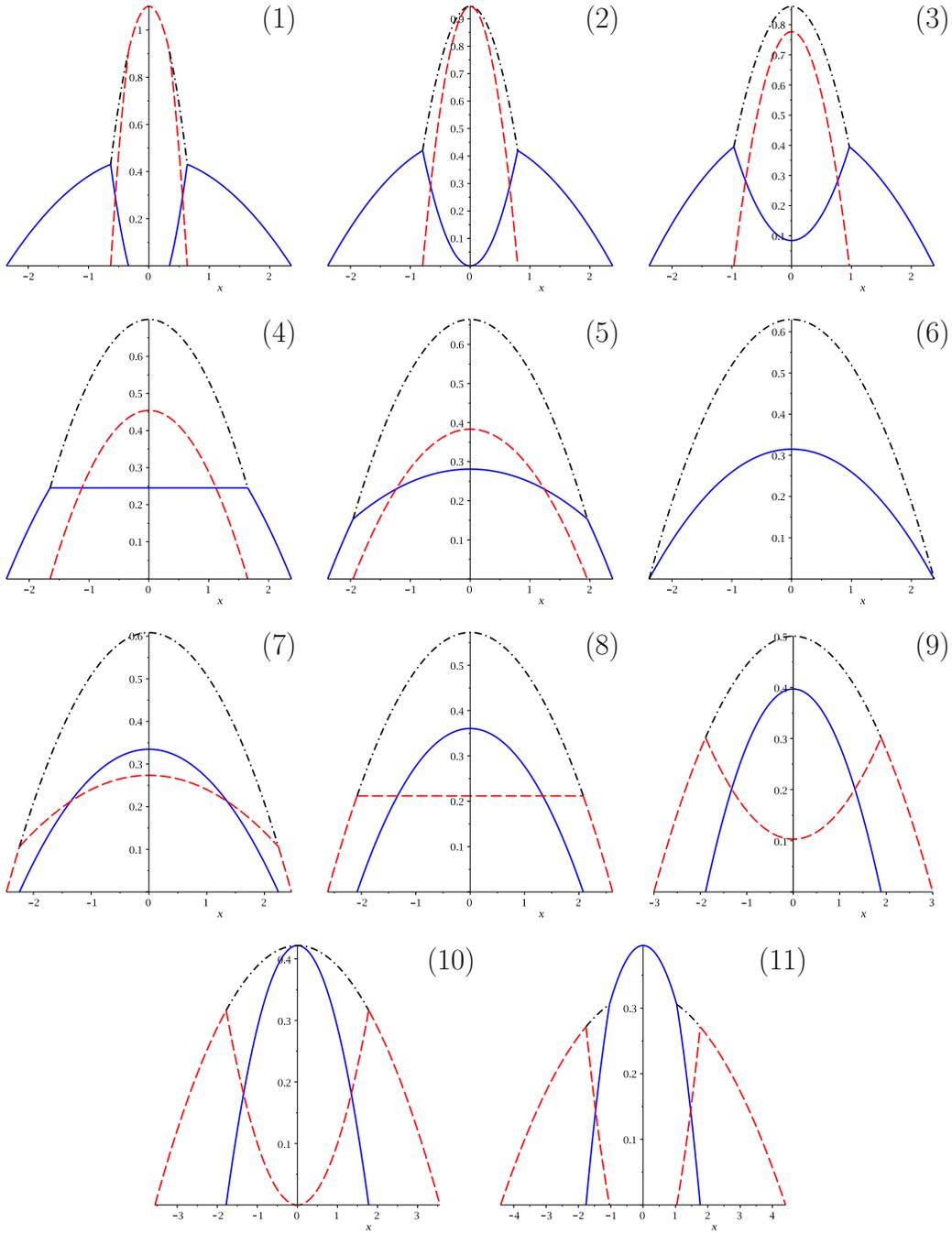}
\caption[Even self-similar profiles]{\small Even self-similar profiles of \eqref{Pb}  for $\eta=R=1$, and: 
   $(1)$ $R_\mu=0.1$; $(2)$ $R_\mu=0.2$;  
   $(3)$ $R_\mu=1/3$; $(4)$ $R_\mu=1$;
   $(5)$ $R_\mu=1.25$; $(6)$ $R_\mu=1.5$;
   $(7)$ $R_\mu=5/3$; $(8)$ $R_\mu=2$;
   $(9)$ $R_\mu=3$; $(10)$ $R_\mu=5$;
   $(11)$ $R_\mu=10$.
   The blue line is $F$, the dashed red line is $G$, and the dash-dotted black line is $\eta^2F+G$. The pair $(f_s,g_s)$ is a self-similar profile of \eqref{eq:S2}, 
   whereby  $f_s:=\|f_0\|_1 F$ is the interface between the layers and $f_s+g_s:=\|g_0\|_1(\eta^2F+G)$ is the  upper boundary of the less dense layer.} \label{F:1}
\end{figure}
%%%%%%%%%%%%%%%%%%%%%%%%%%%%%%%%%%%%%%%%

\begin{proof}[Proof of Proposition~\ref{P:-1}] 
According to Lemma~\ref{L:T1}~$(i)$, there is at least one non-empty connected component $\cI=(\alpha,\beta)$ of $\cP_F\cap \cP_G$ and we necessarily have $\alpha\in\R$ and $\beta\in\R$ by Lemma~\ref{L:T1}~$(ii)$. Then $(FG)(\alpha)=(FG)(\beta)=0$ and we classify the (even) solutions of \eqref{ststa} by considering all possible cases determined by these relations.

\medskip

\noindent\textbf{Case (I): $F(\alpha)=F(\beta)=0$.}  By \eqref{F1},  $F$ and $G$ are given by
\begin{equation}
\eta^2 F(x) = a-b-\frac{R_\mu-R}{6R_\mu}x^2 \;\;\text{ and }\;\; G(x) = \frac{(1+R)b-aR}{R}-\frac{1+R-R_\mu}{6R_\mu}x^2 \label{F2001}
\end{equation}
for $x\in\cI$ for some $(a,b)\in \R^2$. Since $F>0$ in $\cI$ and $F(\alpha)=F(\beta)=0$, we realize that necessarily $\alpha=-\beta<0$ and $R_\mu>R$.
Combining the latter with Lemma~\ref{L:T1}~$(iii)$ implies that $\cP_F=(-\beta,\beta)$.

Next either $G(-\beta) G(\beta)=0$ and \eqref{F2001} entails that $G(\beta)=G(-\beta)=0$. Or $G(\beta) G(-\beta)>0$ and we denote the connected component of $\cP_G$ 
containing $(-\beta,\beta)$ by $(\delta,\gamma)$. Clearly, $\delta<-\beta < \beta < \gamma$ and, due to \eqref{F3}, there are $b_1, b_2$ such that 
$$
G(x) = \frac{b_1}{R}-\frac{x^2}{6R_\mu}\ , \quad x\in (\beta,\gamma)\ , \;\;\text{ and }\;\; G(x) = \frac{b_2}{R} - \frac{x^2}{6R_\mu}\ , \quad x\in (\delta,-\beta)\ .
$$
Since $G(\beta)=G(-\beta)$ by \eqref{F2001}, we realize that $b_1=b_2$ and thus that $\delta=-\gamma$. Furthermore the continuity of $G$ at $x=\beta$ and the
property $F(\beta)=0$ give $b_1=b$.  

Finally let $\cJ$ be a connected component of $\cP_G$ lying outside $(-\beta,\beta)$ (resp. $(-\gamma,\gamma)$) if $G(\beta)=0$ (resp. $G(\beta)>0$). 
Since $F$ vanishes on $\cJ$, the function $G$ is given by \eqref{F3} and thus is monotone in $\cJ$, leading us again to a contradiction. 
Therefore, $\cP_G=(-\beta,\beta)$ (resp. $\cP_G=(-\gamma,\gamma)$). 

Summarizing, we have shown that there is $0<\beta \le \gamma$ and $(a,b)\in\R^2$ such that
$$
\eta^2F(x)= a-b-\frac{R_\mu-R}{6R_\mu}x^2\ , \quad|x|\le \beta\ ,
$$
and
$$
G(x)=
\left\{
\begin{array}{ll}
\displaystyle \frac{(1+R)b-aR}{R}-\frac{1+R-R_\mu}{6R_\mu}x^2\ ,& |x|\leq\beta\ ,\\
 & \\
\displaystyle \frac{b}{R}-\frac{x^2}{6R_\mu}\ , & \beta\le |x| \le \gamma \;\; (\text{if }\;\; \beta\ne\gamma).
\end{array}
\right.
$$
In view of $F(\beta)=G(\gamma)=0$ we have the following relations
$$
a-b=\frac{R_\mu-R}{6R_\mu} \beta^2 \qquad\text{and}\qquad \frac{b}{R}=\frac{\gamma^2}{6R_\mu}.
$$
Moreover, since $\|F\|_1=\|G\|_1=1,$ we also find 
$$
(R_\mu-R) \beta^3 = \frac{9\eta^2R_\mu}{2} \qquad \text{and} \qquad \gamma^3 -(R_\mu-R) \beta^3 = \frac{9R_\mu}{2}.
$$
Consequently, $\beta,\gamma,a,b$ are uniquely determined by $(R, R_\mu,\eta)$ and
\begin{align}
&\beta^3 = \frac{9 \eta^2 R_\mu}{2 (R_\mu-R)}\ , \qquad \gamma^3 = \frac{9 R_\mu}{2} (1+\eta^2)\ , \label{c2:V} \\
&a = \frac{\gamma R}{6 R_\mu} + \frac{R_\mu-R}{6 R_\mu}\ , \qquad b = \frac{\gamma R}{6 R_\mu}\ . \nonumber
\end{align}
Imposing that $\gamma\ge\beta$ and that $G(0)>0,$  
we obtain that this case occurs exactly when
\begin{equation}\label{a2}
R_\mu^0(R,\eta) = R+\frac{\eta^2}{1+\eta^2}\le R_\mu<R+\left(\frac{1+\eta^2}{\eta^2}\right)^2 = R_\mu^+(R,\eta),
\end{equation}
with $\beta=\gamma$ if and only if $R_\mu = R_\mu^0(R,\eta)$. Note that the constraints \eqref{a2} are consistent with the condition $R_\mu>R$. Observe also that $G$ 
is convex in $(-\beta,\beta)$ if  $R_\mu>1+R$, $G$ is concave in $(-\beta,\beta)$ if  $R_\mu<1+R$, and $G$ is constant in $(-\beta,\beta)$ if $R_\mu=1+R$. 
This completes the proof of Proposition~\ref{P:-1}~(i)-(ii).

\medskip

\noindent\textbf{Case (II): $F(\beta)=G(\alpha)=0$.} We may additionally assume that $F(\alpha)>0$ since the case where $F$ vanishes at both $\alpha$ and $\beta$
has been handled in \textbf{Case~(I)}. 
Next, assume for contradiction that $G(\beta)=0$. Since $F$ and $G$ are given by
\begin{equation}
\eta^2 F(x) = a-b-\frac{R_\mu-R}{6R_\mu}x^2 \;\;\text{ and }\;\; G(x) = \frac{(1+R)b-aR}{R}-\frac{1+R-R_\mu}{6R_\mu}x^2 \label{F2002}
\end{equation}
in $(\alpha,\beta)$ for some $(a,b)\in \R^2$ according to \eqref{F1}, the property $G(\alpha)=G(\beta)=0$ and \eqref{F2002} imply that $\alpha=-\beta<0$. 
Using again \eqref{F2002}, we realize that it gives $F(\alpha)=F(\beta)=0$ and a contradiction. Therefore $G(\beta)>0$ and we may define
\begin{equation}
\begin{split}
\beta_1 & := \inf\{x<\alpha\,:\, F>0 \;\;\text{ in }\;\; (x,\beta) \} < \alpha\ , \\
\gamma & := \sup\{x>\beta\,:\, G>0 \;\;\text{ in }\;\; (\alpha,x)\}> \beta\ . 
\end{split} \label{pachelbel}
\end{equation}
Since $\partial_x F(\beta-)\le 0$ and $\partial_x G(\alpha+)\ge 0$, we deduce from \eqref{F2002} that
\begin{equation}
(R_\mu - R) \beta \ge 0 \;\;\text{ and }\;\; (1+R-R_\mu) \alpha \le 0\ . \label{ce4}
\end{equation}
In addition, 
\begin{equation}
- \beta \not\in [\alpha,\beta)\ , \quad  -\alpha \not\in (\alpha,\beta]\ , \;\;\text{ and }\;\; \alpha \beta \ge 0\ . \label{ce5}
\end{equation}
Indeed, assume for contradiction that $- \beta \in [\alpha,\beta)$. Then $F(-\beta)=F(\beta)=0$ by \eqref{F2002} and a contradiction. 
A similar argument gives the second claim in \eqref{ce5}. 
Finally, assume for contradiction that $\alpha\beta<0$, so that $\alpha<0<\beta$. It then follows from the first two statements in \eqref{ce5} 
that $-\beta<\alpha$ and $\beta<-\alpha$, and a contradiction. 
As a consequence of \eqref{ce5}, we realize that either $0\le \alpha < \beta$ or $\alpha< \beta \le 0$ and study separately these two cases.

\smallskip

\noindent\textbf{Case~(II-a):} We first consider the case $0 \le \alpha < \beta$. Since $F$ and $G$ are not constant in $(\alpha,\beta)$ we infer from 
\eqref{ce4} that $R_\mu>R$ and either $\alpha>0$ and $R_\mu>R+1$ or $\alpha=0$. 
In the latter, $G(0)=0$ and the positivity of $G$ in $(0,\beta)$ entails that $R_\mu>R+1$ as well. We have thus shown that $R_\mu>R+1$ in that case. 
We then infer from Lemma~\ref{L:T1}~$(iii)$ that $\cP_F=(\beta_1,\beta)$ and $\beta_1<0$. 
The assumed evenness of $F$ (which has not be used up to this point) entails that $\beta_1=-\beta$. 
Assume next for contradiction that there is $x_0\in (0,\alpha)$ such that $G(x_0)>0$, a situation which can only occur if $\alpha>0$. 
Then $x_0\in \cP_F\cap \cP_G$ and it follows from \eqref{F1} and the property $R_\mu>R+1$ that $G$ is increasing on $(x_0,\alpha)$ which contradicts $G(\alpha)=0$.
Therefore, recalling that $G$ is assumed to be even, we conclude that $G\equiv 0$ in $(-\alpha,\alpha)$. 
Finally, since $\cP_F = (-\beta,\beta)$, we deduce from \eqref{F3} that $G$ shall be monotone on any connected component of $\cP_G\setminus(-\gamma,\gamma)$, 
so that necessarily $\cP_G =(-\gamma,-\beta)\cup(\beta,\gamma).$  

Summarizing, there are $0\le \alpha < \beta < \gamma$ such that $\cP_F=(-\beta,\beta)$ and $\cP_G=(-\gamma,-\alpha)\cup (\alpha,\gamma)$ and it
follows from \eqref{F1}-\eqref{F3}, 
the continuity of $F$ and $G$, and the constraints $F(\beta) = G(\alpha) = G(\gamma) = 0$ that there are real numbers $(a,b)$ such that
$$
\eta^2F(x)=
\left\{
\begin{array}{ll}
 \displaystyle\frac{a}{1+R}-\frac{x^2}{6(1+R)}\ , &  |x|\leq\alpha\ ,\\
  & \\
\displaystyle a-b-\frac{R_\mu-R}{6R_\mu}x^2\ , & \alpha\le |x| \le \beta\ ,
\end{array}
\right.
$$
and
$$
G(x)=
\left\{
\begin{array}{ll}
\displaystyle \frac{(1+R)b-aR}{R}-\frac{1+R-R_\mu}{6R_\mu}x^2\ , & \alpha\le |x| \le \beta\ ,\\
 & \\
\displaystyle \frac{b}{R}-\frac{x^2}{6R_\mu}\ , & \beta\le |x|\le \gamma\ .
\end{array}
\right. 
$$
The parameters $a,b, \alpha,\beta,\gamma$ satisfy
\begin{equation}\label{eq:as}
a-b=\frac{R_\mu-R}{6R_\mu}\beta^2,\qquad \frac{(1+R)b-aR}{R}=\frac{1+R-R_\mu}{6R_\mu}\alpha^2,\qquad\frac{b}{R}=\frac{\gamma^2}{6R_\mu},
\end{equation}
as well as
\begin{eqnarray}
 (R_\mu-R)\beta^3-\frac{R(R_\mu-R-1)}{1+R}\alpha^3 &=&\frac{9\eta^2R_\mu}{2}\ ,\label{eq:41}\\
\gamma^3-(R_\mu-R)\beta^3+(R_\mu-R-1)\alpha^3&=&\frac{9R_\mu}{2} \ ,\label{eq:42}
\end{eqnarray}
since $\|F\|_1=\|G\|_1=1$. We are left with solving the algebraic system \eqref{eq:as}-\eqref{eq:42} for the unknowns $(a,b,\alpha,\beta,\gamma)$, keeping 
in mind the constraint $0\le \alpha < \beta < \gamma$. 
It however easily follows from \eqref{eq:as} that $a$ and $b$ can be computed in terms of $(\alpha,\beta,\gamma)$ and that \eqref{eq:as} reduces to
\begin{equation}
\gamma^2-(R_\mu-R)\beta^2+(R_\mu-R-1)\alpha^2=0\ .\label{eq:43}
\end{equation}
Thus, we only have to solve the system of three algebraic equations \eqref{eq:41}-\eqref{eq:43} for $(\alpha,\beta, \gamma)$ and
find out for which values of the parameters $(R,R_\mu,\eta)$ 
satisfying $R_\mu>R+1$ it has a solution enjoying the property $0\le \alpha < \beta < \gamma$. 
According to Lemma~\ref{L:Z1} which is stated and proved in the appendix the system \eqref{eq:41}-\eqref{eq:43} 
has a unique solution $(\alpha,\beta, \gamma)$ satisfying $0\le \alpha < \beta < \gamma$ if and only if $R\ge R_\mu^+(R,\eta)$.
Moreover, $\alpha>0$ if $R_\mu>R_\mu^+(R,\eta)$ and $\alpha=0$ if $R_\mu=R_\mu^+(R,\eta)$.  
We have thus proved Proposition~\ref{P:-1}~(iii).

\smallskip

\noindent\textbf{Case (II-b):} We are left with the case $\alpha < \beta \le 0$ which actually can be deduced from the previous one with the help of Lemma~\ref{L:2.4c}. 
Indeed, define the parameters $(R_{\mu,1},\eta_1)$ and $\lambda$ as in Lemma~\ref{L:2.4c} and set
$$
F_1(x) = \lambda G(-\lambda x)\ , \qquad G_1(x) = \lambda F(-\lambda x)\ , \qquad x\in\R\ .
$$ 
Then $(F_1,G_1)$ is a solution of \eqref{ststa} with parameters $(R,R_{\mu,1},\eta_1)$, $F_1(-\alpha/\lambda)=G_1(-\beta/\lambda)=0$,
and $(-\beta/\lambda,-\alpha/\lambda)$ is a connected component of $\cP_{F_1}\cap \cP_{G_1}$. In addition, recalling the definition \eqref{pachelbel} 
of $\beta_1$ and $\gamma$,
the interval $(-\gamma/\lambda,-\alpha/\lambda)$ is a connected component of $\cP_{F_1}$ while the interval $(-\beta/\lambda,-\beta_1/\lambda)$ is a
connected component of $\cP_{G_1}$. 
We are thus in the situation analysed in \textbf{Case~(II-a)} for $(F_1,G_1)$ and $(F_1,G_1)$ is given by
\begin{align*}
&\eta_1^2 F_1(x)=
\left\{
\begin{array}{ll}
\displaystyle\frac{R_{\mu,1}-R}{6R_{\mu,1}} \frac{\alpha^2}{\lambda^2} + \frac{R(1+R-R_{\mu,1})}{6 R_{\mu,1}(1+R)}\frac{\beta^2}{\lambda^2} - \frac{x^2}{6(1+R)}\ ,& 
|x|\leq \displaystyle -\frac{\beta}{\lambda},\\
 & \\
\displaystyle \frac{R_{\mu,1}-R}{6R_{\mu,1}} \left( \frac{\alpha^2}{\lambda^2} -x^2 \right), & \displaystyle -\frac{\beta}{\lambda} \le |x|\le -\frac{\alpha}{\lambda}\ ,
\end{array}
\right.  \\
& G_1(x)=
\left\{
\begin{array}{ll}
\displaystyle \frac{1+R-R_{\mu,1}}{6R_{\mu,1}} \left( \frac{\beta^2}{\lambda^2} -x^2 \right)\ , & 
\displaystyle -\frac{\beta}{\lambda} \le |x|\leq - \frac{\alpha}{\lambda},\\
& \\
\displaystyle\frac{1}{6R_{\mu,1}} \left( \frac{\beta_1^2}{\lambda^2} -x^2 \right)\ ,& \displaystyle -\frac{\alpha}{\lambda}\le |x|\le - \frac{\beta_1}{\lambda}\ ,\\
\end{array}
\right. 
\end{align*} 
where $\gamma = -\alpha$ and $0\le -\beta/\lambda<-\alpha/\lambda<-\beta_1/\lambda$ is the solution of \eqref{eq:41}-\eqref{eq:43} with $(R,R_{\mu,1},\eta_1)$ 
instead of $(R,R_\mu,\eta)$ which is known to exist if and only if
$$
R_{\mu,1}\ge R_\mu^+(R,\eta_1) = R + \left( \frac{1+\eta_1^2}{\eta_1^2} \right)^2\ ,
$$ 
owing to the analysis performed in \textbf{Case~(II-a)}. 
Equivalently $R_\mu\le R_\mu^-(R,\eta)$ and $\beta_1<\alpha<\beta \le 0$ is the unique solution of
\begin{eqnarray*}
(1+R-R_\mu) (-\alpha)^3 - (R-R_\mu) (-\beta)^3 & = & \frac{9 R_\mu}{2}\ , \\
R_\mu (-\beta_1)^3 - R ( 1+R-R_\mu) (-\alpha)^3 + (1+R)(R-R_\mu) (-\beta)^3 & = & \frac{9}{2} R_\mu (1+R) \eta^2\ , \\
R_\mu \beta_1^2 - R (1+R-R_\mu) \alpha^2 + (1+R) (R-R_\mu) \beta^2 & = & 0\ .
\end{eqnarray*}
Furthermore, $(F,G)$ are given by
\begin{align*}
&\eta^2 F(x)=
\left\{
\begin{array}{ll}
\displaystyle \frac{R_\mu-R}{6R_\mu} \left( \beta^2 -x^2 \right)\ , & \displaystyle -\beta \le |x|\leq - \alpha,\\
& \\
\displaystyle\frac{1}{6(1+R)} \left( \beta_1^2 -x^2 \right)\ ,& \displaystyle -\alpha\le |x|\le -\beta_1\ ,\\
\end{array}
\right.  
\end{align*} 
\begin{align*}
& G(x)=
\left\{
\begin{array}{ll}
\displaystyle\frac{1+R-R_\mu}{6R_\mu} \alpha^2 + \frac{R_\mu-R}{6 R_\mu} \beta^2 - \frac{x^2}{6R_\mu}\ ,& |x|\leq  -\beta\ ,\\
 & \\
\displaystyle \frac{1+R-R_\mu}{6R_\mu} \left( \alpha^2 -x^2 \right), & -\beta \le |x|\le -\alpha\ .
\end{array}
\right. 
\end{align*} 
Changing the notation $(-\beta,-\alpha,-\beta_1)$ to $(\alpha,\beta,\gamma)$ for consistency, the above analysis shows that $(\alpha,\beta,\gamma)$ is
the unique solution to
\begin{eqnarray}
(1+R-R_\mu) \beta^3 - (R-R_\mu) \alpha^3 & = & \frac{9 R_\mu}{2}\ , \label{eq:51}\\
R_\mu \gamma^3 - R ( 1+R-R_\mu) \beta^3 + (1+R)(R-R_\mu) \alpha^3 & = & \frac{9}{2} R_\mu (1+R) \eta^2\ , \label{eq:52}\\
R_\mu \gamma^2 - R (1+R-R_\mu) \beta^2 + (1+R) (R-R_\mu) \alpha^2 & = & 0\ , \label{eq:53}
\end{eqnarray}
satisfying $0\le \alpha<\beta<\gamma$ which exists if and only if $R_\mu\le R_\mu^-(R,\eta)$. We have thus completed the proof of Proposition~\ref{P:-1}~(iv).

\medskip

\noindent\textbf{Case~(III): $F(\alpha)=G(\beta)=0$.} This case actually reduces to the previous ones thanks to Lemma~\ref{L:2.4c}. 
Indeed, define the parameters $(R_{\mu,1},\eta_1)$ and $\lambda$ as in Lemma~\ref{L:2.4c} and set
$$
F_1(x) = \lambda G(\lambda x)\ , \qquad G_1(x) = \lambda F(\lambda x)\ , \qquad x\in\R\ .
$$ 
Then $(F_1,G_1)$ is a solution to \eqref{ststa} with parameters $(R,R_{\mu,1},\eta_1)$ with $F_1(\beta/\lambda)=G_1(\alpha/\lambda)=0$ and
$(\alpha/\lambda,\beta/\lambda)$ is
a connected component of $\cP_{F_1}\cap \cP_{G_1}$. 
We are thus in the situation already analysed in \textbf{Case~(II)} for $(F_1,G_1)$ and we do not obtain other solutions.

\medskip

\noindent\textbf{Case~(IV): $G(\alpha)=G(\beta)=0$.} Once more, using Lemma~\ref{L:2.4c} and keeping the same notation as in \textbf{Case~(III)} allow us 
to deduce this case from \textbf{Case~(I)}. 
Indeed, arguing as in \textbf{Case~(III)} above we realize that we are in the situation analyzed in \textbf{Case~(I)} for $(F_1,G_1)$. Then $\alpha=-\beta$,
$$
\eta_1^2 F_1(x)= \frac{R_{\mu,1} - R}{6 R_{\mu,1}} \left( \frac{\beta^2}{\lambda^2} - x^2 \right)\ , \quad |x|\le \frac{\beta}{\lambda}\ ,
$$
and
$$
G_1(x)=
\left\{
\begin{array}{ll}
\displaystyle \frac{1}{6R_{\mu,1}} \frac{\gamma^2}{\lambda^2} + \frac{R-R_{\mu,1}}{6 R_{\mu,1}} \frac{\beta^2}{\lambda^2} -\frac{1+R-R_{\mu,1}}{6R_{\mu,1}} x^2\ ,&
\displaystyle |x|\leq\frac{\beta}{\lambda}\ ,\\
 & \\
\displaystyle \frac{1}{6R_{\mu,1}} \left( \frac{\gamma^2}{\lambda^2} - x^2 \right), & \displaystyle \frac{\beta}{\lambda}\le |x| \le \frac{\gamma}{\lambda} \ ,
\end{array}
\right.
$$
where $(\beta/\lambda,\gamma/\lambda)$ are given by
$$
\frac{\beta^3}{\lambda^3} = \frac{9}{2} \frac{\eta_1^2R_{\mu,1}}{R_{\mu,1}-R} \ , \quad \frac{\gamma^3}{\lambda^3} = \frac{9 R_{\mu,1}}{2} (1+\eta_1^2)\ ,
$$
and satisfy $\beta\le \gamma$, the latter being true if and only if $R_{\mu,1}\in [R_\mu^0(R,\eta_1),R_\mu^+(R,\eta_1))$. 
This condition also reads $R_\mu\in (R_\mu^-(R,\eta),R_\mu^0(R,\eta)]$ while $\beta$ and $\gamma$ are explicitly given by
\begin{equation}
\beta^3 = \frac{9}{2} \frac{R_\mu}{1+R-R_\mu}\ , \quad \gamma^3 = \frac{9}{2} ((1+R)\eta^2 + R)\ . \label{c8:V}
\end{equation}
Observing that $\beta=\gamma$ if $R_\mu=R_\mu^0(R,\eta)$ which corresponds to the solution to \eqref{ststa} already described in Proposition~\ref{P:-1}~(i), 
we have shown Proposition~\ref{P:-1}~(v) and thereby completed the proof.
\end{proof}

%%%%%%%%%%%%%%%%%%%%%%%%%%%%%%%%%%%%%%%%
\begin{rem}\label{R:2.7}
It is worth emphasizing here that the assumption of evenness of the solution $(F,G)$ to \eqref{ststa} is used only in the analysis
of \textbf{Case~(II)} and \textbf{Case~(III)} in the proof of Proposition~\ref{P:-1}.
Therefore, on the one hand, only even solutions of \eqref{ststa} exist when $R_\mu\in (R_\mu^-(R,\eta),R_\mu^+(R,\eta))$.
On the other hand, there may exist other, non-symmetric, solutions of \eqref{ststa} when $R_\mu\not\in (R_\mu^-(R,\eta),R_\mu^+(R,\eta))$. 
In the following we shall prove that non-symmetric solutions of \eqref{ststa} exist if and only if $R_\mu\not\in [R_\mu^-(R,\eta),R_\mu^+(R,\eta)]$.
\end{rem}
%%%%%%%%%%%%%%%%%%%%%%%%%%%%%%%%%%%%%%%%

%%%%%%%%%%%%%%%%%%%%%%%%%%%%%%%%%%%%%%%%
%%%%%%%%%%%%%%%%%%%%%%%%%%%%%%%%%%%%%%%%
\subsection{Non-symmetric self-similar profiles with connected supports}\label{S:23}
%%%%%%%%%%%%%%%%%%%%%%%%%%%%%%%%%%%%%%%%
%%%%%%%%%%%%%%%%%%%%%%%%%%%%%%%%%%%%%%%%

Up to now, we have shown that for each choice of the parameters $(R,R_\mu,\eta)$ there exists exactly one even solution of \eqref{ststa}. 
We show next that for certain values of the parameters there exist other solutions $(F,G)$ of \eqref{ststa} which are not symmetric and have 
the property that both $F$ and $G$ have connected supports. Observe that non-symmetric solutions of \eqref{ststa} appear always pairwise 
according to Lemma~\ref{L:2.4c}~$(i)$.

%%%%%%%%%%%%%%%%%%%%%%%%%%%%%%%%%%%%%%%%
\begin{prop}\label{P:0}  Let $(R,R_\mu, \eta)$ be positive parameters. There are $R_\mu^M(R,\eta)>R_\mu^+(R,\eta)$ and $R_\mu^m(R,\eta)<R_\mu^-(R,\eta)$ such that:
\begin{itemize}
\item [$(i)$] if $R_\mu\ge R_\mu^M(R,\eta)$, then the pair $(F,G)$ with $\cP_F=(\beta_1,\beta),$ $\cP_G=(\alpha,\gamma),$ and
$$
\eta^2 F(x)=
\left\{
\begin{array}{ll}
 \displaystyle\frac{\beta_1^2-x^2}{6 (1+R)}\ , &  x\in [\beta_1,\alpha]\ ,\\
  & \\
\displaystyle \frac{R_\mu-R}{6 R_\mu}(\beta^2-x^2)\ , & x\in [\alpha,\beta]\ ,
\end{array}
\right.
$$
$$
G(x)=
\left\{
\begin{array}{ll}
\displaystyle \frac{1+R-R_\mu}{6R_\mu}(\alpha^2-x^2)\ , & x\in [\alpha,\beta]\ ,\\
 & \\
\displaystyle \frac{\gamma^2-x^2}{6R_\mu}\ , & x\in [\beta,\gamma]\ ,
\end{array}
\right. 
$$
is a non-symmetric solution  of \eqref{ststa} where $(\beta_1,\alpha,\beta,\gamma)$ is the unique solution of the system of algebraic equations \eqref{d1:V1}-\eqref{d1:V4} 
satisfying $\beta_1<0\le\alpha<\beta<\gamma$. Additionally, its reflection $x\mapsto (F(-x),G(-x))$ is also a solution of \eqref{ststa}.

\item [$(ii)$] if $R_\mu\le R_\mu^m(R,\eta)$, then the pair $(F,G)$ with $\cP_F=(\alpha,\gamma)$, $\cP_G=(\beta_1,\beta)$, and
$$
\eta^2 F(x)=
\left\{
\begin{array}{ll}
\displaystyle \frac{R_\mu-R}{6 R_\mu}(\alpha^2-x^2)\ , & x\in [\alpha,\beta]\ ,\\
 & \\
 \displaystyle\frac{\gamma^2-x^2}{6 (1+R)}\ , &  x\in [\beta,\gamma]\ ,
\end{array}
\right.
$$
$$
G(x)=
\left\{
\begin{array}{ll}
\displaystyle \frac{\alpha_1^2-x^2}{6R_\mu}\ , & x\in [\beta_1,\alpha]\ ,\\
 & \\
\displaystyle \frac{1+R-R_\mu}{6R_\mu}(\beta^2-x^2)\ , & x\in [\alpha,\beta]\ ,
\end{array}
\right. 
$$
is a non-symmetric solution  of \eqref{ststa} where $(\beta_1,\alpha,\beta,\gamma)$ is the unique solution of the system of algebraic equations \eqref{d2:V1}-\eqref{d2:V4} 
satisfying $\beta_1<0\leq\alpha<\beta<\gamma$. Additionally, its reflection $x \mapsto (F(-x),G(-x))$ is also a solution of \eqref{ststa}.

\item [$(iii)$] if $R_\mu\in (R_\mu^m(R,\eta),R_\mu^M(R,\eta))$, there is no non-symmetric solution $(F,G)$ of \eqref{ststa} which have the property
that the supports of $F$ and $G$ are connected.
\end{itemize}
\end{prop}
%%%%%%%%%%%%%%%%%%%%%%%%%%%%%%%%%%%%%%%%

The threshold value $R_\mu^M(R,\eta)$ is actually the unique solution in $(R+1,\infty)$ of the equation
\begin{equation}
\sqrt{R_\mu^M-R} \left( \eta^2 - \sqrt{\frac{1+R}{R_\mu^M}} \right) = 1 + \eta^2\ , \label{RmuMc}
\end{equation}
while $R_\mu^m(R,\eta)$ is the unique solution in $(0,R+1)$ of
\begin{equation}
\sqrt{1+R-R_\mu^m} \left( \sqrt{\frac{R}{R_\mu^m}} - \eta^2 \frac{1+R}{R} \right) = 1 + \eta^2 \frac{1+R}{R} \ . \label{Rmums}
\end{equation}

%%%%%%%%%%%%%%%%%%%%%%%%%%%%%%%%%%%%%%%%
\begin{figure}
\hspace{32cm}\includegraphics[width=3.5in]{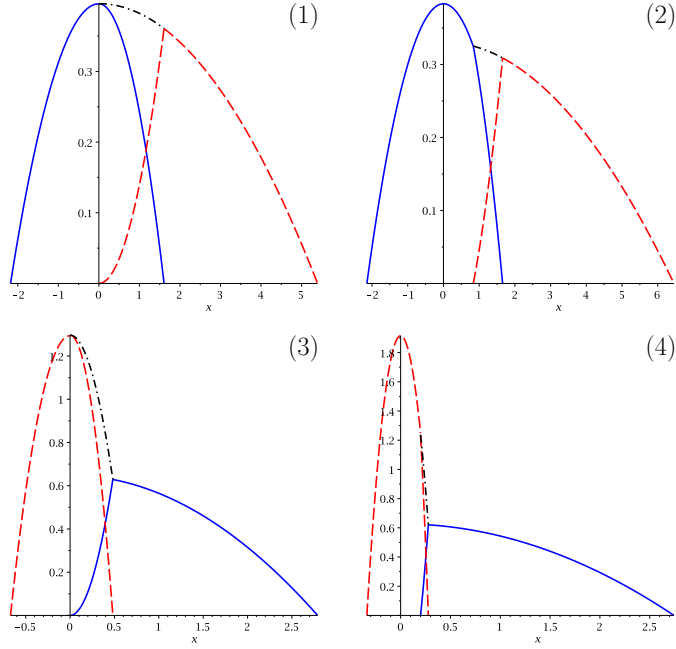}
\caption[Non-symmetric self-similar profiles with connected supports]{\small Non-symmetric self-similar profiles of \eqref{Pb} with connected supports.
The parameters are $\eta=R=1$ and: $(1)$  $R_\mu=R_\mu^M(1,1)\approx 12.258$; $(2)$   $R_\mu=21$; $(3)$  $R_\mu=R_\mu^m(1,1)\approx 0.058$;  $(4)$  $R_\mu=0.01$.} \label{F:2}
\end{figure}
%%%%%%%%%%%%%%%%%%%%%%%%%%%%%%%%%%%%%%%%

\begin{proof} 
As already pointed out in Remark~\ref{R:2.7}, one of the outcome of the proof of Proposition~\ref{P:-1} is that solutions to \eqref{ststa} are 
necessarily even in \textbf{Cases~(I) \&~(IV)}. To prove Proposition~\ref{P:0} we are left to consider \textbf{Cases~(II) \&~(III)} without
assuming that the solutions sought for are even but assuming that their supports are connected. 

\medskip

\noindent\textbf{Case~(II-a):} Recall that we are in the situation where there are $0\le \alpha < \beta$ such that $(\alpha,\beta)$ is a
connected component of $\cP_F\cap \cP_G$ with $F(\beta)=G(\alpha)=0$, $F(\alpha)>0$, and $G(\beta)>0$.
Also, we assume that either $F$ or $G$ is not even. 
As in the proof of  Proposition~\ref{P:-1}, we define
\begin{align*}
\beta_1 & := \inf\{x<\alpha\,:\, F>0 \;\;\text{ in }\;\; (x,\beta) \} < \alpha\ , \\
\gamma & := \sup\{x>\beta\,:\, G>0 \;\;\text{ in }\;\; (\alpha,x)\}> \beta\ ,
\end{align*}
and recall that $R_\mu>R+1$. Then Lemma~\ref{L:T1}~$(iii)$ guarantees that $\beta_1<0$ and $\cP_F=(\beta_1,\beta)$.
Now assume for contradiction that there is $x_0\in (\beta_1,\alpha)$ such that $G(x_0)>0$. 
Since $\overline{\cP_G}$ is connected this implies that $G(x)>0$ for $x\in [x_0,\alpha)$ and thus $\partial_x G(\alpha-) \le 0$ since $G(\alpha)=0$.
By \eqref{F1} $\partial_x G(\alpha-) = -(1+R-R_\mu) \alpha/3 R_\mu$ and combining the previous properties with the inequality $R_\mu>R+1$ 
implies that necessarily $\alpha=0$ and $G$ is decreasing on the connected component of $\cP_F\cap \cP_G$ to which $x_0$ belongs. 
Consequently, $(\beta_1,0)$ is the connected component of $\cP_F\cap \cP_G$ containing $x_0$ and we infer from \eqref{F1} that there 
are real numbers $(a,b,a_1,b_1)$ such that
$$
\eta^2 F(x) = \left\{
\begin{array}{ll}
\displaystyle a_1-b_1-\frac{R_\mu-R}{6R_\mu}x^2\ , & x\in[\beta_1,0]\ , \\
\displaystyle a-b-\frac{R_\mu-R}{6R_\mu}x^2\ , &  x\in [0,\beta] \ ,
\end{array}
\right.
$$
and
$$
G(x)= \left\{
\begin{array}{ll}
\displaystyle \frac{(1+R)b_1-a_1R}{R} -\frac{1+R-R_\mu}{6R_\mu}x^2\ , & x\in[\beta_1,0]\ , \\
\displaystyle \frac{(1+R)b-aR}{R} -\frac{1+R-R_\mu}{6R_\mu}x^2\ , & x\in [0,\beta]\ .
\end{array}
\right.
$$
Since $G(0)=0$ we realize that $(1+R) b_1= R a_1$ and $(1+R)b = Ra$ while the continuity of $F$ requires $a_1-b_1=a-b$. 
Consequently, $a=a_1$, $b=b_1$ and 
$$
\eta^2 F(x) = a-b-\frac{R_\mu-R}{6R_\mu}x^2\ , \quad G(x) = -\frac{1+R-R_\mu}{6R_\mu}x^2\ , \qquad x\in(\beta_1,\beta)\ ,
$$
from which we deduce that $\beta_1=-\beta$. In particular, $\cP_F=(-\beta,\beta)$, $F$ is even, and $G(-\beta) = G(\beta) > 0$. 
Denoting the connected component of $\cP_G$ containing $-\beta$ by $\cJ_1$, 
it follows from \eqref{F3} that there are $b_2\in\R$ and $b_3\in\R$ such that
$$
G(x) = \frac{b_2}{R} - \frac{x^2}{6 R_\mu}\ , \quad x\in (\beta,\gamma)\ , \;\;\text{ and }\;\;
G(x) = \frac{b_3}{R} - \frac{x^2}{6 R_\mu}\ , \quad x\in \cJ_1\cap (-\infty,-\beta)\ .
$$
As $G(-\beta)=G(\beta)$, we realize that $b_2=b_3$ so that $\cJ_1=(-\gamma,0)$ and $G$ is even on $(-\gamma,\gamma)$. 
Furthermore, since $\cP_F=(-\beta,\beta)\subset (-\gamma,\gamma)$, 
Lemma~\ref{L:T1}~$(v)$ entails that $\cP_G=(-\gamma,\gamma)$. 
We have thus shown that $F$ and $G$ are even which contradicts our starting assumption. Consequently,
\begin{equation}
G(x) = 0\ , \quad x\in (\beta_1,\alpha)\ , \label{as1}
\end{equation}
whence $\cP_G\subset (\alpha,\infty)$ since $\overline{\cP_G}$ is connected and $\beta_1<0\le \alpha$. 
Finally, since $\cP_F=(\beta_1,\beta)$,
Lemma~\ref{L:T1}~$(v)$ excludes the existence of another connected component of
$\cP_G$ included in $(\gamma,\infty)$ and we have thus established that $\cP_F=(\beta_1,\beta)$ and $\cP_G=(\alpha,\gamma)$. 
Then, according to \eqref{F1}-\eqref{F3}, there are $(a,b,a_1,b_1)\in \R^4$ such that 
$$
\eta^2F(x)=
\left\{
\begin{array}{ll}
 \displaystyle\frac{a_1}{1+R}-\frac{x^2}{6(1+R)}\ , &  x\in [\beta_1,\alpha]\ ,\\
  & \\
\displaystyle a-b-\frac{R_\mu-R}{6R_\mu}x^2\ , & x\in [\alpha,\beta]\ ,
\end{array}
\right.
$$
and
$$
G(x)=
\left\{
\begin{array}{ll}
\displaystyle \frac{(1+R)b-aR}{R}-\frac{1+R-R_\mu}{6R_\mu}x^2\ , & x\in [\alpha,\beta]\ ,\\
 & \\
\displaystyle \frac{b_1}{R}-\frac{x^2}{6R_\mu}\ , & x\in [\beta,\gamma]\ .
\end{array}
\right. 
$$
Since $F$ and $G$ are both continuous and   $F(\beta_1)=F(\beta)=G(\alpha)=G(\gamma)=0$ we find that   $a_1=a$, $b_1=b$, and moreover
\begin{align} \label{d:1} 
a-b=\frac{R_\mu-R}{6R_\mu}\beta^2\ , \,\ \frac{(1+R)b-aR}{R} = \frac{1+R-R_\mu}{6R_\mu}\alpha^2\ , \,\ a = \frac{\beta_1^2}{6}\ , \,\ \frac{b}{R} = \frac{\gamma^2}{6R_\mu}\ .
\end{align}
Requiring that both $F$ and $G$ have unitary $L_1$-norm, we arrive at the following relations 
\begin{eqnarray}
\label{d1:V1}
R_\mu\beta_1^3-(1+R)(R_\mu-R)\beta^3 +R(R_\mu-R-1)\alpha^3&=&-9\eta^2R_\mu(1+R)\ ,\\
\label{d1:V2}
\gamma^3-(R_\mu-R)\beta^3 +(R_\mu-R-1)\alpha^3&=&9R_\mu\ ,
\end{eqnarray}
while we deduce from \eqref{d:1} that
\begin{eqnarray}
\label{d1:V3}
\gamma^2-(R_\mu-R)\beta^2 +(R_\mu-R-1)\alpha^2&=&0\ ,\\
\label{d1:V4}
R_\mu\beta_1^2-(1+R)(R_\mu-R)\beta^2 +R(R_\mu-R-1)\alpha^2&=&0\ .
\end{eqnarray}
We are thus looking for solutions $(\beta_1,\alpha,\beta,\gamma)$ of \eqref{d1:V1}-\eqref{d1:V4} satisfying $\beta_1<0\leq\alpha<\beta<\gamma$, which is a 
rather involved problem. Nevertheless, according to Lemma~\ref{P:NSC}, there is $R_\mu^M(R,\eta)>R_\mu^+(R,\eta)$ such that \eqref{d1:V1}-\eqref{d1:V4} 
has a unique solution $(\beta_1,\alpha,\beta,\gamma)$ satisfying $\beta_1<0\leq\alpha<\beta<\gamma$ if and only if $R_\mu\ge R_\mu^M(R,\eta)$. 
In that case, the constants $a,b$ are given by \eqref{d:1} and we obtain the solution $(F,G)$ to \eqref{ststa} given in Proposition~\ref{P:0}~(i). 
We then use Lemma~\ref{L:2.4c}~$(i)$ to conclude that $x\mapsto (F(-x),G(-x))$ is also a solution to \eqref{ststa} and complete the proof of Proposition~\ref{P:0}~(i).

\medskip

\noindent\textbf{Case~(II-b):} As in the proof of Proposition~\ref{P:-1}, we define the parameters $(R_{\mu,1},\eta_1)$ and $\lambda$ as in Lemma~\ref{L:2.4c} and set
$$
F_1(x) = \lambda G(-\lambda x)\ , \qquad G_1(x) = \lambda F(-\lambda x)\ , \qquad x\in\R\ .
$$ 
Then $(F_1,G_1)$ is a solution to \eqref{ststa} with parameters $(R,R_{\mu,1},\eta_1)$ satisfying $F_1(-\alpha/\lambda)=G_1(-\beta/\lambda)=0$ 
with $-\beta/\lambda<-\alpha/\lambda$ and $(-\beta/\lambda,-\alpha/\lambda)$ is a connected component of $\cP_{F_1}\cap \cP_{G_1}$. 
Moreover, both $F_1$ and $G_1$ have connected supports. We are therefore back to the situation analysed in \textbf{Case~(II-a)} and,
according to the analysis performed in that case, there are $\beta_1<\alpha$ and $\gamma>0$ such that 
$\cP_{F_1}=(-\gamma/\lambda,-\alpha/\lambda)$, $\cP_{G_1} = (-\beta/\lambda,-\beta_1/\lambda)$, 
$$
\eta_1^2 F_1(x)=
\left\{
\begin{array}{ll}
\displaystyle \frac{1}{6(1+R)} \left( \frac{\gamma^2}{\lambda^2} - x^2 \right), & \displaystyle -\frac{\gamma}{\lambda}
\le x\le - \frac{\beta}{\lambda}\ , \\
 & \\
\displaystyle \frac{R_{\mu,1}-R}{6 R_{\mu,1}} \left( \frac{\alpha^2}{\lambda^2} - x^2 \right) , & 
\displaystyle -\frac{\beta}{\lambda} \le x \le - \frac{\alpha}{\lambda}\ , 
\end{array}
\right.
$$
and
$$
G_1(x)=
\left\{
\begin{array}{ll}
\displaystyle \frac{1+R-R_{\mu,1}}{6R_{\mu,1}} \left( \frac{\beta^2}{\lambda^2} - x^2 \right), 
&  \displaystyle  -\frac{\beta}{\lambda} \le x \le - \frac{\alpha}{\lambda}\ ,\\
& \\
\displaystyle \frac{1}{6R_{\mu,1}} \left( \frac{\beta_1^2}{\lambda^2} - x^2 \right), & \displaystyle - \frac{\alpha}{\lambda} \le x \le -\frac{\beta_1}{\lambda}\ ,
\end{array}
\right. 
$$
if and only if $R_{\mu,1}\ge R_\mu^M(R,\eta_1)$, the parameters $-\gamma/\lambda<0 \le -\beta/\lambda<-\alpha/\lambda < -\beta_1/\lambda$ being 
the unique solution to \eqref{d1:V1}-\eqref{d1:V4} given by Lemma~\ref{P:NSC} with $(R,R_{\mu,1},\eta_1)$ instead of $(R,R_\mu,\eta)$. Setting
\begin{equation}
R_\mu^m(R,\eta) := \frac{(1+R) R}{R_\mu^M(R,\eta_1)}\ , \label{mozart}
\end{equation}
we have thus shown that there is a solution $(F,G)$ to \eqref{ststa} with connected supports given by
$$
\eta^2 F(x)=
\left\{
\begin{array}{ll}
\displaystyle \frac{1}{6(1+R)} \left( \beta_1^2 - x^2 \right), & \beta_1 \le x \le \alpha\ , \\
 & \\
\displaystyle \frac{R_\mu-R}{6 R_\mu} \left( \beta^2 - x^2 \right) , &  \alpha \le x \le \beta\ , 
\end{array}
\right.
$$
and
$$
G(x)=
\left\{
\begin{array}{ll}
\displaystyle \frac{1+R-R_\mu}{6R_\mu} \left( \alpha^2 - x^2 \right), &  \alpha \le x \le \beta\ ,\\
& \\
\displaystyle \frac{1}{6R_\mu} \left( \gamma^2 - x^2 \right), & \beta \le x \le  \gamma\ ,
\end{array}
\right. 
$$
if and only if $R_\mu\le R_\mu^m(R,\eta)$, the parameters $(\beta_1,\alpha,\beta,\gamma)$ satisfying $\beta_1<\alpha<\beta\le 0 <\gamma$ and solving
\begin{eqnarray*}
R_\mu (-\beta_1)^3 + (1+R)(R-R_\mu) (-\beta)^3 - R(1+R-R_\mu)(-\alpha)^3 & = & 9\eta^2 R_\mu (1+R)\ ,\\
(-\gamma)^3 + (R-R_\mu) (-\beta)^3  - (1+R-R_\mu)(-\alpha)^3 & = & - 9R_\mu\ ,\\
R_\mu \beta_1^2 + (1+R)(R-R_\mu) \beta^2 - R(1+R-R_\mu)\alpha^2 & = & 0\ ,\\
\gamma^2 + (R-R_\mu) \beta^2 - (1+R-R_\mu) \alpha^2 & = & 0\ .
\end{eqnarray*}
Changing the notation $(-\gamma,-\beta,-\alpha,-\beta_1)$ to $(\beta_1,\alpha,\beta,\gamma)$ for consistency, the above algebraic system reads
\begin{eqnarray}
R_\mu \gamma^3 + (1+R)(R-R_\mu) \alpha^3 -R(1+R-R_\mu)\beta^3 & = & 9\eta^2R_\mu(1+R)\ , \label{d2:V1} \\
-\beta_1^3 - (R-R_\mu) \alpha^3 + (1+R-R_\mu)\beta^3 & = & 9R_\mu\ ,\label{d2:V2} \\
R_\mu \gamma^2 + (1+R)(R-R_\mu) \alpha^2 - R(1+R-R_\mu)\beta^2 & = & 0\ ,  \label{d2:V3} \\
\beta_1^2 + (R-R_\mu)\alpha^2 - (1+R-R_\mu)\beta^2 & = & 0\ , \label{d2:V4}
\end{eqnarray}
while $(F,G)$ is given by Proposition~\ref{P:0}~(ii). That its reflection $x\mapsto (F(-x),G(-x))$ also solves \eqref{ststa} is a consequence of Lemma~\ref{L:2.4c}~(i). 
This completes the proof of Proposition~\ref{P:0}~(ii). 

Finally, since \textbf{Case~(III)} ($F(\alpha)=G(\beta)=0$) reduces to \textbf{Case~(II)} as already observed in the proof of Proposition~\ref{P:-1}, we
have identified all possible non-symmetric solutions with connected supports, showing in particular that they exist if and only 
if $R_\mu\not\in (R_\mu^m(R,\eta),R_\mu^M(R,\eta))$, and thereby completed the proof of Proposition~\ref{P:0}.
\end{proof}

%%%%%%%%%%%%%%%%%%%%%%%%%%%%%%%%%%%%%%%%
%%%%%%%%%%%%%%%%%%%%%%%%%%%%%%%%%%%%%%%%
\subsection{Non-symmetric self-similar profiles with disconnected supports}\label{S:24}
%%%%%%%%%%%%%%%%%%%%%%%%%%%%%%%%%%%%%%%%
%%%%%%%%%%%%%%%%%%%%%%%%%%%%%%%%%%%%%%%%

Since we have explicitly used the assumption of connected supports of the solution of \eqref{ststa} we were looking for in Proposition~\ref{P:0},  
there may exist solutions $(F,G)$ of \eqref{ststa}  which have the property that at least one of the functions $F$ and $G$ has a disconnected support.  
The following proposition gives a classification of such solutions, showing in particular that only one of the functions $F$ and $G$ may have a disconnected support.

%%%%%%%%%%%%%%%%%%%%%%%%%%%%%%%%%%%%%%%%
\begin{prop}\label{P:1}  Let $R$, $R_\mu$, and $\eta$ be given positive parameters.
\begin{itemize}
 \item [$(i)$] If $R_\mu > R+1$ and the system \eqref{R1:1}-\eqref{R1:5} has a solution $(\gamma_1,\beta_1,\alpha_1,\alpha,\beta,\gamma)\in\mathbb{R}^6$ satisfying 
 \begin{equation}
 \gamma_1<\beta_1<\alpha_1\leq 0\leq\alpha<\beta<\gamma \;\;\text{ and }\;\; \alpha_1\neq-\alpha\ , \label{haydn} 
 \end{equation} 
 then the pair $(F,G)$ given by 
\begin{align*}
\eta^2 F(x)&=
\left\{
\begin{array}{ll}
\displaystyle \frac{R_\mu-R}{6 R_\mu}(\beta_1^2-x^2)\ , & x\in [\beta_1,\alpha_1]\ ,\\
 & \\
 \displaystyle\frac{R\gamma_1^2+(R_\mu-R)\beta_1^2}{6 (1+R)R_\mu}-\frac{x^2}{6(1+R)}\ , &  x\in [\alpha_1,\alpha]\ ,\\ 
  & \\
\displaystyle \frac{R_\mu-R}{6R_\mu}(\beta^2 -x^2)\ , & x\in [\alpha,\beta]\ ,
\end{array}
\right.
\end{align*}
\begin{align*}
G(x)&=
\left\{
\begin{array}{ll}
\displaystyle \frac{\gamma_1^2-x^2}{6R_\mu}\ , & x\in [\gamma_1,\beta_1]\ ,\\
& \\
\displaystyle\frac{1+R-R_\mu}{6R_\mu}(\alpha_1^2-x^2)\ , & x\in [\beta_1,\alpha_1]\ ,\\
& \\
\displaystyle \frac{1+R-R_\mu}{6R_\mu}(\alpha^2-x^2)\ , & x\in [\alpha,\beta]\ ,\\
& \\
\displaystyle\frac{\gamma^2-x^2}{6R_\mu}\ , & x\in [\beta,\gamma]\ ,
\end{array}
\right. 
\end{align*}
and $\cP_F=(\beta_1,\beta),$ $\cP_G=(\gamma_1,\alpha_1)\cup (\alpha,\gamma)$, is a non-symmetric solution of \eqref{ststa}.
Additionally, its reflection $x\mapsto (F(-x),G(-x))$ is also a solution of \eqref{ststa}.

\item [$(ii)$] If $R_\mu<R$ and the system \eqref{R2:1}-\eqref{R2:5} has a solution $(\gamma_1,\beta_1,\alpha_1,\alpha,\beta,\gamma)$ satisfying \eqref{haydn}, 
then the pair $(F,G)$ given by  
\begin{align*}
\eta^2 F(x)&=
\left\{
\begin{array}{ll}
\displaystyle\frac{\gamma_1^2-x^2}{6(1+R)}\ , &  x\in [\gamma_1,\beta_1]\ , \\
& \\
\displaystyle \frac{R_\mu-R}{6 R_\mu}(\alpha_1^2-x^2)\ , & x\in [\beta_1,\alpha_1]\ ,\\
& \\
\displaystyle \frac{R_\mu-R}{6 R_\mu}(\alpha^2-x^2)\ , & x\in [\alpha,\beta]\ ,\\
& \\
 \displaystyle\frac{\gamma^2-x^2}{6 (1+R)}\ , &  x\in [\beta,\gamma]\ , 
\end{array}
\right.
\end{align*}
\begin{align*}
G(x)&=
\left\{
\begin{array}{ll}
\displaystyle \frac{1+R-R_\mu}{6R_\mu}(\beta_1^2-x^2)\ , & x\in [\beta_1,\alpha_1]\ ,\\
& \\
\displaystyle \frac{(R_\mu-R)\alpha_1^2+(1+R-R_\mu)\beta_1^2}{6R_\mu}-\frac{x^2}{6R_\mu}\ , & x\in [\alpha_1,\alpha]\ ,\\
& \\
\displaystyle\frac{1+R-R_\mu}{6R_\mu}(\beta^2-x^2)\ , & x\in [ \alpha,\beta]\ ,
\end{array}
\right. 
\end{align*} 
and $\cP_F=(\gamma_1,\alpha_1)\cup(\alpha,\gamma),$ $\cP_G=(\beta_1,\beta)$,  is a non-symmetric solution  of \eqref{ststa}. 
Additionally, its reflection $x\mapsto (F(-x),G(-x))$ is also a solution of \eqref{ststa}.

\item[$(iii)$] Moreover, there exist no other non-symmetric solutions of \eqref{ststa} which have the property that the support of either $F$ or $G$ is disconnected.
\end{itemize}
\end{prop}
%%%%%%%%%%%%%%%%%%%%%%%%%%%%%%%%%%%%%%%%

%%%%%%%%%%%%%%%%%%%%%%%%%%%%%%%%%%%%%%%%
\begin{figure}
\hspace{32cm}\includegraphics[width=3.5in]{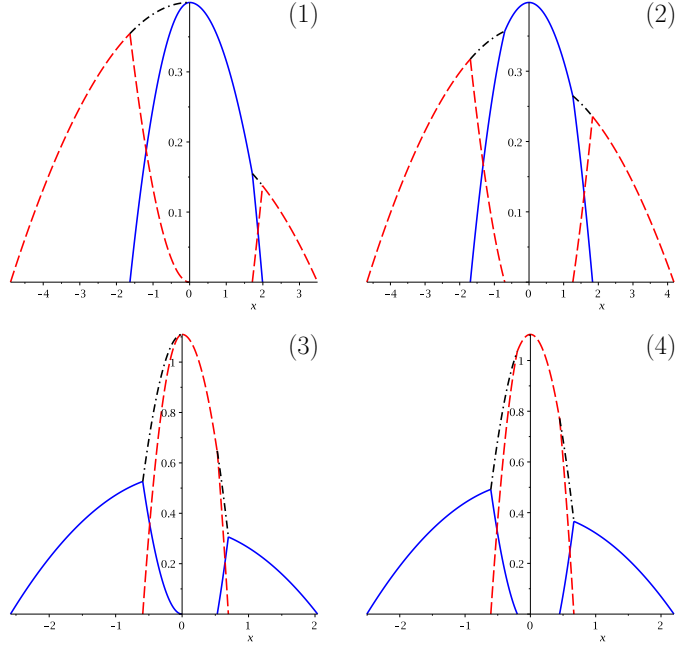}
\caption[]{\small Non-symmetric  self-similar profiles of \eqref{Pb} with disconnected supports. The parameters are $\eta=R=1$ and:  $R_\mu=10$ for $(1)$ and $(2);$ $R_\mu=0.1$ for $(3)$ and $(4)$.} \label{F:3}
\end{figure}
%%%%%%%%%%%%%%%%%%%%%%%%%%%%%%%%%%%%%%%%

\begin{proof} 
Recalling that \eqref{ststa} has only even solutions in \textbf{Cases~(I) \&~(IV)} introduced in the proof of Proposition~\ref{P:-1},
we are left with \textbf{Cases~(II) \&~(III)}.

\medskip

We first return to the \textbf{Case~(II-a)}, that is, $\cP_F\cap \cP_G$ has a connected component $\cI=(\alpha,\beta)$ with 
$0 \leq \alpha<\beta$, $ F(\beta)=G(\alpha)=0$, and   $F(\alpha)G(\beta)>0$. In that case, we already know that necessarily $R_\mu>R+1$ 
and recall the definition \eqref{pachelbel} of $(\beta_1,\gamma)$:
\begin{equation*}
\begin{split}
\beta_1 & := \inf\{x<\alpha\,:\, F>0 \;\;\text{ in }\;\; (x,\beta) \} < \alpha\ , \\
\gamma & := \sup\{x>\beta\,:\, G>0 \;\;\text{ in }\;\; (\alpha,x)\}> \beta\ . 
\end{split}
\end{equation*}
By Lemma~\ref{L:T1}~$(iii)$, $\beta_1<0$ and $\cP_f=(\beta_1,\beta)$ so that it is the support of $G$ which is disconnected.
It then has a connected component $\cJ:= (\gamma_1,\alpha_1)$ which does not intersect $(\alpha,\gamma)$. We claim that
\begin{equation}
\cJ \cap (\beta_1,\alpha) \ne \emptyset \;\;\text{ and }\;\; \beta_1 < \alpha_1 < \alpha\ . \label{ad1}
\end{equation}
Indeed, assume for contradiction that $\cJ \cap (\beta_1,\alpha) = \emptyset$. Then either $\cJ\subset (\gamma,\infty)$ and Lemma~\ref{L:T1}~$(v)$ 
implies that $\cJ\cap \cP_F \ne \emptyset$ and a contradiction. 
Or $\cJ \subset (-\infty,\beta_1)$ and a contradiction follows by the same argument. 
In addition, $\alpha_1<\alpha$ since the support of $G$ is disconnected and we have proved \eqref{ad1}. 

As $\cJ\cap (\beta_1,\alpha) \subset \cP_f \cap \cP_G$, we infer from \eqref{F1} that there are $(a_2,b_2)\in\mathbb{R}^2$ such that
\begin{align*}
\eta^2 F(x) & = a_2 - b_2 - \frac{R_\mu-R}{6 R_\mu} x^2\ , \quad x\in \cJ\cap (\beta_1,\alpha)\ , \\
G(x)  & = \frac{(1+R)b_2-Ra_2}{R} - \frac{1+R-R_\mu}{6R_\mu} x^2\ , \quad x\in \cJ\cap (\beta_1,\alpha)\ ,
\end{align*}
with $G(\alpha_1)=0$ and $\partial_x G(\alpha_1)\le 0$. Recalling that $R_\mu>R+1$, we deduce from the latter that 
\begin{equation}
\alpha_1 \le 0\ . \label{ad2}
\end{equation}
Consequently, $\partial_x G<0$ in $\cJ\cap (\beta_1,\alpha)$ so that $G(x)>G(\alpha_1)=0$ for $x\in \cJ \cap (\beta_1,\alpha)$ and thus $\gamma_1<\beta_1$.
We have thus shown that
\begin{equation}
\gamma_1<\beta_1<\alpha_1\le 0 \le \alpha < \beta < \gamma\ , \label{praetorius}
\end{equation}
and we use once more Lemma~\ref{L:T1}~$(v)$ to conclude that $\cP_G = (\gamma_1,\alpha_1)\cup (\alpha,\gamma)$. Then, according to \eqref{F1}-\eqref{F3}, 
there are $(a,b,a_1,b_1,a_2,b_2,b_3)\in\mathbb{R}^7$ such that 
$$
\eta^2F(x)=
\left\{
\begin{array}{ll} 
\displaystyle a_2-b_2-\frac{R_\mu-R}{6R_\mu}x^2\ , & x\in[\beta_1,\alpha_1]\ , \\
& \\
\displaystyle\frac{a_1}{1+R}-\frac{x^2}{6(1+R)}\ , &  x\in [\alpha_1,\alpha]\ , \\ 
& \\
\displaystyle a-b-\frac{R_\mu-R}{6R_\mu}x^2\ , & x\in [\alpha,\beta]\ ,
\end{array}
\right.
$$
and 
$$
G(x)=
\left\{
\begin{array}{ll}
\displaystyle \frac{b_3}{R}-\frac{x^2}{6R_\mu}\ , & x\in [\gamma_1,\beta_1]\ ,\\
& \\
\displaystyle \frac{(1+R)b_2-Ra_2}{R}-\frac{1+R-R_\mu}{6R_\mu}x^2\ , & x\in [\beta_1,\alpha_1]\ , \\
& \\
\displaystyle \frac{(1+R)b - Ra}{R}-\frac{1+R-R_\mu}{6R_\mu}x^2\ , & x\in [\alpha,\beta]\ ,\\
& \\
\displaystyle\frac{b_1}{R}-\frac{x^2}{6R_\mu}\ , & x\in [\beta,\gamma]\ .
\end{array}
\right. 
$$
We then deduce from the properties of $F$ and $G$ that
\begin{align}
& a_2 - b_2=\frac{R_\mu-R}{6R_\mu}\beta_1^2\ , \qquad a-b = \frac{R_\mu-R}{6R_\mu}\beta^2\ , \label{ad3e4} \\
& a_2 - b_2 - \frac{R_\mu-R}{6 R_\mu} \alpha_1^2 = \frac{a_1}{1+R} - \frac{\alpha_1^2}{6(1+R)} \label{ad5}\\
& \frac{a_1}{1+R} - \frac{\alpha^2}{6(1+R)} = a - b - \frac{R_\mu-R}{6R_\mu} \alpha^2 \ , \label{ad6}\\
& \frac{b_3}{R} = \frac{\gamma_1^2}{6R_\mu}\ , \qquad \frac{(1+R)b_2-Ra_2}{R} = \frac{1+R-R_\mu}{6R_\mu}\alpha_1^2\ , \label{ad7e8}\\
& \frac{(1+R)b-Ra}{R} = \frac{1+R-R_\mu}{6R_\mu}\alpha^2\ , \qquad \frac{b_1}{R}=\frac{\gamma^2}{6R_\mu}\ , \label{ad9e10}\end{align}\begin{align}
& \frac{b_3}{R} - \frac{\beta_1^2}{6R_\mu} = \frac{(1+R)b_2 - R a_2}{R} - \frac{1+R-R_\mu}{6R_\mu}\beta_1^2\ , \label{ad11} \\ 
& \frac{(1+R)b-Ra}{R} - \frac{1+R-R_\mu}{6R_\mu}\beta^2 = \frac{b_1}{R} - \frac{\beta^2}{6R_\mu}\ . \label{ad12}
\end{align}
It follows from \eqref{ad6} and \eqref{ad9e10} that $a=a_1$ and from \eqref{ad5} and \eqref{ad7e8} that $a_1=a_2$. Also, $b_3=b_2$ by \eqref{ad3e4} and \eqref{ad11}
and $b=b_1$ by \eqref{ad3e4} and \eqref{ad12}. Summarizing,
\begin{equation}
a=a_1=a_2\ , \qquad b=b_1\ , \qquad b_2=b_3\ . \label{purcell}
\end{equation} 
Using \eqref{ad3e4}-\eqref{purcell} we identify $a$, $b$, and $b_2$ in terms of $(\gamma_1,\beta_1,\alpha_1,\alpha,\beta,\gamma)$ and find
$$
a = \frac{R}{6 R_\mu}  \gamma^2 + \frac{R_\mu-R}{6R_\mu} \beta^2\ , \quad b_1 = \frac{R}{6 R_\mu} \gamma^2\ , \quad b_2 = \frac{R}{6R_\mu} \gamma_1^2\ .
$$
Combining these identities with \eqref{ad3e4}-\eqref{purcell}, we finally deduce three algebraic equations having $(\gamma_1,\beta_1,\alpha_1,\alpha,\beta,\gamma)$
as unknown
\begin{eqnarray}
 \gamma_1^2-(R_\mu-R)\beta_1^2+(R_\mu-R-1)\alpha_1^2&=&0\ ,\label{R1:1}\\
 \gamma^2-(R_\mu-R)\beta^2+(R_\mu-R-1)\alpha^2&=&0\ ,\label{R1:2}\\
 R(\gamma_1^2-\gamma^2)+(R_\mu-R)(\beta_1^2-\beta^2)&=& 0\ .\label{R1:3}
\end{eqnarray}
There are two more equations obtained from the constraints $(F,G)\in\cK^2$, namely 
\begin{eqnarray}
(R_\mu-R)(\beta^3-\beta_1^3)-\frac{R(R_\mu-R-1)}{1+R}(\alpha^3-\alpha_1^3)&=&9\eta^2R_\mu\ ,\label{R1:4}\\
 (\gamma^3-\gamma_1^3)-(R_\mu-R)(\beta^3-\beta_1^3)+(R_\mu-R-1)(\alpha^3-\alpha_1^3)&=&9R_\mu\ .\label{R1:5}
\end{eqnarray}
Recalling \eqref{praetorius}, it follows from Lemma~\ref{L:2.12b}~$(iii)$ that, if $\alpha_1=-\alpha$, then $(F,G)$ is an even solution of \eqref{ststa} which
contradicts the assumption of non-symmetric profiles. Consequently, $(\gamma_1,\beta_1,\alpha_1,\alpha,\beta,\gamma)$ satisfies \eqref{haydn}. 
We have thus established that, if the system \eqref{R1:1}-\eqref{R1:5} has a solution $(\gamma_1,\beta_1,\alpha_1,\alpha,\beta,\gamma)$ satisfying \eqref{haydn}, 
then the pair $(F,G)$ 
defined in Proposition~\ref{P:1}~$(i)$ is a non-symmetric solution of \eqref{ststa}, the function $G$ having clearly a disconnected support.
Since its reflection $x\mapsto (F(-x),G(-x))$ 
also solves \eqref{ststa} by Lemma~\ref{L:2.4c}~$(i)$, this completes the proof of Proposition~\ref{P:1}~$(i)$.

\medskip

We next consider \textbf{Case~(II-b)}, where $\cP_F \cap \cP_G$ has a connected component $\cI = (\alpha,\beta)$ with $\alpha<\beta\le 0$, $F(\beta)=G(\alpha)=0$,
and $F(\alpha)G(\beta)>0$. 
As in the proofs of Propositions~\ref{P:-1} and~\ref{P:0}, we use Lemma~\ref{L:2.4c} to map this case to the one previously studied.
Recalling the definition \eqref{pachelbel} of $\beta_1$ and $\gamma$ and defining the parameters $(R_{\mu,1},\eta)$ and $\lambda$ as in Lemma~\ref{L:2.4c}, 
the pair $(F_1,G_1)$ given by
$$
F_1(x) = \lambda G(-\lambda x)\ , \qquad G_1(x) = \lambda F_1(-\lambda x)\ , \qquad x\in\R\ ,
$$
is a solution of \eqref{ststa} with parameters $(R,R_{\mu,1},\eta_1)$ such that $(-\beta/\lambda,-\alpha/\lambda)$ is a connected component 
of $\cP_{F_1} \cap \cP_{G_1}$ with $0\le -\beta/\lambda<-\alpha/\lambda$, $F_1(-\alpha/\lambda)=G_1(-\beta/\lambda)=0$, and $F_1(-\beta/\lambda) G_1(-\alpha/\lambda)>0$.
We are then in the situation analysed in \textbf{Case~(II-a)} and deduce that, if $R_{\mu,1}>R+1$ and the system \eqref{R1:1}-\eqref{R1:5} with 
parameters $(R,R_{\mu,1},\eta_1)$ instead of $(R,R_\mu,\eta)$ has a solution $(\gamma_1',\beta_1',\alpha_1',\alpha',\beta',\gamma')\in \R^6$ satisfying 
$$
\gamma_1'<\beta_1'<\alpha_1'\le 0 \le \alpha' < \beta'< \gamma'\ , \quad \alpha_1'\ne -\alpha'\ ,
$$
then $(F_1,G_1)$ are given by
\begin{align*}
\eta_1^2 F_1(x)&=
\left\{
\begin{array}{ll}
\displaystyle \frac{R_{\mu,1}-R}{6 R_{\mu,1}}((\beta_1')^2-x^2)\ , & x\in [\beta_1',\alpha_1']\ ,\\
 & \\
 \displaystyle\frac{R(\gamma_1')^2+(R_{\mu,1}-R) (\beta_1')^2}{6 (1+R)R_{\mu,1}}-\frac{x^2}{6(1+R)}\ , &  x\in [\alpha_1',\alpha']\ ,\\ 
  & \\
\displaystyle \frac{R_{\mu,1}-R}{6R_{\mu,1}}((\beta')^2 -x^2)\ , & x\in [\alpha',\beta']\ ,
\end{array}
\right.\\
G_1(x)&=
\left\{
\begin{array}{ll}
\displaystyle \frac{(\gamma_1')^2-x^2}{6R_{\mu,1}}\ , & x\in [\gamma_1',\beta_1']\ ,\\
& \\
\displaystyle\frac{1+R-R_{\mu,1}}{6R_{\mu,1}}((\alpha_1')^2-x^2)\ ,& x\in [\beta_1',\alpha_1']\ ,\\
& \\
\displaystyle \frac{1+R-R_{\mu,1}}{6R_{\mu,1}}((\alpha')^2-x^2)\ ,& x\in [\alpha',\beta']\ ,\\
& \\
\displaystyle\frac{(\gamma')^2-x^2}{6R_{\mu,1}}\ , & x\in [\beta',\gamma']\ ,
\end{array}
\right. 
\end{align*}
and $\cP_{F_1}=(\beta_1',\beta'),$ $\cP_{G_1}=(\gamma_1',\alpha_1')\cup (\alpha',\gamma')$. 
The condition $R_{\mu,1}>1+R$ readily implies that $R_\mu<R$ while the properties of $(\gamma_1',\beta_1',\alpha_1',\alpha',\beta',\gamma')$ entail that the system 
\begin{eqnarray}
R_\mu\gamma_1^2-R(1+R-R_\mu)\beta_1^2+(1+R)(R-R_\mu)\alpha_1^2 & =&0 \ ,\label{R2:1}\\
R_\mu\gamma^2-R(1+R-R_\mu)\beta^2+(1+R)(R-R_\mu)\alpha^2 & =&0\ ,\label{R2:2}\\
R_\mu(\gamma^2-\gamma_1^2)+(R-R_\mu)(\alpha^2-\alpha_1^2) & =& 0\ ,\label{R2:3} \\
R_\mu(\gamma^3-\gamma_1^3)-R(1+R-R_\mu)(\beta^3-\beta_1^3)\!\!\!\! & +& \!\!\!\!(1+R)(R-R_\mu)(\alpha^3-\alpha_1^3) \nonumber \\
& =&9\eta^2R_\mu(1+R)\ ,\label{R2:4}\\
(1+R-R_\mu)(\beta^3-\beta_1^3)-(R-R_\mu)(\alpha^3-\alpha_1^3) & =&9R_\mu\ \label{R2:5}
\end{eqnarray}
has a solution $(\gamma_1,\beta_1,\alpha_1,\alpha,\beta,\gamma)
=(-\gamma'/\lambda,-\beta'/\lambda,-\alpha'/\lambda,-\alpha_1'/\lambda,-\beta_1'/\lambda,-\gamma_1'/\lambda)$ satisfying \eqref{haydn}. 
Using this notation, the above identities for $(F_1,G_1)$ ensure that $(F,G)$ are given by Proposition~\ref{P:1}~$(ii)$. 
That its reflection $x\mapsto (F(-x),G(-x))$ is also a solution of \eqref{ststa} follows again from Lemma~\ref{L:2.4c}~$(i)$.

\medskip

Finally, as already observed in the proofs of Propositions~\ref{P:-1} and~\ref{P:0}, the \textbf{Case~(III)} ($F(\alpha)=G(\beta)=0$) reduces to the \textbf{Case~(II)}. 
We have thus identified all possible non-symmetric solutions with disconnected supports and thereby completed the proof of Proposition~\ref{P:1}.
\end{proof}

We note that the systems \eqref{R1:1}-\eqref{R1:5} and \eqref{R2:1}-\eqref{R2:5} are both under-determined, having six unknowns and only five equations. 
In this case we can no longer expect a uniqueness result as for the system \eqref{eq:41}-\eqref{eq:43} determining the 
even profiles or for the system \eqref{d1:V1}-\eqref{d1:V4} 
determining the profiles with connected supports. Instead, we can prove that each solution $(\gamma_1,\beta_1,\alpha_1,\alpha,\beta,\gamma)$ of \eqref{R1:1}-\eqref{R1:5} 
and  \eqref{R2:1}-\eqref{R2:5} satisfying 
\begin{equation}
\gamma_1<\beta_1<\alpha_1<0<\alpha<\beta<\gamma \label{clementi}
\end{equation} 
belongs to a real-analytic curve consisting, with the exception of the even profile, only of non-symmetric solutions of \eqref{ststa}. 

%%%%%%%%%%%%%%%%%%%%%%%%%%%%%%%%%%%%%%%%
\begin{prop}\label{P:2}
 Let $R$, $R_\mu$, and $\eta$ be given positive real numbers such that $R_\mu>R+1$ and consider a solution
 $(\gamma_1^*,\beta_1^*,\alpha_1^*,\alpha^*,\beta^*,\gamma^*)$ of \eqref{R1:1}-\eqref{R1:5} satisfying \eqref{clementi}. 
 Then there exist $\overline{\alpha}_1\in (\alpha_1^*,0]$, $\underline{\alpha}_1\in (-\infty,\alpha_1^*)$, and a bounded continuous function 
\begin{equation}
\varphi:=(\varphi_1,\varphi_2,\varphi_3,\varphi_4,\varphi_5,\varphi_6):
[\underline{\alpha}_1,\overline{\alpha}_1]\to\R^6 \;\;\text{ with }\;\; \varphi_3 \equiv \mathrm{id}\ , \label{vivaldi}
\end{equation} 
which is real-analytic in $(\underline{\alpha}_1,\overline{\alpha}_1)$  and has the following properties:

\begin{itemize}

\item[$(a)$] Given $\alpha_1\in (\underline{\alpha}_1,\overline{\alpha}_1),$ the sextuplet $\varphi(\alpha_1)$ is a 
solution of \eqref{R1:1}-\eqref{R1:5} satisfying \eqref{clementi};

\item[$(b)$] The end points $\varphi(\underline{\alpha}_1)$ and $\varphi(\overline{\alpha}_1)$ satisfy
$$
\varphi_{7-k}(\underline{\alpha}_1) = - \varphi_k(\overline{\alpha}_1)\ , \quad 1 \le k \le 6\ .
$$
If $R_\mu\ge R_\mu^M(R,\eta)$, then 
$$
\varphi_1(\underline{\alpha}_1)=\varphi_2(\underline{\alpha}_1)=\varphi_3(\underline{\alpha}_1)<0\leq
\varphi_4(\underline{\alpha}_1)< \varphi_5(\underline{\alpha}_1)< \varphi_6(\underline{\alpha}_1)
$$ 
and $(\varphi_2(\underline{\alpha}_1), \varphi_4(\underline{\alpha}_1), \varphi_5(\underline{\alpha}_1), \varphi_6(\underline{\alpha}_1))$ 
solves \eqref{d1:V1}-\eqref{d1:V4} and is given by Lemma~\ref{P:NSC}. 
If $R_\mu\in \left(R_\mu^+(R,\eta),R_\mu^M(R,\eta) \right)$, then $\varphi(\underline{\alpha}_1)$ is the solution of \eqref{R1:1}-\eqref{R1:5}
given by Lemma~\ref{L:Nconn} and satisfies $\varphi_4(\underline{\alpha}_1)=0$. 

\item[$(c)$] If $R_\mu>R_\mu^+(R,\eta)$, we denote the unique solution to \eqref{eq:41}-\eqref{eq:43} given by Lemma~\ref{L:Z1} by $(\alpha_*,\beta_*,\gamma_*)$. 
Then
$$
-\alpha_* \in (\underline{\alpha}_1,\overline{\alpha}_1) \;\;\text{ and }\;\; \varphi(-\alpha_*) = (-\gamma_*,-\beta_*,-\alpha_*,\alpha_*,\beta_*,\gamma_*)\ .
$$
\end{itemize}
\end{prop}
%%%%%%%%%%%%%%%%%%%%%%%%%%%%%%%%%%%%%%%%

\begin{proof} Using $\alpha_1$ as a parameter, we show that we may apply the implicit function theorem to the system \eqref{R1:1}-\eqref{R1:5} in a 
neighborhood of $(\gamma_1^*,\beta_1^*,\alpha_1^*,\alpha^*,\beta^*,\gamma^*)$.
To this end we recast the system \eqref{R1:1}-\eqref{R1:5} as an equation $\Phi(\gamma_1, \beta_1,\alpha_1, \alpha,\beta,\gamma)=0,$ where
$\Phi:\R^6\to\R^5$ is the real-analytic 
function with components defined by the equations \eqref{R1:1}-\eqref{R1:5}. We need to show that  the derivative
$\p_{(\gamma_1,\beta_1,\alpha,\beta,\gamma)}\Phi(\gamma^*_1,\beta^*_1,\alpha^*_1,\alpha^*, \beta^*,\gamma^*) $ is invertible.
It turns out that 
\begin{align*}
& J_\Phi(\gamma^*_1,\beta^*_1,\alpha^*_1,\alpha^*, \beta^*,\gamma^*) :=
\left| \mathrm{det}\,  \p_{ (\gamma_1,\beta_1,\alpha,\beta,\gamma)}\Phi(\gamma^*_1,\beta^*_1,\alpha^*_1,\alpha^*, \beta^*,\gamma^*) \right| \\
& = 72(R_\mu-R-1)(R_\mu-R)^2\alpha^*\beta^*\beta_1^*\gamma^*\gamma_1^* \ \left[ (\beta^*-\beta_1^*)(\gamma^*-\alpha^*)+R(\gamma^*-\gamma_1^*)(\beta^*-\alpha^*) \right]\ ,
\end{align*}
and we infer from \eqref{clementi} that $J_\Phi(\gamma^*_1,\beta^*_1,\alpha^*_1,\alpha^*, \beta^*,\gamma^*)>0$. 
We are thus in a position to use the implicit function theorem and obtain the existence of a maximal open interval
$\mathcal{I}^* := (\underline{\alpha}_1,\overline{\alpha}_1)$ containing $\alpha_1^*$ 
and a real-analytic function $\varphi=(\varphi_i)_{1\le i \le 6}: \mathcal{I}^*\to\mathbb{R}^6$ such that  $\varphi_3\equiv \mathrm{id}$
and $\varphi(\alpha_1)$ solves \eqref{R1:1}-\eqref{R1:5} and satisfies \eqref{clementi} for all $\alpha_1\in \mathcal{I}^*$. 

\smallskip

We now claim that $\mathcal{I}^*$ is a bounded interval included in $(-\infty,0)$ and that  $\varphi$ is  also bounded. 
Indeed, the equation \eqref{R1:4} implies, in view of $R_\mu>R+1$, $\varphi_5^3>\varphi_4^3$, and $-\varphi_2^3>-\varphi_3^3$ that 
$$
0<\max\left\{ \varphi_4(\alpha_1)^3, - \alpha_1^3 \right\} \le \varphi_4(\alpha_1)^3-\alpha_1^3<9\eta^2(1+R) \ , \quad \alpha_1\in \mathcal{I}^*\ ,
$$
which implies that
$$
\mathcal{I}^* \subset \left( - (9 (1+R)\eta^2)^{1/3} , 0 \right) \;\text{ and }\; \varphi_4(\mathcal{I}^*) \subset \left( 0 , (9 (1+R)\eta^2)^{1/3} \right)\ .
$$
Using again \eqref{R1:4} and the positivity of $R_\mu-R$ we realize that 
$$
0 < \max\left\{ \varphi_5(\alpha_1)^3, - \varphi_2(\alpha_1)^3 \right\} \le \varphi_5(\alpha_1)^3-\varphi_2(\alpha_1)^3<9\eta^2(1+R) \ , \quad \alpha_1\in \mathcal{I}^*\ ,
$$
hence the boundedness of $\varphi_2$ and $\varphi_5$. In a similar way, we use \eqref{R1:5} to establish the boundedness of $\varphi_1$ and $\varphi_6$. 
Next, as a consequence of the boundedness of $\mathcal{I}^*$ and $\varphi$, there are two sequences  $\underline{\alpha}_{1,n}\searrow\underline{\alpha}_1$
and $\overline{\alpha}_{1,n}\nearrow\overline{\alpha}_1$  such that the limits
$$
(\underline{\gamma}_1,\underline{\beta}_1,\underline{\alpha}_1,\underline{\alpha},\underline{\beta},\underline{\gamma})
:= \lim_{n\to\infty}\varphi(\underline{\alpha}_{1,n})\qquad\text{and}\qquad 
(\overline{\gamma}_1,\overline{\beta}_1,\overline{\alpha}_1,\overline{\alpha},\overline{\beta},\overline{\gamma}) := \lim_{n\to\infty}\varphi(\overline{\alpha}_{1,n})
$$
exist in $\R^6.$ Clearly $(\underline{\gamma}_1,\underline{\beta}_1,\underline{\alpha}_1,\underline{\alpha},\underline{\beta},\underline{\gamma})$ 
and $(\overline{\gamma}_1,\overline{\beta}_1,\overline{\alpha}_1,\overline{\alpha},\overline{\beta},\overline{\gamma})$ solve \eqref{R1:1}-\eqref{R1:5} but the 
maximality of the interval $\mathcal{I}^* = (\underline{\alpha}_1,\overline{\alpha}_1)$ prevents them from satisfying \eqref{clementi}.
However, recalling that $\underline{\alpha}_1<\alpha_1^*<0$ entails that we necessarily have
\begin{equation}
0 \in \left\{ \underline{\beta}_1 - \underline{\gamma}_1 , 
\underline{\alpha}_1 - \underline{\beta}_1 , \underline{\alpha} , \underline{\beta} - \underline{\alpha} , \underline{\gamma} - \underline{\beta} \right\}\ , \label{verdi}
\end{equation}
and
\begin{equation}
0\in \left\{ \overline{\beta}_1 - \overline{\gamma}_1 , \overline{\alpha}_1 - \overline{\beta}_1 , \overline{\alpha}_1, \overline{\alpha} , \overline{\beta} - \overline{\alpha} , \overline{\gamma} - \overline{\beta} \right\}\ . \label{puccini}
\end{equation}
We shall now prove that 
$(\overline{\gamma}_1,\overline{\beta}_1,\overline{\alpha}_1,\overline{\alpha},\overline{\beta},\overline{\gamma})=(-\underline\gamma,-\underline\beta,-\underline\alpha,-\underline\alpha_1,-\underline\beta_1,-\underline\gamma_1)$ and that
\begin{itemize}
 \item for $R_\mu \ge R_\mu^M(R,\eta)$ we 
 $\underline{\gamma}_1=\underline{\beta}_1 = \underline{\alpha}_1 <0\leq\underline{\alpha}<\underline{\beta}<\underline{\gamma}$
 and $( \underline{\alpha}_1,\underline{\alpha},\underline{\beta},\underline{\gamma})$
is the unique solution of \eqref{d1:V1}-\eqref{d1:V4}  satisfying \eqref{chopin}. 
\item for $R_\mu\in \left( R_\mu^+(R,\eta),R_\mu^M(R,\eta) \right)$ we have 
$\underline{\gamma}_1 < \underline{\beta}_1 < \underline{\alpha}_1 < 0 =\underline{\alpha} < \underline{\beta} < \underline{\gamma}$
and  $(\underline{\gamma}_1, \underline{\beta}_1, \underline{\alpha}_1,  0, \underline{\beta}, \underline{\gamma}) $ is
the unique solution of \eqref{R1:1}-\eqref{R1:5} satisfying \eqref{monteverdi}.
\end{itemize}
By \eqref{verdi} we may face the following three situations:
$\underline{\beta}_1 = \underline{\gamma}_1$ or $\underline{\beta}_1 = \underline{\alpha}_1$, $\underline{\beta} = \underline{\gamma}$ 
or $\underline{\beta} = \underline{\alpha}$, and $\underline{\alpha}=0$ which we handle separately.

\medskip

\noindent\textbf{Case~(i).} Assume that $\underline{\beta}_1 = \underline{\gamma}_1$ or $\underline{\beta}_1 = \underline{\alpha}_1$. 
In both cases we deduce from \eqref{R1:1}-\eqref{R1:3} that $\underline{\gamma}_1=\underline{\beta}_1 = \underline{\alpha}_1 <0$. 
It then follows from Lemma~\ref{L:2.12b}~$(i)$ that $(\underline{\alpha}_1,\underline{\alpha},\underline{\beta},\underline{\gamma})$ 
satisfies $\underline{\alpha}_1<0\leq\underline{\alpha}<\underline{\beta}<\underline{\gamma}$ and solves \eqref{d1:V1}-\eqref{d1:V4}. 
According to Lemma~\ref{P:NSC} such a solution exists only if 
\begin{equation}
R_\mu \ge R_\mu^M(R,\eta)\ . \label{supertramp}
\end{equation}

Concerning $(\overline{\gamma}_1,\overline{\beta}_1,\overline{\alpha}_1,\overline{\alpha},\overline{\beta},\overline{\gamma})$, let us first consider the following case:

\smallskip

\noindent\textbf{Case~(i1): $\overline{\alpha}_1<0$.} 
Assume for contradiction that $\overline{\beta}_1 = \overline{\gamma}_1$ or $\overline{\beta}_1 = \overline{\alpha}_1$. 
Arguing as above, we deduce from \eqref{R1:1}-\eqref{R1:3}  and Lemma~\ref{L:2.12b}~$(i)$ that $\overline{\beta}_1 = \overline{\alpha}_1 = \overline{\gamma}_1<0$ and 
$(\overline{\alpha}_1,\overline{\alpha},\overline{\beta},\overline{\gamma})$ satisfies $\overline{\alpha}_1<0\leq\overline{\alpha}<\overline{\beta}<\overline{\gamma}$ 
and is the unique solution of \eqref{d1:V1}-\eqref{d1:V4} given by Lemma~\ref{P:NSC}. Thus $(\overline{\alpha}_1,\overline{\alpha},\overline{\beta},\overline{\gamma})$ 
coincides 
with $( \underline{\alpha}_1,\underline{\alpha},\underline{\beta},\underline{\gamma})$. 
Since this fact contradicts the property $\underline{\alpha}_1<\alpha_1^*<\overline{\alpha}_1$, we conclude that
$$
\overline{\gamma}_1 < \overline{\beta}_1 < \overline{\alpha}_1 < 0\ . 
$$
Assume next for contradiction that $\overline{\alpha}<\overline{\beta}$ or $\overline{\beta}<\overline{\gamma}$.
Then $\overline{\alpha}<\overline{\beta}<\overline{\gamma}$ by \eqref{R1:2} and it follows from \eqref{puccini} that $\overline{\alpha}=0$.
According to Lemma~\ref{L:Nconn}, the system \eqref{R1:1}-\eqref{R1:5} has such a solution only if $R_\mu^+(R,\eta)<R_\mu<R_\mu^M(R,\eta)$
which is not compatible with \eqref{supertramp}. 
Consequently,
$$
0<\overline{\alpha}=\overline{\beta}=\overline{\gamma}\ ,
$$
the positivity of $\overline{\alpha}=0$ being a consequence of \eqref{R1:3}. 
We then infer from Lemma~\ref{L:2.12b}~$(ii)$ that $(-\overline{\alpha},-\overline{\alpha}_1,-\overline{\beta}_1,-\overline{\gamma}_1)$ is
the unique solution of \eqref{d1:V1}-\eqref{d1:V4} given by
Lemma~\ref{P:NSC} which is known to exist for $R> R_\mu^M(R,\eta)$ since $-\overline{\alpha}_1>0$.
We have thus shown that, in \textbf{Case~(i1)}, one has necessarily $R_\mu>R_\mu^M(R,\eta)$ and both
$(\underline{\alpha}_1,\underline{\alpha},\underline{\beta},\underline{\gamma})$ 
and $(-\overline{\alpha},-\overline{\alpha}_1,-\overline{\beta}_1,-\overline{\gamma}_1)$ solve \eqref{d1:V1}-\eqref{d1:V4}. 
According to Lemma~\ref{P:NSC}, 
$$
( \underline{\alpha}_1,\underline{\alpha},\underline{\beta},\underline{\gamma}) = (-\overline{\alpha},-\overline{\alpha}_1,-\overline{\beta}_1,-\overline{\gamma}_1)
$$
and the uniqueness statement in Lemma~\ref{P:NSC} entails that the limits as $\alpha_1\searrow \underline{\alpha}_1$ and $\alpha_1\nearrow \overline{\alpha}_1$ are
both well-defined and uniquely determined. We may thus extend $\varphi$ by continuity to $[\underline{\alpha}_1,\overline{\alpha}_1]$ by 
$$
\varphi(\underline{\alpha}_1) := (\underline{\alpha}_1,\underline{\alpha}_1,\underline{\alpha}_1,\underline{\alpha},\underline{\beta},\underline{\gamma}) \;\text{ and }\; 
\varphi(\overline{\alpha}_1) := (-\underline{\gamma},-\underline{\beta},-\underline{\alpha},-\underline{\alpha}_1,-\underline{\alpha}_1,-\underline{\alpha}_1)
$$ 
and complete the proof of Proposition~\ref{P:2}~$(a)$-$(b)$ in that case.

\smallskip

\noindent\textbf{Case~(i2): $\overline{\alpha}_1=0$.} There are several possibilities which we analyze successively.

\begin{itemize}

\item[$-$] If $\overline{\gamma}_1<\overline{\beta}_1<0=\overline{\alpha}_1<\overline{\alpha}<\overline{\beta}<\overline{\gamma}$, 
then $(-\overline{\gamma},-\overline{\beta},-\overline{\alpha},0,-\overline{\beta}_1,-\overline{\gamma}_1)$ is 
the unique solution to \eqref{R1:1}-\eqref{R1:5} found in Lemma~\ref{L:Nconn}. As it only exists when $R^+_\mu(R,\eta)<R_\mu<R_\mu^M(R,\eta)$,
this case is excluded according to \eqref{supertramp}.

\item[$-$] If $\overline{\gamma}_1=\overline{\beta}_1$ or $\overline{\alpha}_1=\overline{\beta}_1$, then $\overline{\gamma}_1=\overline{\beta}_1=\overline{\alpha}_1=0$  
by \eqref{R1:1}.
Owing to \eqref{R1:2} and \eqref{R1:3}, this in turn implies that $\overline{\gamma}=\overline{\beta}=\overline{\alpha}=0$ and
a contradiction with \eqref{R1:4}. This case is thus excluded as well.

\item[$-$] If $\overline{\alpha}=0,$ then $\overline{\alpha}_1=\overline{\alpha}$ so that $-\overline{\gamma}=\overline{\gamma}_1 < \overline{\beta}_1=-\overline{\beta}$, 
and $(\overline{\alpha},\overline{\beta},\overline{\gamma})$ solves \eqref{eq:41}-\eqref{eq:43} by Lemma~\ref{L:2.12b}~$(iii)$ with $\overline{\alpha}=0$. 
Recalling Lemma~\ref{L:Z1} this implies that $R_\mu=R_\mu^+(R,\eta)$ and does not match \eqref{supertramp} since $R_\mu^+(R,\eta)<R_\mu^M(R,\eta)$. 
This case is thus also excluded.

\item[$-$] If $\overline{\gamma}=\overline{\beta}$ or 
$\overline{\alpha}=\overline{\beta},$ then $\overline{\gamma}_1<\overline{\beta}_1<\overline{\alpha}_1=0<\overline{\alpha}=\overline{\beta}=\overline{\gamma}$ 
and $(-\overline{\alpha},0,-\overline{\beta}_1,-\overline{\gamma}_1)$ 
solves \eqref{d1:V1}-\eqref{d1:V4} by Lemma~\ref{L:2.12b}~$(ii)$ and the previous case. Gathering these information we deduce from Lemma~\ref{P:NSC} 
that necessarily $R_\mu=R_\mu^M(R,\eta)$.
Since $( \underline{\alpha}_1,\underline{\alpha},\underline{\beta},\underline{\gamma})$ is also a solution of \eqref{d1:V1}-\eqref{d1:V4} when $R_\mu=R_\mu^M(R,\eta)$, 
we use again Lemma~\ref{P:NSC} to conclude that $\underline{\alpha}=0$ and 
$( \underline{\alpha}_1,\underline{\alpha},\underline{\beta},\underline{\gamma}) = (-\overline{\alpha},0,-\overline{\beta}_1,-\overline{\gamma}_1)$. 
We then extend $\varphi$ by continuity to $[\underline{\alpha}_1,\overline{\alpha}_1]$ by 
$$
\varphi(\underline{\alpha}_1) := (\underline{\alpha}_1,\underline{\alpha}_1,\underline{\alpha}_1,0,\underline{\beta},\underline{\gamma}) 
\;\text{ and }\; \varphi(\overline{\alpha}_1) := (-\underline{\gamma},-\underline{\beta},0,-\underline{\alpha}_1,-\underline{\alpha}_1,-\underline{\alpha}_1)\ ,
$$
and complete the proof of Proposition~\ref{P:2}~$(a)$-$(b)$ in that case.

\end{itemize}

\medskip

\noindent\textbf{Case~(ii).} 
We now turn to the case $\underline{\alpha}=\underline{\beta}$ or $\underline{\beta}=\underline{\gamma}$. 
Then, according to Lemma~\ref{L:2.12b}~$(ii)$, 
$\underline{\gamma}_1<\underline{\beta}_1<\underline{\alpha}_1\le 0<\underline{\alpha}=\underline{\beta}=\underline{\gamma}$ and 
$(-\underline{\alpha},-\underline{\alpha}_1,-\underline{\beta}_1,-\underline{\gamma}_1)$ is a solution of \eqref{d1:V1}-\eqref{d1:V4}. 
According to Lemma~\ref{P:NSC}, such a solution exists only if $R_\mu\ge R_\mu^M(R,\eta)$, that is, \eqref{supertramp} holds true and it satisfies 
\begin{equation}
\underline{\alpha}>-\underline{\alpha}_1\ .\label{rossini}
\end{equation}

As for $(\overline{\gamma}_1,\overline{\beta}_1,\overline{\alpha}_1,\overline{\alpha},\overline{\beta},\overline{\gamma})$, 
we study separately the cases $\overline{\alpha}_1<0$ and $\overline{\alpha}_1=0$.

\smallskip

\noindent\textbf{Case~(ii1): $\overline{\alpha}_1<0$.} 
Assume first for contradiction that $\overline{\alpha}=\overline{\beta}$ or $\overline{\beta}=\overline{\gamma}$. 
Then $\overline{\gamma}<\overline{\beta}<\overline{\alpha}\le 0< \overline{\alpha}=\overline{\beta}=\overline{\gamma}$ and 
$(-\overline{\alpha},-\overline{\alpha}_1,-\overline{\beta}_1,-\overline{\gamma}_1)$ solves \eqref{d1:V1}-\eqref{d1:V4} by Lemma~\ref{L:2.12b}~$(ii).$ 
We then infer from Lemma~\ref{P:NSC} that 
$$
(-\underline{\alpha},-\underline{\alpha}_1,-\underline{\beta}_1,-\underline{\gamma}_1) 
= (-\overline{\alpha},-\overline{\alpha}_1,-\overline{\beta}_1,-\overline{\gamma}_1)\ ,
$$
hence $\underline{\alpha}_1 = \overline{\alpha}_1$ and a contradiction with $\underline{\alpha}_1 < \alpha_1^* < \overline{\alpha}_1$.
Therefore
\begin{equation}
\overline{\alpha} < \overline{\beta} < \overline{\gamma}\ . \label{silbermond}
\end{equation}
Assume next for contradiction that $\overline{\gamma}_1 <\overline{\beta}_1 <\overline{\alpha}_1$. 
It then follows from \eqref{puccini} and \eqref{silbermond} that $\overline{\alpha}=0$, 
so that $(\overline{\gamma}_1,\overline{\beta}_1,\overline{\alpha}_1,\overline{\alpha},\overline{\beta},\overline{\gamma})$ is
the solution of \eqref{R1:1}-\eqref{R1:5} given by Lemma~\ref{L:Nconn}. 
Since the existence of such a solution requires $R_\mu\in (R_\mu^+(R,\eta),R_\mu^M(R,\eta))$ according to Lemma~\ref{L:Nconn}, 
this is not compatible with \eqref{supertramp} and we again end up with a contradiction. 
Consequently, $\overline{\alpha}_1=\overline{\beta}_1$ or $\overline{\beta}_1=\overline{\gamma}_1$.
Then $\overline{\alpha}_1=\overline{\beta}_1=\overline{\gamma}_1 <0\le \overline{\alpha} <\overline{\beta}<\overline{\gamma}$
and $(\overline{\alpha}_1,\overline{\alpha},\overline{\beta},\overline{\gamma})$ solves \eqref{d1:V1}-\eqref{d1:V4} 
by Lemma~\ref{L:2.12b}~$(i)$. Lemma~\ref{P:NSC} then guarantees that
$(\overline{\alpha}_1,\overline{\alpha},\overline{\beta},\overline{\gamma}) 
= (-\underline{\alpha}, - \underline{\alpha}_1, - \underline{\beta}_1 , - \underline{\gamma}_1)$. 
Recalling \eqref{rossini}, we realize that $\overline{\alpha}_1 = - \underline{\alpha} < \underline{\alpha}_1$ and 
thereby obtain a contradiction. We have therefore excluded that $\overline{\alpha}_1<0$ in \textbf{Case~(ii)}.

\smallskip

\textbf{Case~(ii2): $\overline{\alpha}_1=0$.} Arguing as in the analysis of \textbf{Case~(i2)} we exclude the following situations:
\begin{itemize}

\item[$-$] $\overline{\gamma}_1<\overline{\beta}_1<\overline{\alpha}_1=0<\overline{\alpha}<\overline{\beta}<\overline{\gamma}$, %then $(-\overline{\gamma},-\overline{\beta},-\overline{\alpha},0,-\overline{\beta}_1,-\overline{\gamma}_1)$ is the solution of \eqref{R1:1}-\eqref{R1:5} given by Lemma~\ref{L:Nconn} which only exists if $R_\mu\in (R_\mu^+(R,\eta),R_\mu^M(R,\eta))$ and \eqref{supertramp} excludes this possibility.

\item[$-$] $\overline{\gamma}_1=\overline{\beta}_1$ or $\overline{\beta}_1=\overline{\alpha}_1$, %then $\overline{\gamma}_1=\overline{\beta}_1=\overline{\alpha}_1=0$ which, together with \eqref{R1:2} and \eqref{R1:3}, implies that $\overline{\alpha}=\overline{\beta} = \overline{\gamma}=0$ and contradicts \eqref{R1:4} and \eqref{R1:5}. This case is thus also excluded.

\item[$-$] $\overline{\alpha}=0$. %then $\overline{\alpha}_1=\overline{\alpha}$ so that $\overline{\gamma}_1=-\overline{\gamma}$, $\overline{\beta}_1=-\overline{\beta}$, and $(\overline{\alpha},\overline{\beta},\overline{\gamma})$ is a solution of \eqref{eq:41}-\eqref{eq:43} byLemma~\ref{L:2.12b}~$(iii)$ with $\overline{\alpha}=0$. Recalling Lemma~\ref{L:Z1} this implies that $R_\mu=R_\mu^+(R,\eta)$ and \eqref{supertramp} and the property $R_\mu^+(R,\eta)<R_\mu^M(R,\eta)$ exclude the occurrence of this case as well.

\end{itemize}

Consequently $\overline{\alpha}=\overline{\beta}$ or $\overline{\gamma}=\overline{\beta}$ with $\overline{\alpha}>0$.
We then argue as at the end of the analysis of \textbf{Case~(i2)} to deduce from Lemma~\ref{P:NSC} and Lemma~\ref{L:2.12b}~$(ii)$ that 
necessarily $R_\mu=R_\mu^M(R,\eta)$ and that $\underline{\alpha}=0$ and 
$(\underline{\alpha}_1,\underline{\alpha},\underline{\beta},\underline{\gamma}) = (-\overline{\alpha},0,-\overline{\beta}_1,-\overline{\gamma}_1)$
is the solution of \eqref{d1:V1}-\eqref{d1:V4} given by  Lemma~\ref{P:NSC} in that case. 
We then extend $\varphi$ by continuity to $[\underline{\alpha}_1,\overline{\alpha}_1]$ by 
$$
\varphi(\underline{\alpha}_1) := (\underline{\alpha}_1,\underline{\alpha}_1,\underline{\alpha}_1,0,\underline{\beta},\underline{\gamma}) \;\text{ and }\;
\varphi(\overline{\alpha}_1) := (-\underline{\gamma},-\underline{\beta},0,-\underline{\alpha}_1,-\underline{\alpha}_1,-\underline{\alpha}_1)\ ,
$$
and complete the proof of Proposition~\ref{P:2}~$(a)$-$(b)$ in that case.

\medskip

\noindent\textbf{Case~(iii).} Owing to the above analysis, we may assume that
$$
\underline{\gamma}_1 < \underline{\beta}_1 < \underline{\alpha}_1 < 0 \;\;\text{ and }\;\; \underline{\alpha} < \underline{\beta} < \underline{\gamma}\ ,
$$
and infer from \eqref{verdi} that $\underline{\alpha}=0$.
Then $(\underline{\gamma}_1, \underline{\beta}_1, \underline{\alpha}_1,  0, \underline{\beta}, \underline{\gamma})$ is 
the unique solution of \eqref{R1:1}-\eqref{R1:5} given by Lemma~\ref{L:Nconn} which 
only exists for $R_\mu\in (R_\mu^+(R,\eta),R_\mu^M(R,\eta)$. Owing to \eqref{puccini}, Lemma~\ref{L:Z1}, Lemma~\ref{P:NSC}, Lemma~\ref{L:2.12b}, 
and Lemma~\ref{L:Nconn}, this constraint on $R_\mu$ ensures that the only possibility 
for $(\overline{\gamma}_1,\overline{\beta}_1,\overline{\alpha}_1,\overline{\alpha},\overline{\beta},\overline{\gamma})$ is to   be 
$$
(\overline{\gamma}_1,\overline{\beta}_1,\overline{\alpha}_1,\overline{\alpha},\overline{\beta},\overline{\gamma}) 
= (-\underline{\gamma},-\underline{\beta},0,-\underline{\alpha}_1,- \underline{\beta}_1,-\underline{\gamma}_1)\ .
$$
Since the limits $(\underline{\gamma}_1,\underline{\beta}_1,\underline{\alpha}_1,\underline{\alpha},\underline{\beta},\underline{\gamma})$ 
and $(\overline{\gamma}_1,\overline{\beta}_1,\overline{\alpha}_1,\overline{\alpha},\overline{\beta},\overline{\gamma})$
are uniquely determined, we extend $\varphi$ by continuity to $[\underline{\alpha}_1,\overline{\alpha}_1]$ by 
$$
\varphi(\underline{\alpha}_1) := (\underline{\gamma}_1,\underline{\beta}_1,\underline{\alpha}_1,0,\underline{\beta},\underline{\gamma}) \;\text{ and }
\; \varphi(\overline{\alpha}_1) := (-\underline{\gamma},-\underline{\beta},0,-\underline{\alpha}_1,-\underline{\beta}_1,-\underline{\gamma}_1)\ .
$$
This last step completes the proof of Proposition~\ref{P:2}~$(a)$-$(b)$.

\medskip

We next turn to the proof of the property~$(c)$ stated in Proposition~\ref{P:2}. We first consider the case $R_\mu\ge R_\mu^M(R,\eta)$. 
Then $\varphi(\underline{\alpha}_1)$ is given by Lemma~\ref{P:NSC} with $\varphi_1(\underline{\alpha}_1) = \varphi_2(\underline{\alpha}_1) = \underline{\alpha}_1$ 
and satisfies $- \underline{\alpha}_1> \varphi_4(\underline{\alpha}_1)$. 
Equivalently, $(\varphi_3+\varphi_4)(\underline{\alpha}_1)<0$ which implies, together with Proposition~\ref{P:2}~$(b)$ that 
$(\varphi_3+\varphi_4)(\overline{\alpha}_1) = - (\varphi_3+\varphi_4)(\underline{\alpha}_1) >0$. 
Owing to the continuity of $\varphi_3+\varphi_4$, we readily conclude that there is $\alpha_1'\in (\underline{\alpha}_1,\overline{\alpha}_1)$ 
such that $(\varphi_3+\varphi_4)(\alpha_1')=0$. 
Now, $\varphi(\alpha_1')$ is a solution of \eqref{R1:1}-\eqref{R1:5} satisfying \eqref{clementi} and it follows from Lemma~\ref{L:2.12b} 
that $(\varphi_1+\varphi_6)(\alpha_1')=(\varphi_2+\varphi_5)(\alpha_1')=0$ and $(-\alpha_1', \varphi_5(\alpha_1'),\varphi_6(\alpha_1'))$ solves \eqref{eq:41}-\eqref{eq:43}. 
Consequently, $-\alpha_1'=\alpha_*$ by Lemma~\ref{L:Z1} and $\varphi(-\alpha_*)$ is given by Proposition~\ref{P:2}~$(c)$. 

Consider next the case $R_\mu\in (R_\mu^+(R,\eta) , R_\mu^M(R,\eta))$. Then $\varphi(\underline{\alpha}_1)$ is 
given by Lemma~\ref{L:Nconn} with $\underline{\alpha}_1 < 0 = - \varphi_4(\underline{\alpha}_1)$. 
Therefore, $(\varphi_3+\varphi_4)(\underline{\alpha}_1)<0$ and we deduce from Proposition~\ref{P:2}~$(b)$ 
that $(\varphi_3+\varphi_4)(\overline{\alpha}_1) = - (\varphi_3+\varphi_4)(\underline{\alpha}_1) >0$. 
We then proceed as in the case $R_\mu\ge R_\mu^M(R,\eta)$ to complete the proof of Proposition~\ref{P:2}~$(c)$. 
\end{proof}

A direct consequence of Lemma~\ref{P:NSC}, Lemma~\ref{L:Nconn}, and Proposition~\ref{P:2} is the following non-existence result.

%%%%%%%%%%%%%%%%%%%%%%%%%%%%%%%%%%%%%%%%
\begin{lemma}\label{L:NON}
 If $R_\mu\in(R+1,  R_\mu^+(R,\eta)],$ then there is no solution $(\gamma_1,\beta_1,\alpha_1,\alpha,\beta,\gamma)$ of the system \eqref{R1:1}-\eqref{R1:5} satisfying 
\begin{equation}
\gamma_1<\beta_1<\alpha_1\leq 0 \leq \alpha< \beta< \gamma \;\;\text{ and }\;\; \alpha_1\neq-\alpha\ .\label{ravel}
\end{equation}
\end{lemma}
%%%%%%%%%%%%%%%%%%%%%%%%%%%%%%%%%%%%%%%%

\begin{proof}
Let $R_\mu\in (R+1,R_\mu^+(R,\eta)]$ and assume for contradiction that there exists a solution $(\gamma_1^*,\beta_1^*,\alpha_1^*,\alpha^*,\beta^*,\gamma^*)$ 
of \eqref{R1:1}-\eqref{R1:5} satisfying \eqref{ravel}. 
As $(-\gamma^*,-\beta^*,-\alpha^*,-\alpha_1^*,-\beta_1^*,-\gamma_1^*)$ is also a solution of \eqref{R1:1}-\eqref{R1:5} and $\alpha_1$ and $\alpha$ 
do not vanish simultaneously, we infer from Lemma~\ref{L:Nconn} that $\alpha_1<0<\alpha.$ Consequently, $(\gamma_1^*,\beta_1^*,\alpha_1^*,\alpha^*,\beta^*,\gamma^*)$ 
is a solution of \eqref{R1:1}-\eqref{R1:5} satisfying \eqref{clementi}. By Proposition~\ref{P:2} this solution belongs to a continuous curve of 
solutions of \eqref{R1:1}-\eqref{R1:5} with one end being a solution of \eqref{R1:1}-\eqref{R1:5} given by Lemma~\ref{P:NSC} or Lemma~\ref{L:Nconn}. 
Since such solutions only exist for $R_\mu>R_\mu^+(R,\eta)$, we obtain a contradiction and complete the proof. 
\end{proof}
  
Now, the outcome of the above analysis enables us to provide a complete picture of the non-symmetric self-similar profiles.

%%%%%%%%%%%%%%%%%%%%%%%%%%%%%%%%%%%%%%%%
\begin{prop}\label{P:6} Let $(R,R_\mu,\eta)$ be three given positive real numbers.

\begin{itemize}

\item[$(a)$] If $R_\mu>R_\mu^+(R,\eta)$, there are a bounded interval $\Lambda := [\ell_-,\ell_+]$ and  a function $\zeta := (\gamma_1,\beta_1,\alpha_1,\alpha,\beta,\gamma)\in C(\Lambda,\R^6)$ that determines a one-parameter family $(F_\ell,G_\ell)_{\ell\in\Lambda}$ of solutions of \eqref{ststa} such that

\begin{itemize}
\item[$(a1)$] The function $\zeta$ is real-analytic in $(\ell_-,\ell_+)$ and $\ell_-< 0 < \ell_+$. 

\item[$(a2)$] The pair $(F_0,G_0)$ is the unique even solution of \eqref{ststa} given by Proposition~\ref{P:-1}. 

\item[$(a3)$] For each $\ell\in (\ell_-,\ell_+)\setminus\{0\}$ the sextuplet $\zeta(\ell)$ is a solution of \eqref{R1:1}-\eqref{R1:5} satisfying \eqref{clementi} and the pair $(F_\ell,G_\ell)$ is the non-symmetric solution of \eqref{ststa} given by Proposition~\ref{P:1} corresponding to the parameters $\zeta(\ell)$.

\item[$(a4)$] If $R_\mu\geq R_\mu^M(R,\eta),$ then $\gamma_1(\ell_-)=\beta_1(\ell_-)=\alpha_1(\ell_-)$, the quadruplet 
$$
(\beta_1(\ell_-), \alpha(\ell_-), \beta(\ell_-),\gamma(\ell_-))$$ 
is the solution of \eqref{d1:V1}-\eqref{d1:V4} given by Lemma~\ref{P:NSC}, and 
$$
\zeta(\ell_+) = (-\gamma(\ell_-),-\beta(\ell_-),-\alpha(\ell_-),-\alpha_1(\ell_-),-\beta_1(\ell_-),-\gamma_1(\ell_-))\ .
$$
The pair $(F_{\ell_\pm},G_{\ell_\pm})$ is the non-symmetric solution of \eqref{ststa} given by Proposition~\ref{P:0} corresponding to the
parameters $(\beta_1(\ell_\pm),\alpha(\ell_\pm),\beta(\ell_\pm),\gamma(\ell_\pm))$.

\item[$(a5)$] If $R_\mu\in(R_\mu^+(R,\eta),R_\mu^M(R,\eta)),$ then $\zeta(\ell_-)$ is the solution of \eqref{R1:1}-\eqref{R1:5} given by 
Lemma~\ref{L:Nconn} satisfying $\alpha(\ell_-)=0$ and 
$$
\zeta(\ell_+) = (-\gamma(\ell_-),-\beta(\ell_-),-\alpha(\ell_-),-\alpha_1(\ell_-),-\beta_1(\ell_-),-\gamma_1(\ell_-))\ .
$$
The pair $(F_{\ell_\pm},G_{\ell_\pm})$ is the non-symmetric solution of \eqref{ststa} given by Proposition~\ref{P:1} corresponding to the parameters $\zeta(\ell_\pm)$.

\item[$(a6)$] There is no other non-symmetric solution of \eqref{ststa}.

\end{itemize}

\item[$(b)$]  If $R_\mu\in (0,R_\mu^-(R,\eta))$, there are a bounded interval $\Lambda := [\ell_-,\ell_+]$ and a function  $\zeta := (\gamma_1,\beta_1,\alpha_1,\alpha,\beta,\gamma)\in C(\Lambda,\R^6)$ that determines a one-parameter family $(F_\ell,G_\ell)_{\ell\in\Lambda}$ 
of solutions of \eqref{ststa} such that 

\begin{itemize}
\item[$(b1)$] The function $\zeta$ is real-analytic in $(\ell_-,\ell_+)$ and $\ell_-< 0 < \ell_+$. 

\item[$(b2)$] The pair $(F_0,G_0)$ is the unique even solution of \eqref{ststa} given by Proposition~\ref{P:-1}. 

\item[$(b3)$] For each $\ell\in (\ell_-,\ell_+)\setminus\{0\}$ the sextuplet $\zeta(\ell)$ is a solution of \eqref{R2:1}-\eqref{R2:5} satisfying \eqref{clementi} and the pair $(F_\ell,G_\ell)$ is the non-symmetric solution of \eqref{ststa} given by Proposition~\ref{P:1} corresponding to the parameters $\zeta(\ell)$.

\item[$(b4)$] If $R_\mu\in (0, R_\mu^m(R,\eta)],$ then $\gamma_1(\ell_-)=\beta_1(\ell_-)=\alpha_1(\ell_-)$, 
the quadruplet 
$$
(\beta_1(\ell_-), \alpha(\ell_-), \beta(\ell_-),\gamma(\ell_-))$$ 
is the solution of \eqref{d2:V1}-\eqref{d2:V4} associated to that of \eqref{d1:V1}-\eqref{d1:V4} given by Lemma~\ref{P:NSC} by the transformation described in Lemma~\ref{L:2.4c}~$(ii)$, and 
$$
\zeta(\ell_+) = (-\gamma(\ell_-),-\beta(\ell_-),-\alpha(\ell_-),-\alpha_1(\ell_-),-\beta_1(\ell_-),-\gamma_1(\ell_-))\ .
$$
The pair $(F_{\ell_\pm},G_{\ell_\pm})$ is the non-symmetric solution of \eqref{ststa} given by Proposition~\ref{P:0} corresponding to the
parameters $(\beta_1(\ell_\pm),\alpha(\ell_\pm),\beta(\ell_\pm),\gamma(\ell_\pm))$.

\item[$(b5)$] If $R_\mu\in (R_\mu^m(R,\eta),R_\mu^-(R,\eta)),$ then $\zeta(\ell_-)$ is the solution of \eqref{R2:1}-\eqref{R2:5}  satisfying $\alpha(\ell_-)=0$ associated 
that of \eqref{R1:1}-\eqref{R1:5} 
given by Lemma~\ref{L:Nconn} by the transformation described in Lemma~\ref{L:2.4c}~$(ii)$ and 
$$
\zeta(\ell_+) = (-\gamma(\ell_-),-\beta(\ell_-),-\alpha(\ell_-),-\alpha_1(\ell_-),-\beta_1(\ell_-),-\gamma_1(\ell_-))\ .
$$
The pair $(F_{\ell_\pm},G_{\ell_\pm})$ is the non-symmetric solution of \eqref{ststa} given by Proposition~\ref{P:1} corresponding to the parameters $\zeta(\ell_\pm)$.

\item[$(b6)$] There is no other non-symmetric solution of \eqref{ststa}.

\end{itemize}

\end{itemize}
\end{prop}
%%%%%%%%%%%%%%%%%%%%%%%%%%%%%%%%%%%%%%%%

\begin{proof} 
\textbf{Case~1: $R_\mu>R_\mu^+(R,\eta)$.} Let $(\alpha_*,\beta_*,\gamma_*)$ be the unique solution of \eqref{eq:41}-\eqref{eq:43} given by Lemma~\ref{L:Z1}.
Observing that  $(-\gamma_*, -\beta_*, -\alpha_*, \alpha_*, \beta_*, \gamma_*)$ solves \eqref{R1:1}-\eqref{R1:5} and satisfies \eqref{clementi} for $R_\mu>R_\mu^+(R,\eta)$ 
we infer 
from Proposition~\ref{P:2} that there are $\overline{\alpha}_1\in (-\alpha_*,0]$, $\underline{\alpha}_1\in (-\infty,-\alpha_*)$, and a bounded
continuous function $\varphi\in C([\underline{\alpha}_1,\overline{\alpha}_1],\R^6)$ such that $\varphi$ is real-analytic in $(\underline{\alpha}_1,\overline{\alpha}_1)$ 
with $\varphi_3 = \mathrm{id}$ 
and satisfies the properties~$(a)$ and~$(b)$ of Proposition~\ref{P:2}. Setting $\ell_-:= \alpha_*+\underline{\alpha}_1<0$, $\ell_+:= \alpha_*+\overline{\alpha}_1>0$, 
and $\zeta(\ell) := \varphi(\ell-\alpha_*)$ for $\ell\in \Lambda := [\ell_-,\ell_+]$,
the statements~(a1)-(a5) of Proposition~\ref{P:6} are straightforward consequences of Proposition~\ref{P:2}.

\medskip

In order to prove $(a6)$, assume for contradiction that there exists a non-symmetric self-similar profile solving \eqref{ststa} 
with one of its components having a disconnected support which does not lie on the curve $\varphi([\underline{\alpha}_1,\overline{\alpha}_1])$ constructed above. 
In view of Proposition~\ref{P:1}, this solution corresponds to a 
sextuplet $(\gamma_1',\beta_1',\alpha_1',\alpha',\beta',\gamma')$ solving  \eqref{R1:1}-\eqref{R1:5} and satisfying \eqref{haydn}. 
The possibility that $\alpha_1'=0$ or $\alpha'=0$ is actually excluded since it 
corresponds to the solutions of \eqref{R1:1}-\eqref{R1:5} described in Lemma~\ref{L:Nconn} which
are already on the curve $\varphi([\underline{\alpha}_1,\overline{\alpha}_1])$. 
Consequently, $(\gamma_1',\beta_1',\alpha_1',\alpha',\beta',\gamma')$ also satisfies \eqref{clementi} and we infer
from Proposition~\ref{P:2} that there is a function $\psi:=(\psi_k)_{\{1\le k \le 6\}} \in C([\underline{\alpha}_1,\overline{\alpha}_1],\mathbb{R}^6)$
which is real-analytic in $(\underline{\alpha}_1,\overline{\alpha}_1)$ 
such that $\psi_3\equiv \mathrm{id}$ and $\psi(\alpha_1)$ is a solution to \eqref{R1:1}-\eqref{R1:5} satisfying \eqref{clementi} for every 
$\alpha_1\in (\underline{\alpha}_1,\overline{\alpha}_1)$.
We emphasize here that $\varphi$ and $\psi$ are defined on the same interval as their end points are uniquely identified in Proposition~\ref{P:2}~$(b)$ 
and this also implies that 
$$
\varphi(\underline{\alpha}_1)=\psi(\underline{\alpha}_1) \;\;\text{ and }\;\;\varphi(\overline{\alpha}_1)=\psi(\overline{\alpha}_1)\ .
$$ 
We further infer from Proposition~\ref{P:2}~$(c)$ that $\psi(-\alpha_*)=\varphi(-\alpha_*)$ and the local uniqueness 
stemming from the implicit function theorem implies that $\psi$ and $\varphi$ coincide in a neighborhood of $-\alpha_*$. 
Being both real-analytic they actually coincide on $[\underline{\alpha_1},\overline{\alpha_1}]$.

\bigskip

\noindent\textbf{Case~2: $R_\mu\in (0,R_\mu^-(R,\eta))$}. This case can actually be deduced from the previous one, thanks to the transformation 
described in Lemma~\ref{L:2.4c}~$(ii)$. 
Recall in particular that the parameter $R_{\mu,1} = R(1+R)/R_\mu$ defined in \eqref{spirou} ranges in $(R_\mu^+(R,\eta_1),\infty)$ when $R_\mu\in (0,R_\mu^-(R,\eta))$.
\end{proof}

%%%%%%%%%%%%%%%%%%%%%%%%%%%%%%%%%%%%%%%%
\begin{lemma}\label{L:MI}
Let $R$, $R_\mu$, and $\eta$ be given positive real numbers. 
If $R_\mu>R_\mu^+(R,\eta)$ or $R_\mu\in (0,R_\mu^-(R,\eta))$, then $\ell\mapsto \mathcal{E}_*(F_\ell,G_\ell)$ is decreasing on $[\ell_-,0]$ and 
increasing on $[0,\ell_+]$ where $(F_\ell,G_\ell)_{\ell\in\Lambda}$ is the curve of solutions of \eqref{ststa} described in Proposition~\ref{P:6}.
\end{lemma}
%%%%%%%%%%%%%%%%%%%%%%%%%%%%%%%%%%%%%%%%

\begin{proof} 
\textbf{Case~1: $R_\mu>R_\mu^+(R,\eta)$.} 
We keep the notation introduced in Proposition~\ref{P:6} and its proof. We infer from Theorem~\ref{T:WP}~$(v)$ that
$\E_*(F_\ell,G_\ell)=\cM_2(F_\ell,G_\ell)/2$ for all $\ell\in \Lambda$. 
Therefore, using the explicit formula found in Proposition~\ref{P:1}~$(i)$, we have that $\xi(\ell):= \E_*(F_{\ell},G_\ell)$ satisfies
\begin{align}\label{xi}
\xi&=\frac{1}{90\eta^2R_\mu^2}\left[R(\gamma^5-\gamma_1^5)+(R_\mu-R)^2(\beta^5-\beta_1^5)-\frac{R(R_\mu-R-1)^2}{1+R}(\alpha^5-\alpha_1^5)\right]
\end{align}
where $(\gamma_1,\beta_1,\alpha_1,\alpha,\beta,\gamma)(\ell)$ is the corresponding solution of \eqref{R1:1}-\eqref{R1:5}.
In order to study the sign of $\xi'$, we need to determine the derivative of $\zeta=(\gamma_1,\beta_1,\alpha_1,\alpha,\beta,\gamma).$ 
It follows from the equations \eqref{R1:1}-\eqref{R1:5}, after rather lengthy computations, that
\begin{align*}
\gamma_1\gamma_1'=&\frac{R_\mu-R-1}{1+R}\frac{R(\beta-\alpha)(\gamma-\alpha_1)+R(\gamma-\alpha)(\alpha_1-\beta_1)+(\beta-\beta_1)(\gamma-\alpha)}{R(\gamma_1-\gamma)(\beta-\alpha)+(\beta_1-\beta)(\gamma-\alpha)}\alpha_1>0,\\
\beta_1\beta_1'=&\frac{R(R_\mu-R-1)}{(1+R)(R_\mu-R)}\frac{R(\gamma_1-\gamma)(\beta-\alpha)+(\gamma-\alpha)(\alpha_1-\beta)+(\beta-\alpha)(\gamma_1-\alpha_1)}{R(\gamma_1-\gamma)(\beta-\alpha)+(\beta_1-\beta)(\gamma-\alpha)}\alpha_1<0,\\
   \alpha\alpha'= &\frac{ R(\gamma_1-\gamma)(\beta_1-\alpha_1)+(\beta_1-\beta)(\gamma_1-\alpha_1)}{R(\gamma_1-\gamma)(\beta-\alpha)+(\beta_1-\beta)(\gamma-\alpha)}\alpha_1>0,\\
\beta\beta'=&\frac{R(R_\mu-R-1)}{(1+R)(R_\mu-R)}\frac{R(\gamma_1-\gamma)(\beta_1-\alpha_1)+(\gamma-\alpha)(\alpha_1-\beta_1)+(\gamma_1-\alpha_1)(\beta_1-\alpha)}{R(\gamma_1-\gamma)(\beta-\alpha)+(\beta_1-\beta)(\gamma-\alpha)}\alpha_1>0,\\
    \gamma\gamma'=& \frac{R_\mu-R-1}{1+R}\frac{R(\gamma_1-\alpha)(\alpha_1-\beta_1)+R(\gamma_1-\alpha_1)(\beta-\alpha)+(\beta-\beta_1)(\gamma_1-\alpha_1)}{R(\gamma_1-\gamma)(\beta-\alpha)+(\beta_1-\beta)(\gamma-\alpha)}\alpha_1<0.
\end{align*}
Recalling \eqref{clementi}, we realize that $\alpha'>0$, $\beta'>0$, $\beta_1'>0$, $\gamma'<0$, and $\gamma_1'<0$ in $(\underline{\alpha_1}, \overline{\alpha_1}).$
Differentiating \eqref{xi} and making use of the previous formulas, we deduce that
\begin{align}\label{dbog1}
 \xi'=\frac{1}{18\eta^2R_\mu^2}\frac{R(R_\mu-R-1)}{(1+R)}\frac{\alpha_1  }{R(\gamma_1-\gamma)(\beta-\alpha)+(\beta_1-\beta)(\gamma-\alpha)}(T_1+RT_2)
\end{align}
where, making use of \eqref{R1:1}-\eqref{R1:3}, the terms $T_1$ and $T_2$ may be expressed, after rather lengthy algebraic manipulations, as follows:
\begin{align*}
T_1  :=& R(\gamma_1^2-\gamma^2)\Big\{ \beta(\gamma-\alpha)(\gamma_1-\beta_1)-\beta_1(\gamma_1-\alpha_1)(\gamma-\beta) \nonumber\\
&\ +\frac{1+R}{R}(\beta_1-\beta)(\gamma_1-\alpha_1)(\gamma-\alpha)\Big\}\ , \\
T_2  :=& R(\gamma_1^2-\gamma^2)\Big\{ \gamma(\beta-\alpha)(\gamma_1-\beta_1)-\gamma_1(\beta_1-\alpha_1)(\gamma-\beta) \nonumber\\
&\ + \frac{1+R}{R}\left[ \gamma_1(\beta_1-\alpha_1)(\gamma-\alpha)-\gamma(\beta-\alpha)(\gamma_1-\alpha_1) \right]\Big\}\ .
\end{align*}
Thanks to \eqref{clementi}, the expressions in the curly brackets of both $T_1$ and $T_2$ are positive. We finally note that, since
$$
\left( \gamma_1^2 - \gamma^2 \right)' = 2 \gamma_1 \gamma_1' - 2 \gamma \gamma' >0
$$
by the explicit formulas above and $\left( \gamma_1^2 - \gamma^2 \right)(0)=0$ due to the evenness of $(F_0,G_0)$, $\gamma_1^2-\gamma^2$ is an increasing 
function on $\Lambda$ which vanishes at zero. Inserting these information in \eqref{dbog1} completes the proof.

\medskip

\noindent\textbf{Case~2: $R_\mu\in (0,R_\mu^-(R,\eta))$.} We use once more the transformation found in  Lemma~\ref{L:2.4c}~$(ii)$ 
to deduce this case from the previous one.
\end{proof}

Collecting all the results established in this section allows us to complete the proof of Theorem~\ref{T:MT1}.

\begin{proof}[Proof of Theorem~\ref{T:MT1}]
The statements $(i)$-$(iii)$ follow from Propositions~\ref{P:-1} and~\ref{P:6}, the unimodal
structure of the rescaled energy $\mathcal{E}_*$ stated in $(iv)$ being a consequence of
Lemma~\ref{L:MI}. The properties $(v)$-$(vii)$ of the supports of steady-state solutions of \eqref{RS} are obtained by
combining the outcome of Propositions~\ref{P:0}, \ref{P:1}, and \ref{P:6} .
\end{proof}

%%%%%%%%%%%%%%%%%%%%%%%%%%%%%%%%%%%%%%%%
%%%%%%%%%%%%%%%%%%%%%%%%%%%%%%%%%%%%%%%%
\subsection{Variational characterization of even self-similar profiles}\label{S:22var}
%%%%%%%%%%%%%%%%%%%%%%%%%%%%%%%%%%%%%%%%
%%%%%%%%%%%%%%%%%%%%%%%%%%%%%%%%%%%%%%%%

We conclude the analysis of self-similar profiles to \eqref{Pb} by a variational characterization of the even self-similar profiles we constructed in Proposition~\ref{P:-1}. 
More precisely, each of them is the unique minimizer of the scaled energy functional $\E_{*} $ defined in \eqref{eq:sup3}.

%%%%%%%%%%%%%%%%%%%%%%%%%%%%%%%%%%%%%%%%
\begin{prop}\label{P1} Given positive parameters $(R, R_\mu,\eta),$ there exists a unique minimizer $(F,G)$ of the functional $\E_{*}$ within $\cK^2$. 
Additionally, both functions $F$ and $G$ are even, belong to $H^1(\R)$, and solve the system \eqref{ststa} (and are explicitly computed in Proposition~\ref{P:-1}).
\end{prop}
%%%%%%%%%%%%%%%%%%%%%%%%%%%%%%%%%%%%%%%%

\begin{proof} Since $\E_{*}$ is a non-negative and strictly convex functional on the convex set $\cK^2$, the existence and 
uniqueness of a unique minimizer $(F,G)\in\cK^2$ of $\E_{*}$ in $\cK^2$ can be established by classical variational arguments, 
following for instance the lines of the proof of \cite[Lemma~2.1]{LM12x}. Moreover, the uniqueness of the minimizer and the 
rotational invariance of the functional $\E_*$ imply that both functions $F$ and $G$ are even. 
The $H^1$-regularity can then be proved as in \cite[Lemma~2.1]{LM12x} with the help of a technique developped in \cite{MMS09}.
We finally argue as in \cite[Lemma~2.2]{LM12x} to establish that $(F,G)$ solves \eqref{ststa} and complete the proof.
\end{proof}

%%%%%%%%%%%%%%%%%%%%%%%%%%%%%%%%%%%%%%%%
%%%%%%%%%%%%%%%%%%%%%%%%%%%%%%%%%%%%%%%%
\section{Asymptotic behavior in self-similar variables}\label{S:3}
%%%%%%%%%%%%%%%%%%%%%%%%%%%%%%%%%%%%%%%%
%%%%%%%%%%%%%%%%%%%%%%%%%%%%%%%%%%%%%%%%

This section is dedicated to the study of the asymptotic behavior of weak solutions to \eqref{RS}, as defined in Theorem~\ref{T:WP}.

%%%%%%%%%%%%%%%%%%%%%%%%%%%%%%%%%%%%%%%%
%%%%%%%%%%%%%%%%%%%%%%%%%%%%%%%%%%%%%%%%
\subsection{Weak solutions}\label{S:31}
%%%%%%%%%%%%%%%%%%%%%%%%%%%%%%%%%%%%%%%%
%%%%%%%%%%%%%%%%%%%%%%%%%%%%%%%%%%%%%%%%

Since  the change of variables induces a one-to-one correspondence  between the set of  weak solutions of \eqref{Pb} and that of the system
\eqref{RS}, we obtain  from \cite{LM12x} the following existence result.

%%%%%%%%%%%%%%%%%%%%%%%%%%%%%%%%%%%%%%%%
\begin{thm}\label{T:WP} Let  $R,$ $R_\mu,$ and $\eta$ be given positive constants. Given $(f_0,g_0)\in\cK^2,$ there exists at least a pair
$(f,g):[0,\infty)\to  \cK^2$ such that
\begin{itemize}
\item[$(i)$] $(f,g)\in L_\infty(0,\infty; \cK^2)$, $ (f,g) \in L_2(0,t;H^1(\R;\R^2))$,
\item[$(ii)$] $(f,g)\in C ([0,\infty);H^{-3}(\R;\R^2))$ with $(f,g)(0)=(f_0,g_0),$
\end{itemize}
and $(f,g)$ is a weak solution of the rescaled system \eqref{RS} in the sense that 
\begin{align}\label{T2a}   
&\int_\R f(t) \xi\, dx-\int_\R f_0 \xi\, dx + \int_0^t\int_\R f \p_x\left( \eta^2(1+R) f+R g+\frac{x^2}{6}\right) \p_x\xi\, dx\, d\sigma=0\, ,\\[1ex]
&\int_\R g(t) \xi\, dx - \int_\R g_0 \xi\, dx+\ \int_0^t \int_\R g\p_x\left( \eta^2R_\mu f+   R_\mu g+\frac{x^2}{6} \right) \p_x\xi\, dx\, d\sigma=0\,,\label{T2b}  
\end{align}
for all $\xi\in C_0^\infty(\R) $ and all $t\ge 0.$
 In addition, $(f,g)$ satisfies the following estimates:
\begin{align*}
(iii) \quad & \mathcal{H}(f(t),g(t)) +\int_s^t \left( \eta^2\|\p_x f\|^2_2 + R \eta^{-2}\|\eta^2\p_x f+\p_x g\|_2^2 \right)  \, d\sigma\\
&\hspace{2.25cm}\leq  \mathcal{H}(f(s),g(s))+\frac{1+\Theta}{3}(t-s),\\[1ex]
(iv)\quad &\E_{*}(f(t), g(t))
+ \frac{1}{2}\int_{s}^t\int_\R f\left( \eta^2(1+R)\p_x f+R\p_x g+\frac{x }{3}\right)^2 \, dx\,d\sigma\\
&\hspace{2.3cm}+\frac{\Theta}{2} \int_{s}^t \int_\R g\left( \eta^2 R_\mu \p_x f+ R_\mu\p_x g + \frac{x}{3}\right)^2   dx\,d\sigma \leq \E_{*}(f(s),g(s)),\\
(v)\quad &\cM_{2}(f(t), g(t))+\int_{s}^t\cM_{2}(f(\sigma),g(\sigma))\,d\sigma =\cM_{2}(f(s),g(s)) + 2 \int_{s}^t \E_{*}(f(\sigma),g(\sigma)) \,d\sigma
\end{align*}
for all $s\in [0,\infty)\setminus \mathcal{N}$ and $t\in [0,\infty)$ with $s\leq t$, $\mathcal{N}$ being a measurable subset of $(0,\infty)$ with Lebesgue measure zero.
The functionals $\E_*$, $\cM_2$, and $\mathcal{H}$ are defined by \eqref{eq:sup3}, \eqref{eq:sup1c}, and
\begin{equation}
\mathcal{H}(u,v) := \int_\R \left( u \ln{u} + \Theta v \ln{v} \right)\, dx\, ,
\label{eq:sup1b}
\end{equation}
respectively.

Furthermore, if $f_0$ and $g_0$ are even, then $f(t)$ and $g(t)$ are even functions for all $ t\in(0,\infty).$
\end{thm}
%%%%%%%%%%%%%%%%%%%%%%%%%%%%%%%%%%%%%%%%

A classical consequence of Theorem~\ref{T:WP} is that $(f,g)$ solves \eqref{RS} in the sense of distributions, that is, 
\begin{align}
 &\p_tf=\p_x\left(f\p_x\left( \eta^2(1+R) f+R g+\frac{x^2}{6}\right)\right)=:\p_xJ_f\label{Beth1}\ ,\\
 &\p_tg=\p_x\left(g\p_x\left( \eta^2R_\mu f+   R_\mu g+\frac{x^2}{6} \right)\right)=:\p_xJ_g\label{Beth2}
\end{align}
in $\mathcal{D}'((0,\infty)\times\R).$

%%%%%%%%%%%%%%%%%%%%%%%%%%%%%%%%%%%%%%%%
\begin{rem}\label{Rem1}
It is easy to see  that the identity~$(v)$ is valid for all $0\le s\le t.$
\end{rem}
%%%%%%%%%%%%%%%%%%%%%%%%%%%%%%%%%%%%%%%%

\begin{proof}[Proof of Theorem~\ref{T:WP}]
Let $(\phi,\psi) $ denote the weak solution of \eqref{Pb} constructed in \cite[Theorem~1.1]{LM12x} by
using the gradient flow structure of \eqref{Pb} with respect to the $2$-Wasserstein distance in the space of probability measures with finite second moment. 
The function $(f,g)$ defined by the transformation \eqref{eq:Cha2} is then a weak solution of \eqref{RS} and all the properties stated in Theorem~\ref{T:WP} readily 
follow from 
those enjoyed by $(\phi,\psi)$ which are established in \cite{LM12x}.
Moreover, the time evolution~$(v)$ of the second moment is derived from \eqref{T2a}, \eqref{T2b}, and the estimate~$(iv)$ 
by choosing a suitable approximating sequence $(\xi_n)_n\subset C_0^\infty(\R)$ for the function $x\mapsto x^2.$

Finally, if $f_0 $ and $g_0$ are even, then the  solution $(\phi,\psi)$ of \eqref{Pb}  constructed in \cite{LM12x} has the property that both $\phi(t)$ and
$\psi(t)$ are even for almost all $t\geq 0,$ and, by \eqref{eq:Cha2} and the continuity property established in Theorem~\ref{T:WP}~$(ii)$, $f(t)$ and $g(t)$ also enjoy this property for all $t>0$. 
\end{proof}

%%%%%%%%%%%%%%%%%%%%%%%%%%%%%%%%%%%%%%%%
%%%%%%%%%%%%%%%%%%%%%%%%%%%%%%%%%%%%%%%%
\subsection{Convergence}\label{S:32}
%%%%%%%%%%%%%%%%%%%%%%%%%%%%%%%%%%%%%%%%
%%%%%%%%%%%%%%%%%%%%%%%%%%%%%%%%%%%%%%%%

In order to prove Theorem~\ref{T:MT3} we exploit the estimates recalled for weak solutions $(f,g)$ of \eqref{RS} in Theorem~\ref{T:WP} to identify the cluster points of $(f(t),g(t))_{t\ge 0}$ as $t\to\infty$ for the weak topology of $L_2(\R,\R^2)$. 
More precisely, given $(f_0,g_0)\in\mathcal{K}^2$ and a 
weak solution $(F,G)$ to \eqref{RS} as in Theorem~\ref{T:WP}, we define the $\omega$-limit set $\omega(f_0,g_0)$ for the weak topology of $L_2(\R,\R^2)$ as follows:

\begin{equation}
\omega(f_0,g_0) := \left\{(F_\infty,G_\infty) \,:\,\begin{minipage}{9cm} 
$(F_\infty,G_\infty) \in \mathcal{K}^2$ and there exists a sequence $(t_n)_{n\ge 1}$ of positive 
real numbers satisfying $t_n\to \infty$ and $(f(t_n),g(t_n))\rightharpoonup (F_\infty,G_\infty)$ in $L_2(\R,\R^2)$ as $n\to \infty$
\end{minipage}\right\}. \label{ol}
\end{equation}

%%%%%%%%%%%%%%%%%%%%%%%%%%%%%%%%%%%%%%%%
\begin{prop}\label{prpl1}
The $\omega$-limit set $\omega(f_0,g_0)$ is non-empty and bounded in $H^1(\R,\R^2)$ and in $L_1(\R, (1+x^2) dx,\R^2)\big)$ and is connected in $H^{-3}(\R,\R^2)$. 
In addition, if $(F_\infty,G_\infty)\in \omega(f_0,g_0)$, then $(F_\infty,G_\infty)$ solves \eqref{ststa}, i.e. is a stationary solution of \eqref{RS}.
\end{prop}
%%%%%%%%%%%%%%%%%%%%%%%%%%%%%%%%%%%%%%%%

\begin{proof}
We first note that Theorem~\ref{T:WP}~ $(iv)$ and~$(v)$ guarantee that
\begin{align}
\mathcal{E}_*(f(t),g(t)) & \le \mathcal{E}_*(f_0,g_0)\ , \quad t\ge 0\ , \label{pl1} \\
\int_0^\infty \mathcal{I}(f(s),g(s)) ds & \le \mathcal{E}_*(f_0,g_0)\ , \label{pl2} \\
\mathcal{M}_2(f(t),g(t)) & \le \mathcal{M}_2(f_0,g_0) e^{-t} + 2 \mathcal{E}_*(f_0,g_0) (1-e^{-t})\ , \quad t\ge 0\ , \label{pl3}
\end{align}
where the entropy dissipation is defined in \eqref{EP}. 
We first deduce from \eqref{pl1} and \eqref{pl3} that the trajectory $\{(F(t),G(t))\ :\ t\ge 0\}$ is bounded in $\mathcal{K}^2$. 
The reflexivity of $L_2(\R,\R^2)$ and the Dunford-Pettis theorem then ensure that $\omega(f_0,g_0)$ is non-empty and bounded 
in $\big(L_1(\R, (1+x^2) dx,\R^2\big) \cap L_2(\R,\R^2)$. 
It further follows from \eqref{pl1}, \eqref{pl3}, and the classical bounds
\begin{align*}
\int_\R h(x)\ |\ln{h(x)}|\, dx & \le C + \int_\R h(x)\ \left( 1 + x^2 \right)\ dx + \|h\|_2^2\ ,  \\
\int_\R h(x)\ \ln{h(x)} & \ge - C - \int_\R h(x)\ \left( 1 + x^2 \right)\ dx\ , 
\end{align*}
see \cite[Lemma~A.1]{LM12x} for instance, that
$$
\left| \mathcal{H}(f(t),g(t)) \right| \le C\ , \quad  t\ge 0\ ,
$$
the constant $C$ being independent of $t$. Together with Theorem~\ref{T:WP}~$(iii)$, this gives
\begin{equation}
\int_{t-1}^{t+1} \left( \eta^2 \|\partial_x f(s)\|_2^2 + R \eta^{-2} \|(\eta^2 \partial_x f + \partial_x g)(s)\|_2^2 \right) ds \le C\ , \quad t\ge 1\ . \label{pl4}
\end{equation}

Consider now $(F_\infty,G_\infty)\in\omega(f_0,g_0)$ and a sequence $(t_n)_{n\ge 1}$ of positive real numbers such that 
\begin{equation}
t_n\to \infty \;\;\text{ and }\;\; (f(t_n),g(t_n)) \rightharpoonup (F_\infty,G_\infty) \;\;\text{ in }\;\; L_2(\R,\R^2)\ . \label{pl4b}
\end{equation}
Owing to \eqref{pl4b}, we may assume without loss of generality that $t_n>1$ for all $n\ge 1$ and we define functions $(f_n,g_n): (-1,1)\times\R \to \R^2$ by the relation
\begin{equation}
(f_n(s,x),g_n(s,x)) := (f(s+t_n,x),g(s+t_n,x))\ , \quad (s,x)\in (-1,1)\times\R\ . \label{pl5}
\end{equation} 
We infer from \eqref{pl1}-\eqref{pl4} that
\begin{align}
 \big( f_n,g_n \big)_n&\quad \text{is bounded in }\;\; L_\infty((-1,1);\cK^2)\ ,\label{pl6}\\
 \big( f_n,g_n \big)_n&\quad \text{is bounded in }\;\; L_2((-1,1);H^1(\R,\R^2))\ ,\label{pl7}
 \end{align}
 and
\begin{equation}
\lim_{n\to\infty} \int_{-1}^1 \cI(f_n(s),g_n(s))\ ds = \lim_{n\to\infty} \int_{-1+t_n}^{1+t_n} \cI(f(t),g(t))\ dt = 0\ . \label{pl8}
\end{equation}
Moreover, it follows from \eqref{Beth1} that $\p_t f_n=\p_x J_{f_n} $ in $\mathcal{D}'((-1,1)\times\R),$ whereby in virtue of \eqref{pl6}, \eqref{pl8},
and H\"older's inequality we have
\begin{align*}
 \|J_{f_n}\|_{L_2((-1,1); L_{4/3}(\R))}^2\leq
 &\int_{-1}^1 \left\| \sqrt{f_n(s)} \right\|_{4}^2\left\| \sqrt{f_n(s)} \left( \eta^2 (1+R) \partial_x f_n + R \partial_x g_n + \frac{x}{3} \right)(s) \right\|_2^2 \, ds\\
 \leq &C\int_{-1}^1 \cI(f_n(s),g_n(s))\ ds\ \to \ 0
\end{align*}
as $n\to\infty.$
Consequently $\p_t f_n=\p_x J_{f_n}\in L_2((-1,1); \big(W_4^1(\R)\big)')$ for all $n\ge 1$ and
\begin{equation}
\partial_t f_n \ \to \ 0 \;\;\text{ in }\;\; L_2((-1,1); \big(W_4^1(\R)\big)')\ . \label{pl9}
\end{equation}
Proceeding in a similar way, we deduce from \eqref{Beth2}, \eqref{pl6}, \eqref{pl8}, and H\"older's inequality that
\begin{equation}
\partial_t g_n \ \to \ 0 \;\;\text{ in }\;\; L_2((-1,1); \big(W_4^1(\R)\big)')\ . \label{pl10}
\end{equation}

We now infer from \eqref{pl4b}, \eqref{pl9}, and \eqref{pl10} that
\begin{equation}
(f_n,g_n) \ \to \ (F_\infty,G_\infty) \;\;\text{ in }\;\; C([-1,1]; \big(W_4^1(\R)\big)')\ . \label{pl11}
\end{equation}
Indeed, for $\xi\in W_4^1(\R)$,
\begin{align*}
\left| \int_\R (f_n(s) - F_\infty) \xi\ dx \right| = & \left| \int_0^s \int_\R \partial_t f_n(s) \xi\ dx \ ds \right| \\
\le & \int_0^s \|\partial_t f_n(s)\|_{(W_4^1(\R))'} \|\xi\|_{W_4^1}\ ds \\
\le & \|\xi\|_{W_4^1} \int_{-1}^1 \|\partial_t f_n(s)\|_{(W_4^1(\R))'} \ ds\ ,
\end{align*}
hence the claim. 
Furthermore, invoking \cite[Lemma 3.2]{LM12x}, the embedding $H^1(\R)\cap L_1(\R,(1+x^2)\, dx)$ in $L_2(\R)$ is compact and moreover the
embedding of $L_2(\R)$ in $\big(W_4^1(\R)\big)'$ is continuous. 
Thanks to \eqref{pl6}, \eqref{pl7}, \eqref{pl9}, and \eqref{pl10}, we are in a position to apply \cite[Corollary~4]{Si87} and use \eqref{pl11}
to conclude that there is a subsequence of $\big((f_n,g_n)\big)_n$  (not relabeled) such that
\begin{align}
 & (f_{n},g_{n}) \ \to \ (F_\infty, G_\infty) \;\;\text{ in }\;\; L_2((-1,1)\times\R,\R^2)\ , \label{pl12}\\
& (\p_x f_{n},\p_x g_{n}) \rightharpoonup (\p_x F_\infty,\p_x G_\infty) \;\;\text{ in}\;\; L_2((-1,1)\times \R,\R^2)\ .\label{pl13}
\end{align}
Consequently, owing to \eqref{pl4} and \eqref{pl13}, $(F_\infty,G_\infty)$ lies in a bounded subset of $H^1(\R,\R^2)$, 
which proves the boundedness of $\omega(f_0,g_0)$ in $H^1(\R,\R^2)$. Additionally, we deduce from \eqref{pl12} that there 
is (at least) a sequence $(s_n)_n$ such that $s_n\in (-1+t_n,1+t_n)\setminus\mathcal{N}$ for all $n\ge 1$ and
\begin{equation}
(f(s_n),g(s_n)) \ \to \ (F_\infty,G_\infty) \;\;\text{ in }\;\; L_2(\R,\R^2)\ . \label{pl13b}
\end{equation}
Finally, it follows from \eqref{pl6}, \eqref{pl12}, and \eqref{pl13} that
$$
\sqrt{f_n} \partial_x \left( \eta^2 (1+R) f_n + R g_n + \frac{x^2}{6} \right) 
\rightharpoonup \sqrt{F_\infty} \partial_x \left( \eta^2 (1+R) F_\infty + R G_\infty + \frac{x^2}{6} \right)
$$
in $L_1((-1,1)\times\R)$, while \eqref{pl8} guarantees that it converges strongly to zero in $L_2((-1,1)\times\R)$. Therefore, 
$$
\sqrt{F_\infty} \partial_x \left( \eta^2 (1+R) F_\infty + R G_\infty + \frac{x^2}{6} \right) = 0 \;\;\text{ a.e. in }\;\; \R\ .
$$
A similar argument ensures that
$$
\sqrt{G_\infty} \partial_x \left( R F_\infty + R \eta^{-2} G_\infty + \Theta \frac{x^2}{6} \right)  = 0 \;\;\text{ a.e. in }\;\; \R\ , 
$$
so that $(F_\infty,G_\infty)$ solves \eqref{ststa}.

Finally the fact that $\omega(f_0,g_0)$ is connected in $H^{-3}(\R)$ is a consequence of the time continuity of $f$ and $g$  in $H^{-3}(\R)$ and the 
compactness of $\omega(f_0,g_0)$ in $L_2(\R)$.
\end{proof}

%%%%%%%%%%%%%%%%%%%%%%%%%%%%%%%%%%%%%%%%
\begin{lemma}\label{lepl2} 
There exists $L> 0$ such that, for all $(F_\infty,G_\infty) \in \omega(f_0,g_0)$,
\begin{equation}
\mathcal{E}_*(F_\infty,G_\infty) = \frac{1}{2} \mathcal{M}_2(F_\infty,G_\infty) = L\ . \label{pl14}
\end{equation}
\end{lemma} 
%%%%%%%%%%%%%%%%%%%%%%%%%%%%%%%%%%%%%%%%

\begin{proof}
Since $t\mapsto \E_*(f(t),g(t))$ is a positive function which is non-increasing on $[0,\infty)\setminus\mathcal{N}$,  it follows from Theorem~\ref{T:WP}~$(iv)$ that there exists a constant $L\geq 0$ such  that 
\begin{equation}
\E_*(f(t),g(t))\ \searrow\  L\qquad\text{ as $t\nearrow\infty$, $ t\notin \mathcal{N}$}\ .\label{pl15}
\end{equation}

Defining the function $m_2(t):=\cM_{2}(f(t), g(t))-2L$ for $t\geq 0,$ we deduce from the assertions~$(iv)$ and~$(v)$ of Theorem~\ref{T:WP} that $m_2$ is differentiable almost everywhere in $(0,\infty)$, with
$$
\frac{dm_2}{dt} + m_2 = 2(\E_*(f,g)-L) \;\;\text{ a.e. in }\;\; (0,\infty)\ .
$$
Consequently, $m_2$ is a non-negative function and 
\begin{align*}
m_2(t) = & m_2(0) e^{-t} + 2\int_0^\tau (\E_*(f(s),g(s))-L) e^{s-t}\, ds + 2 \int_\tau^t (\E_*(f(s),g(s))-L) e^{s-t}\, ds
\end{align*}
for all $0<\tau<t.$ Given $\e>0$, we infer from \eqref{pl15} that there is $t_\e>0$ such that $L \le \E_*(f(s),g(s))<L+\e$ for every $s\geq t_\e$  with $s\not\in\mathcal{N}$. Taking $t>t_\e$ and $\tau=t_\e$ in the above identity and using \eqref{pl1}, we obtain
$$
0 \leq m_2(t) \le m_2(0) e^{-t} + 2(\E_*(f_0,g_0)-L)  \left( e^{t_\e} - 1 \right) e^{-t} + 2 \e \left( 1 - e^{t_\e-t} \right)\ .
$$
Letting first $t\to\infty$ and then $\e\to 0$ we conclude that 
\begin{equation}
\lim_{t\to\infty} \cM_{2}(f(t),g(t)) = 2L\ . \label{pl16}
\end{equation}

Now, take $(F_\infty,G_\infty) \in \omega(f_0,g_0)$. 
By Proposition~\ref{prpl1} and in particular \eqref{pl13b}, $(F_\infty,G_\infty)$ is a stationary solution to \eqref{RS} and there is a sequence
$(t_n)_n\subset (0,\infty)\setminus\mathcal{N}$  such that $t_n\to\infty$ and 
\begin{equation}
(f(t_n),g(t_n)) \ \to \ (F_\infty,G_\infty) \;\;\text{ in }\;\; L_2(\R,\R^2)\ . \label{pl17b}
\end{equation}
Since $(F_\infty,G_\infty)$ is a stationary solution to \eqref{RS}, we infer from Theorem~\ref{T:WP}~$(v)$ that 
\begin{equation}
\mathcal{M}_2(F_\infty,G_\infty) = 2 \E_*(F_\infty,G_\infty)\ , \label{pl17a}
\end{equation}
or, alternatively, owing to \eqref{eq:sup3},
\begin{equation}
\mathcal{M}_2(F_\infty,G_\infty) = 3 \E(F_\infty,G_\infty)\ . \label{pl17}
\end{equation}
Next, the convergence \eqref{pl17b} gives
\begin{equation}
\E(F_\infty,G_\infty) = \lim_{n\to\infty} \E(f(t_n),g(t_n))\ . \label{pl18}
\end{equation}
Since $\E_*(f,g)=\E(f,g)+\mathcal{M}_2(f,g)/6$ by \eqref{eq:sup3}, it follows from \eqref{pl15}, \eqref{pl16}, \eqref{pl17}, and \eqref{pl18} that 
\begin{align*}
\E_*(F_\infty,G_\infty) = & \E(F_\infty,G_\infty) + \frac{1}{6} \mathcal{M}_2(F_\infty,G_\infty) = \frac{3}{2} \E(F_\infty,G_\infty) \\
= & \frac{3}{2} \lim_{n\to\infty} \E(f(t_n),g(t_n)) = \frac{3}{2} \lim_{n\to\infty} \left[ \E_*(f(t_n),g(t_n)) - \frac{1}{6} \mathcal{M}_2(f(t_n),g(t_n)) \right] \\
= & \frac{3}{2} \left( L - \frac{L}{3} \right) = L\ .
\end{align*}
Recalling \eqref{pl17a}, we find $\mathcal{M}_2(F_\infty,G_\infty)=2L>0$. 
\end{proof}

\bigskip
 We are now in a position to prove our convergence  result Theorem~\ref{T:MT3}. 

\begin{proof}[Proof of Theorem~\ref{T:MT3}]
Consider $(f_0,g_0)\in \cK^2$ and let $(f,g)$ be the corresponding solution to \eqref{RS} given by Theorem~\ref{T:WP}. We aim at showing that $\omega(f_0,g_0)$ contains only one element. 
Indeed, we infer from Theorem~\ref{T:MT1}, Proposition~\ref{prpl1}, and Lemma~\ref{lepl2} that there is $L > 0$ such that 
$$
\omega(f_0,g_0) \subset \mathcal{S}_L := \left\{ (F_\ell,G_\ell)\ :\ \ell\in \Lambda \;\text{ and }\; \mathcal{E}_*(F_\ell,G_\ell) = L \right\}\ .
$$
According to Theorem~\ref{T:MT1} the set $\mathcal{S}_L$ contains at most two elements so that $\omega(f_0,g_0)$ is a discrete set and also contains at most two elements. 
Since it is connected in $H^{-3}(\R)$ by Proposition~\ref{prpl1} we conclude that it is reduced to a single element $(F_\infty,G_\infty)\in \mathcal{S}_L$. 
Consequently, $(f(t),g(t))_{ t\ge 0}$ converges weakly towards $(F_\infty,G_\infty)$ in $L_2(\R,\R^2)$ as $t\to\infty$. 

We now claim that  $(f(t),g(t))_{t\ge 0}$ converges  towards $(F_\infty,G_\infty)$ as $t\to\infty$  in $L_2(\R,\R^2)$.
To this end, we argue by contradiction and assume that there exist a sequence $t_n\to\infty$ and $\e>0$ such that 
$$
\E(f(t_n)-F_\infty, g(t_n)-G_\infty)\ge \e\qquad\text{for all $n\ge 1\ .$}
$$
Owing to the estimate~$(iv)$ of Theorem~\ref{T:WP} we may assume, after extracting eventually a subsequence, that  $\big(\E(f(t_n),g(t_n))\big)_{n\ge 1}$  converges in $\R.$
Since
\begin{align*}
\e&\le \E(f(t_n)-F_\infty, g(t_n)-G_\infty)\\
&=\E(f(t_n), g(t_n))+\E(F_\infty,G_\infty)-\eta^2\int_{\R} f(t_n) F_\infty\ dx \\
& \quad - R \int_{\R} \left( \eta f(t_n) + \frac{1}{\eta} g(t_n) \right) \left( \eta F_\infty + \frac{1}{\eta} G_\infty \right)\ dx\ , 
\end{align*}
the weak convergence in $L_2(\R,\R^2)$ ensures, after passing to the limit $n\to\infty,$ that
\begin{equation}\label{dl10}
\lim_{n\to\infty}\E(f(t_n), g(t_n))\geq \e+\E(F_\infty,G_\infty)\ .
\end{equation}
Due to Theorem~\ref{T:WP}~$(iv)$, there exists a sequence $s_n\to\infty$ with $s_n<t_n$  and $s_n\not\in\mathcal{N} $ for all $n\ge 1$. In view of Theorem~\ref{T:WP}~$(iv)$ we get
$$
\E_*(f(t_n), g(t_n))\leq \E_*(f(s_n), g(s_n)) \qquad\text{for all $n\ge 1\ .$}
$$
Since $\E_*(f(s_n), g(s_n))\to \E_*(F_\infty, G_\infty)=\E(F_\infty, G_\infty)+\cM_2(F_\infty, G_\infty)/6$ as $n\to\infty$, we obtain from \eqref{pl14}, \eqref{pl16}, and \eqref{dl10}, after passing to the limit $n\to\infty$ in the previous inequality,  that
 \begin{align*}
  \E(F_\infty, G_\infty)+\cM_2(F_\infty, G_\infty)/6 \geq& \lim_{n\to\infty}\E_*(f(t_n), g(t_n))\\
  =&\lim_{n\to\infty}\left(\E(f(t_n), g(t_n))+\frac{1}{6}\cM_2(f(t_n), g(t_n))\right)\\
  \geq &\e+\E(F_\infty,G_\infty)+\cM_2(F_\infty, G_\infty)/6\ ,
 \end{align*}
which is a contradiction.
This shows that our assumption was false, thus  $(f(t),g(t))_{t\geq0}$ converges  towards $(F_\infty,G_\infty)$ as $t\to\infty$  in $L_2(\R,\R^2)$.

Now, assume for contradiction that $\mathcal{M}_2(|f(t)-F_\infty|,|g(t)-G_\infty|)$ does not converge to zero as $t\to\infty$. There are then a sequence of positive times $(t_k)_{k\ge 1}$, $t_k\to\infty$, and $\delta>0$ such that $\mathcal{M}_2(|f(t_k)-F_\infty|,|g(t_k)-G_\infty|)>\delta$. 
Owing to the strong convergence of $(f(t),g(t))$ towards $(F_\infty,G_\infty)$ in $L_2(\R,\R^2)$ we may assume, after possibly extracting a further subsequence, 
that $(f(t_k),g(t_k))$ converges almost everywhere in $\R$ towards $(F_\infty,G_\infty)$. 
 Since $(F_\infty,G_\infty)\in\cK^2$, we infer from the dominated convergence theorem
\begin{align*}
\lim_{k\to\infty} \int_{\R} \left[ x^2 f(t_k,x) - x^2 |f(t_k,x)-F_\infty(x)| \right]\ dx & = \int_{\R} x^2 F_\infty(x)\ dx\ , \\
\lim_{k\to\infty} \int_{\R} \left[ x^2 g(t_k,x) - x^2 |g(t_k,x)-G_\infty(x)| \right]\ dx & = \int_{\R} x^2 G_\infty(x)\ dx\ .
\end{align*}
This implies that
$$
\lim_{k\to\infty} \left[ \mathcal{M}_2(f(t_k),g(t_k)) - \mathcal{M}_2(|f(t_k)-F_\infty|,|g(t_k)-G_\infty|) \right] = \mathcal{M}_2(F_\infty,G_\infty)\ ,
$$
hence, thanks to Lemma~\ref{lepl2} and \eqref{pl16},
$$
\lim_{k\to\infty} \mathcal{M}_2(|f(t_k)-F_\infty|,|g(t_k)-G_\infty|) = 0\ ,
$$
and a contradiction. 
We have thus shown that $(f(t),g(t))_{t\geq0}$ converges towards $(F_\infty,G_\infty)$ in  $L_1(\R, (1+x^2) dx,\R^2) \cap L_2(\R,\R^2)$.
\end{proof}

\bigskip

Finally, we show that the first moment of each weak solution of \eqref{RS} vanishes at an exponential rate.

%%%%%%%%%%%%%%%%%%%%%%%%%%%%%%%%%%%%%%%%
\begin{prop}\label{P:3}
 Define the first moment
 \[
\cM_{1}(u,v):=\int_\R ( u+\Theta v)(x) x\, dx\qquad \text{for $(u,v)\in\cK^2.$}
\]
If $(f,g)$ is a non-negative weak solution of \eqref{RS}, then
\begin{equation}\label{E:Eb}
 \cM_{1}(f(t),g(t))=\cM_{1}(f_0,g_0)e^{-t/3} \qquad\text{for all $t\geq 0.$}
\end{equation}
\end{prop}
%%%%%%%%%%%%%%%%%%%%%%%%%%%%%%%%%%%%%%%%

%%%%%%%%%%%%%%%%%%%%%%%%%%%%%%%%%%%%%%%%
\begin{rem}\label{R:4}
Particularly, relation \eqref{E:Eb} ensures that every steady state $(F_\infty,G_\infty)$ of \eqref{RS} satisfies the identity  $\cM_{1}(F_\infty,G_\infty)=0.$
\end{rem}
%%%%%%%%%%%%%%%%%%%%%%%%%%%%%%%%%%%%%%%%

\begin{proof} Choosing a suitable approximating sequence $(\xi_n)_n\subset C_0^\infty(\R)$ for the identity mapping on $\R,$ we obtain in view of  
\eqref{T2a}, \eqref{T2b}, and the estimate $(iv)$ of Theorem~\ref{T:WP} the following relation
\begin{equation*} 
 \cM_{1}(f(t),g(t))-\cM_{1}(f(s),g(s))+\frac{1}{3}\int_s^t\cM_{1}(f(\sigma),g(\sigma))\, d\sigma=0
\end{equation*}
for all $t\ge s \ge 0.$ This yields the desired claim.
\end{proof}

%%%%%%%%%%%%%%%%%%%%%%%%%%%%%%%%%%%%%%%%
%%%%%%%%%%%%%%%%%%%%%%%%%%%%%%%%%%%%%%%%
\section{Numerical simulations}\label{S:5}
%%%%%%%%%%%%%%%%%%%%%%%%%%%%%%%%%%%%%%%%
%%%%%%%%%%%%%%%%%%%%%%%%%%%%%%%%%%%%%%%%

In this section we present the results of several numerical simulations realized in the context of the rescaled system \eqref{RS}. We use the fully discrete finite volume
scheme for degenerate parabolic equations presented in \cite[Section~3.2]{FF12}, its accuracy being tested for the numerical simulation of various degenerate and 
non-degenerate parabolic equations in \cite{FF12} and which we present below.
More precisely, we will  compute the evolution of non-negative initial configurations $(f_0,g_0)$ that are compactly supported in the interval $\mathcal{I}:=(-5,5).$
This interval is discretized uniformly as follows 
$$
-5:=x_{1/2}<x_1<x_{3/2}<\ldots<x_{N_x}<x_{N_x+1/2}=5\ ,$$ 
whereby $N_x\in\N$  and $N_x\h=10.$ Here $\h/2$ denotes the spatial step size and $N_x$ is the number of control volumes $\{K_i=(x_{i-1/2}, x_{i+1/2})\}_{1\leq i\leq N_x}.$ 
The time step is denoted by $\Delta t>0$ and  $t^n:=n\Delta t$ for all non-negative integers $n$  less than or equal to the integer value $[T/\Delta t],$ $T>0$ 
being a positive fixed time.
The initial data $(f_0,g_0)$ are discretized as follows:
\begin{equation}\label{eq:IC}
 f_i^0:=\h^{-1}\int_{K_i} f_0\, dx\ ,\qquad  g_i^0:=\h^{-1}\int_{K_i} g_0\, dx, \qquad 1\leq i\leq N_x\ .
\end{equation}
Observe that the  system \eqref{RS} is written more compactly in the form
\begin{equation*} 
\left\{
\begin{array}{lll}
\p_t f+\p_x(- J_f)=0,\\[1ex]
\p_t g+\p_x (-J_g)=0,
\end{array}\right.
\end{equation*}
 with $-J_f=fV_f$ and $-J_g=gV_g$ being the advective fluxes defined in \eqref{Beth1}-\eqref{Beth2} and the velocities $V_f$ and $V_g$  given by
\[
V_f:=-\p_x\left( \eta^2(1+R) f+R g+\frac{x^2}{6}\right)\ , \qquad V_g:=-\p_x\left( \eta^2R_\mu f+   R_\mu g+\frac{x^2}{6} \right).
\]
In our setting the fully discrete scheme developed in \cite{FF12} for computing the approximation $(f_i^n, g_i^n)$ of the weak solution $(f,g)$ of \eqref{RS} 
on $K_i$ at time $t^n$ reads
\begin{equation}\label{NS} 
\left\{
\begin{array}{l}
\displaystyle{ \h\frac{f_i^{n+1}-f_i^n}{\Delta t}+\mathcal{F}_{i+1/2}^{n}-\mathcal{F}_{i-1/2}^n=0\ ,}\\
 \\
\displaystyle{ \h\frac{g_i^{n+1}-g_i^n}{\Delta t}+\mathcal{G}_{i+1/2}^{n}-\mathcal{G}_{i-1/2}^n=0\ ,}
\end{array}\right.
\end{equation}
for $1\leq i\leq N_x-1$ and $0\leq n\leq [T/\Delta t]-1.$
Here, $\mathcal{F}_{i+1/2}^{n}$ and $\mathcal{G}_{i+1/2}^{n}$ approximate the fluxes $-J_f$ and $-J_g$ at $(t^n,x_{i+1/2})$, respectively,  and  are discretized 
by the upwind method
\[\mathcal{F}_{i+1/2}^{n}=(A_{i+1/2}^n)^+f_i^n-(A_{i+1/2}^n)^-f_{i+1}^n\ ,\qquad\mathcal{G}_{i+1/2}^{n}=(B_{i+1/2}^n)^+g_i^n-(B_{i+1/2}^n)^-g_{i+1}^n\ ,\]
  where $x^+= \max\{0, x\}$ and $x^-= \max\{0, -x\}$. 
  Furthermore, $A_{i+1/2}$ and $B_{i+1/2}$ approximate the velocities $V_f$ and $V_g$ at  $(t^n,x_{i+1/2})$, respectively, and are defined by
  \begin{align*}
  A_{i+1/2}&:=-\frac{x_{i+1}+x_i}{6}-(1+R)\eta^2\frac{f^n_{i+1}-f^n_i}{\h}-R\frac{g^n_{i+1}-g^n_i}{\h},\\
  B_{i+1/2}&:=-\frac{x_{i+1}+x_i}{6}-\eta^2R_\mu\frac{f^n_{i+1}-f^n_i}{\h}-R_\mu\frac{g^n_{i+1}-g^n_i}{\h}.
  \end{align*}
  Finally, because we expect the weak solutions to remain compactly supported, which is also suggested by the numerical simulations, we supplement \eqref{eq:IC} 
  and \eqref{NS} by no-flux conditions on the boundary $\p\mathcal{I}.$ 
%%%%%%%%%%%%%%%%%%%%%%%%%%%%%%%%%%%%%%%%  
\begin{figure}[h]
\hspace{1cm}\includegraphics[width=2.5in, angle=270]{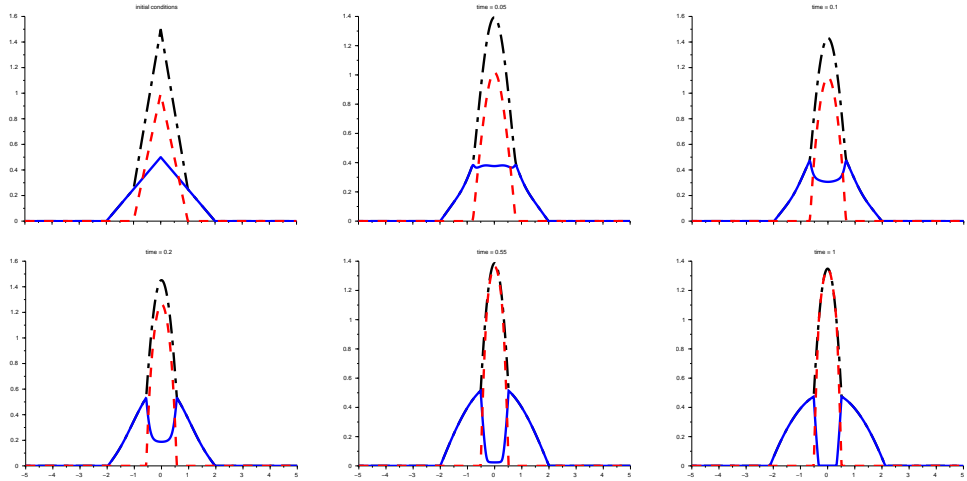}
\caption[Even initial data]{\small Time evolution (from  left up to right down) of the weak solution of \eqref{RS} corresponding to an even initial configuration
for $\eta=1$, $R=1$, $R_\mu=0.05$, $\Delta t=10^{-5},$ and $N_x=1000$. The blue line is $f$, the dashed red line is $g$, and the dash-dotted
 black line is $\eta^2f+g$.}\label{F:4}
\end{figure}
%%%%%%%%%%%%%%%%%%%%%%%%%%%%%%%%%%%%%%%%

Our simulations are all performed in the regime  $R_\mu<R_\mu^-(R,\eta)$.
  This regime is physically highly  relevant as $R_\mu<R_\mu^-(R,\eta)$ exactly when 
\begin{equation}\label{RRR}
 \frac{\mu_-}{\mu_+}<\frac{\rho_-\rho_+^2}{\rho_+^3+(\eta^2\rho_-+\rho_+)^2(\rho_--\rho_+)}.
\end{equation}
%%%%%%%%%%%%%%%%%%%%%%%%%%%%%%%%%%%%%%%%
\begin{figure}[h]
\hspace{1cm}\includegraphics[width=2.5in, angle=270]{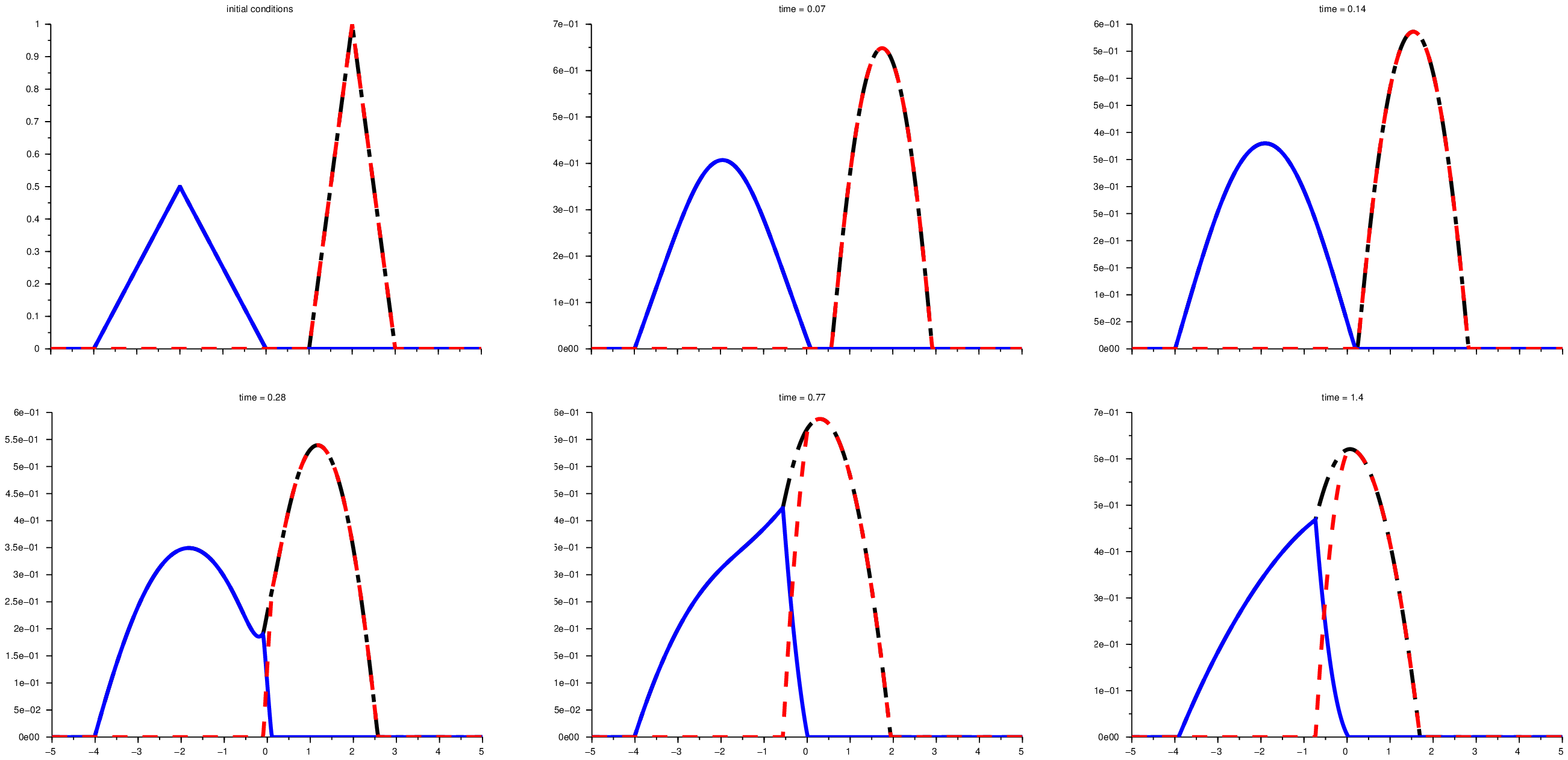}
\caption[Odd initial data]{\small Time evolution of the weak solution corresponding to a non-symmetric initial configuration for $\eta=1$, 
 $R=4$, $R_\mu=0.7$, $\Delta t=10^{-5},$ and $N_x=1000$. 
   The blue line is $f$, the dashed red line is $g$, and the dash-dotted  black line is $\eta^2f+g$.
   The solution converges towards the self-similar profile $(F_{\ell_-}, G_{\ell_-}).$ }\label{F:5}
\end{figure}
%%%%%%%%%%%%%%%%%%%%%%%%%%%%%%%%%%%%%%%%

\noindent The inequality \eqref{RRR} holds for example when  the denser fluid is water,  the other one is rapeseed oil, and $\eta=1$.
Indeed, at $20^\circ {\rm C},$ water has density $\rho_-\approx\ 1\ {\rm kg/litre}$  and viscosity $\mu_-\approx \ 1\ {\rm mPa \cdot s}$,
respectively $\rho_+\approx\ 0.92\ {\rm kg/litre}$  and  $\mu_+\approx\ 67.84\ {\rm mPa \cdot s}$ for rapeseed oil, cf. \cite{Esteban12}.

%%%%%%%%%%%%%%%%%%%%%%%%%%%%%%%%%%%%%%%%
\begin{figure}[h]
\hspace{1cm}\includegraphics[width=2.5in, angle=270]{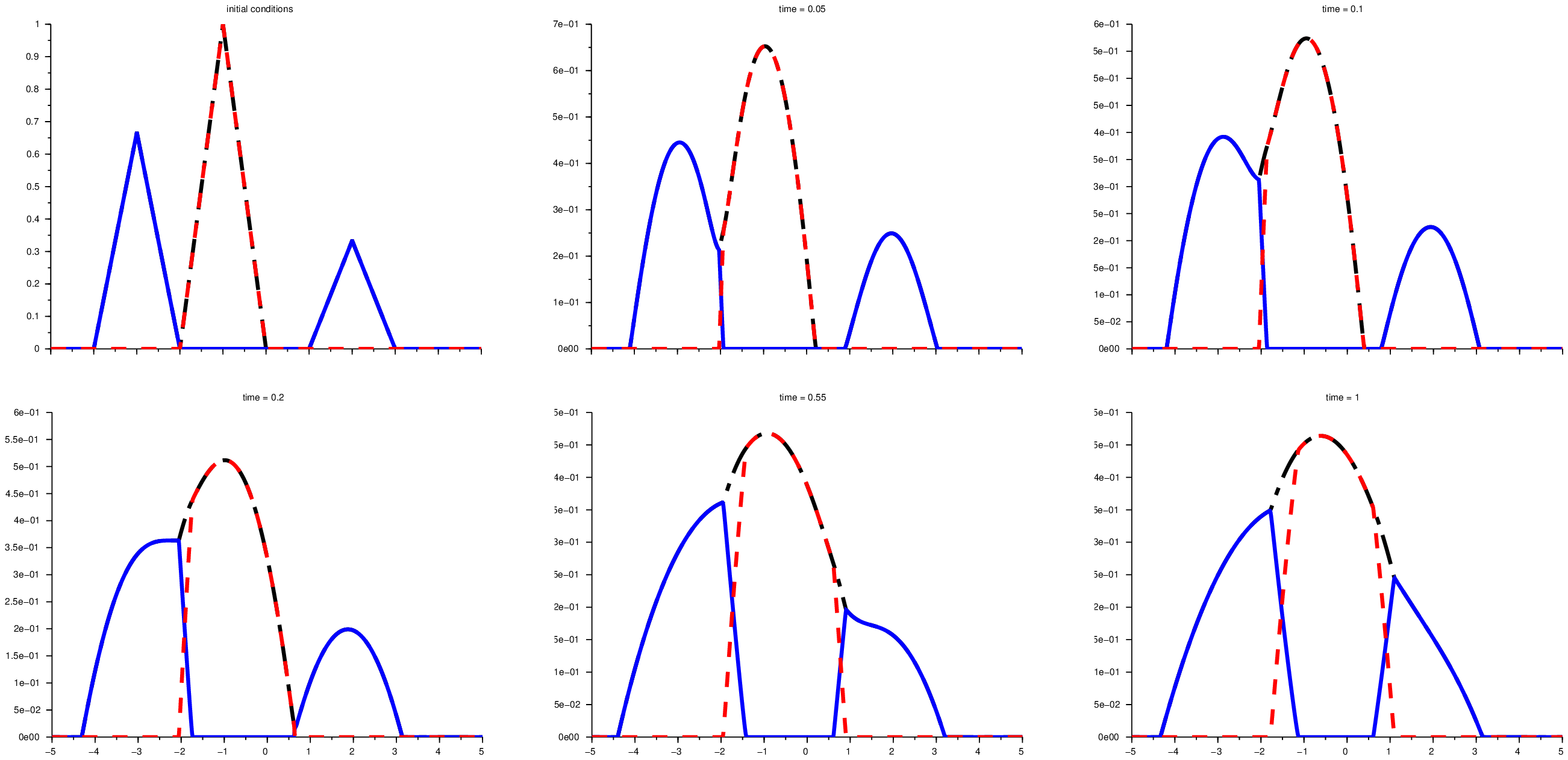}
\caption[Odd self-similar profiles]{\small Time evolution of the weak solution corresponding to a non-symmetric initial configuration for $\eta=1$,
$R=4$, $R_\mu=2$, $\Delta t=10^{-5},$ and $N_x=1000$. 
   The blue line is $f$, the dashed red line is $g$, and the dash-dotted black line is $\eta^2f+g$.
   The solution converges towards  a self-similar profile $(F_{\ell}, G_{\ell}) $ with $\ell\in(\ell_-,\ell_+)\setminus\{0\}$.}\label{F:6}
\end{figure}
%%%%%%%%%%%%%%%%%%%%%%%%%%%%%%%%%%%%%%%%

The scope of the simulations is threefold.
First, it can be seen from Figures~\ref{F:4}-\ref{F:6} that if the initial data are compactly supported they remain so as time evolves. 
This suggests that the  supports of weak  solutions of \eqref{RS}, and also of \eqref{eq:S2}, propagate with finite speed.

Secondly,  we have rigorously established in Theorem~\ref{T:MT3} that  weak solutions which correspond  to  even initial data  converge
towards the unique even stationary solution $(F_0,G_0)$ of \eqref{RS}. This even self-similar profile has the property that the 
positivity set of $F_0$ consists on two intervals if $R_\mu<R_\mu^-(R,\eta)$, cf. Proposition~\ref{P:1}. 
Hence, if the initial data have connected positivity sets, then the denser film will break at least in infinite time. 
Figure~\ref{F:4} suggests that in fact the film rupture occurs in finite time.   

At last Figures~\ref{F:5}-\ref{F:6} display the fact that the even self-similar profile is not a universal attractor 
for the dynamics and that other profiles belonging to the continuum found in Theorem~\ref{T:MT1} attract certain weak solutions of \eqref{RS}.  

Let us emphasize that  the above numerical simulations reveal some qualitative properties of the dynamics of \eqref{RS} which have not 
yet been studied analytically, including:
\begin{itemize}
\item the property of finite propagation speed of solutions of \eqref{RS},
\item the finite time film rupture in the  small/large viscosities ratio regime,
\item the fact that in the  small/large viscosities ratio regime each of the self-similar profiles  attracts certain weak solutions of the rescaled system \eqref{RS}.
%                       \item the uniqueness of weak solutions of \eqref{RS},
\end{itemize}

%%%%%%%%%%%%%%%%%%%%%%%%%%%%%%%%%%%%%%%%
%%%%%%%%%%%%%%%%%%%%%%%%%%%%%%%%%%%%%%%%

\appendix
\section{Solvability of the auxiliary algebraic systems}
%%%%%%%%%%%%%%%%%%%%%%%%%%%%%%%%%%%%%%%%
%%%%%%%%%%%%%%%%%%%%%%%%%%%%%%%%%%%%%%%%

We first study the system of three algebraic equations \eqref{eq:41}-\eqref{eq:43} arising in the analysis of even self-similar profiles in Section~\ref{S:22}. 

%%%%%%%%%%%%%%%%%%%%%%%%%%%%%%%%%%%%%%%%
\begin{lemma}\label{L:Z1}
Let $(R,R_\mu,\eta)$ be three positive real numbers such that $R_\mu>R+1$. The system of algebraic equations \eqref{eq:41}-\eqref{eq:43} has a 
unique solution $(\alpha,\beta, \gamma)$ satisfying 
$0\le \alpha < \beta < \gamma$ if and only if $R\ge R_\mu^+(R,\eta)$. Moreover, $\alpha>0$ if $R_\mu>R_\mu^+(R,\eta)$ and $\alpha=0$ if $R_\mu=R_\mu^+(R,\eta)$.
\end{lemma}
%%%%%%%%%%%%%%%%%%%%%%%%%%%%%%%%%%%%%%%%

\begin{proof}
Combining \eqref{eq:41}-\eqref{eq:43} gives
\begin{equation}
\begin{split}
\beta^3 & = \frac{R(R_\mu-R-1)}{(R+1)(R_\mu-R)} \alpha^3 + \frac{9 R_\mu}{2 (R_\mu-R)} \eta^2\ , \\
\gamma^3 & = -\frac{R_\mu-R-1}{R+1} \alpha^3 + \frac{9 R_\mu}{2} (1+\eta^2)\ ,
\end{split} \label{gaston}
\end{equation}
and
\begin{equation}
(A_1 - B_1 \alpha^3)^{2/3} - (A_2+B_2 \alpha^3)^{2/3} +(R_\mu-R-1) \alpha^2 = 0 \label{2.14b}
\end{equation}
with
\begin{align*}
&A_1 := \frac{9R_\mu}{2}  (1+\eta^2)> 0\ ,  \qquad B_1 := \frac{R_\mu-R-1}{R+1}>0 \ , \\
&A_2 := \frac{9R_\mu}{2}  \eta^2 \sqrt{R_\mu-R} > 0\ ,  \qquad B_2 := \frac{R (R_\mu-R-1)}{R+1} \sqrt{R_\mu-R} > 0 \ .
\end{align*}
We also observe that, if $(\alpha,\beta, \gamma)$ solves \eqref{eq:41}-\eqref{eq:43} with $0\le \alpha < \beta < \gamma$, then 
\begin{align*}
0 < (R_\mu-R) (\beta^3 - \alpha^3) & = \frac{R (R_\mu-R-1)}{R+1} \alpha^3 + \frac{9R_\mu}{2}  \eta^2 - (R_\mu-R) \alpha^3 \\
& = \frac{9R_\mu}{2}  \eta^2 - \frac{R_\mu}{1+R} \alpha^3\ ,
\end{align*}
whence
\begin{equation}
\alpha^3 < \frac{9}{2} (1+R) \eta^2\ . \label{2.14bb}
\end{equation}

We first look for positive solutions to \eqref{2.14b}.
To this end we set $Y=\alpha^{-3}$ and multiply \eqref{2.14b} by $\alpha^{-2}$ to obtain that $Y$ is a positive zero of the function 
\begin{equation}
\xi(y) := (A_1 y - B_1)^{2/3} - (A_2 y + B_2)^{2/3} + R_\mu-R-1\ , \quad y\ge 0\ . \label{2.14c}
\end{equation}
Then, for $y\ne y_1:=B_1/A_1$, 
$$
\xi'(y) = \frac{2}{3} \frac{A_1}{(A_2y+B_2)^{1/3}} \left[ \mathrm{sign}(y-y_1) \left( \frac{A_2 y + B_2}{|A_1 y - B_1|} \right)^{1/3} - \frac{A_2}{A_1} \right]\ ,
$$
so that 
\begin{itemize}
\item $\xi'(y)<0$ for $y\in [0,y_1)$,
\item $\xi'(y)>0$ for $y\in (y_1,\infty)$ if $A_2\le A_1$,
\item there is a unique $y_0\in (y_1,\infty)$ such that $\xi'>0$ in $(y_1,y_0)$, $\xi'<0$ in $(y_0,\infty)$, and $\xi'(y_0)=0$ if $A_2>A_1$. 
\end{itemize}
Moreover we note that
$$
y_2 := \frac{2}{9 \eta^2 (1+R)} = \frac{B_1}{A_1} \frac{1+\eta^2}{\eta^2} \frac{R_\mu}{R_\mu-R-1} > y_1\ ,
$$
with
\begin{equation}
\xi(y_2) = \left( 1 + \frac{R_\mu}{\eta^2 (1+R)} \right)^{2/3} - 1 > 0\ . \label{2.14d}
\end{equation}
Finally
\begin{equation*}
\lim_{y\to\infty} \xi(y) = \left\{ 
\begin{array}{lcl}
\infty & \text{ if } & A_1>A_2\ , \\
R_\mu - R - 1 & \text{ if } & A_1=A_2\ , \\
- \infty & \text{ if } & A_1<A_2\ .
\end{array}
\right. %\label{2.14e}
\end{equation*}
Recalling the constraint \eqref{2.14bb}, we thus look for a zero of $\xi$ in $(y_2,\infty)$. 
There is none if $A_1\ge A_2$ according to \eqref{2.14d} and the monotonicity of $\xi$. 
If $A_2>A_1$, we infer from \eqref{2.14d} and the behavior of $\xi$ that $\xi$ has a unique zero $Y>y_2$. 
Setting $\alpha=Y^{-1/3}$ and defining $(\beta,\gamma)$ by \eqref{gaston}, the property $Y>y_2$ implies that the constraint \eqref{2.14bb} 
is satisfied, so that $\alpha<\beta$. 
Furthermore, the properties $\alpha<\beta$ and $R_\mu>R+1$ and \eqref{eq:42} guarantee that $\gamma^3>\beta^3$. Finally, the requirement
$A_2>A_1$ for $Y$ to exist is equivalent to $R_\mu>R_\mu^+(R,\eta)$. 

It remains to consider the possibility of having $\alpha=0$. Then \eqref{2.14b} implies that $A_1=A_2$ and thus $R_\mu=R_\mu^+(R,\eta)$. 
We deduce from \eqref{gaston} that 
$$
\beta^3 = \frac{9R_\mu}{2} \frac{\eta^6}{(1+\eta^2)^2}>0\ , \qquad \gamma^3 = \frac{9R_\mu}{2} (1+\eta^2) > \beta^3\ ,
$$
which completes the proof.
\end{proof}

We next turn to the system of four algebraic equations \eqref{d1:V1}-\eqref{d1:V4} arising in the study of non-symmetric self-similar profiles with connected supports in Section~\ref{S:23}.

%%%%%%%%%%%%%%%%%%%%%%%%%%%%%%%%%%%%%%%%
\begin{lemma}\label{P:NSC}
Let $(R,R_\mu,\eta)$ be three positive real numbers such that $R_\mu>R+1$. 
There exists a constant $R_\mu^M(R,\eta)>R_\mu^+(R,\eta)$ with the property  that the system of algebraic equations \eqref{d1:V1}-\eqref{d1:V4} has a unique solution $(\beta_1,\alpha,\beta,\gamma)$ with 
\begin{equation}
\beta_1<0\leq\alpha<\beta<\gamma \label{chopin}
\end{equation}
for each $R_\mu\geq R_\mu^M(R,\eta),$ and it has no such solution when $R+1<R_\mu<R_\mu^M(R,\eta).$ Moreover, $-\beta_1>\alpha$ for all $R_\mu\ge R_\mu^M(R,\eta)$ and $\alpha=0$ if and only if $R_\mu=R_\mu^M(R,\eta).$ 
\end{lemma}
%%%%%%%%%%%%%%%%%%%%%%%%%%%%%%%%%%%%%%%%

\begin{proof}
We fix $R,\eta\in (0,\infty)$ and consider $R_\mu>R+1$ as being a variable parameter. We observe that, if $(\beta_1,\alpha,\beta,\gamma)$ is a solution of \eqref{d1:V1}-\eqref{d1:V4} satisfying \eqref{chopin}, then the new variables
\begin{align}\label{UN1}
x:=\frac{\gamma}{\beta},\quad y:=\frac{\beta_1}{\beta},\quad z:=\frac{\alpha}{\beta}
\end{align}
are ordered as follows $y<0\leq z<1<x$. 
Moreover, dividing the equations \eqref{d1:V1}-\eqref{d1:V2} by $\beta^3$ and  \eqref{d1:V3}-\eqref{d1:V4} by $\beta^2$, we find the following relations
\begin{eqnarray}
\label{g1:1}
R_\mu y^2+R(R_\mu-R-1)z^2&=&(1+R)(R_\mu-R)\ ,\\
\label{g1:2}
x^2+(R_\mu-R-1)z^2&=&(R_\mu-R)\ ,\\
\label{g1:3}
R_\mu y^3+R(R_\mu-R-1)z^3-(1+R)(R_\mu-R)&=&-\frac{9\eta^2R_\mu(1+R)}{\beta^3}\ ,\\
\label{g1:4}
x^3 +(R_\mu-R-1)z^3-(R_\mu-R)&=&\frac{9R_\mu}{\beta^3}\ .
\end{eqnarray}
Extracting $x$ and $y$ from the relations \eqref{g1:1}-\eqref{g1:2} gives
\begin{equation}\label{xy}
\begin{aligned}
 x&=\sqrt{R_\mu-R-(R_\mu-R-1)z^2}\ ,\\
 y&=-\sqrt{\frac{(1+R)(R_\mu-R)-R(R_\mu-R-1)z^2}{R_\mu}}\ ,
\end{aligned}
\end{equation}
and recalling that $R_\mu>R+1$, we see that indeed $x>1$ and $y<0$ are well-defined for $z\in[0,1).$ 
We eliminate now $\beta$ from \eqref{g1:3}-\eqref{g1:4} and arrive at the equation $\xi_0(R_\mu, z)=0,$ where the function $\xi_0:(R+1,\infty)\times[0,1)\to\R$ is defined by 
\begin{align*}
 \xi_0(R_\mu,z):=&-(1+\eta^2)(1+R)(R_\mu-R)-\frac{1}{\sqrt{R_\mu}}\left[(1+R)(R_\mu-R)-R(R_\mu-R-1)z^2\right]^{3/2}\nonumber\\
 &+\eta^2(1+R)\left[R_\mu-R-(R_\mu-R-1)z^2\right]^{3/2}+ (R_\mu-R-1)(R+\eta^2(1+R))z^3.
\end{align*}
Clearly, if for some $R_\mu>R+1$ the map $\xi_0(R_\mu,\cdot)$ has no zero, then \eqref{d1:V1}-\eqref{d1:V4} has no solution with the desired ordering. 
Conversely, each zero of $\xi_0(R_\mu,\cdot)$ provides a unique solution $(\beta_1,\alpha,\beta,\gamma)$ of \eqref{d1:V1}-\eqref{d1:V4} satisfying \eqref{chopin}. 
Indeed, let $z\in [0,1)$ be a solution  $\xi_0(R_\mu,z)=0$ and define $x$ and $y$ by \eqref{xy}. Then, in view of  $z\in[0,1)$ and $R_\mu>R+1$, both $x$ and $y$ are well-defined and $y<0\leq z<1<x.$ 
Together with \eqref{g1:1} and \eqref{g1:3} this implies that 
\begin{equation}
\frac{9\eta^2R_\mu(1+R)}{\beta^3}=R(R_\mu-R-1)z^2(1-z)+R_\mu y^2(1-y)>0\ , \label{bxy}
\end{equation}
which uniquely determines $\beta$. 
Thus, with $\beta_1,\alpha,\gamma$ given by \eqref{UN1}, we obtain a solution of \eqref{d1:V1}-\eqref{d1:V4} satisfying \eqref{chopin}. 
We emphasize that there is in fact a one-to-one correspondence between the solutions $\beta_1<0\leq\alpha<\beta<\gamma$ of the system \eqref{d1:V1}-\eqref{d1:V4} 
and the solutions $z\in[0,1)$ of $\xi_0(R_\mu, z)=0,$ see \eqref{g1:3}, \eqref{xy}, and \eqref{bxy}. 

\smallskip
 
Thanks to the previous analysis, we are left with the simpler task of determining the zeros of $\xi_0.$ We first note that
\begin{equation}\label{1}
\lim_{z\to 1}\xi_0(R_\mu, z)=-2R_\mu<0\ ,
\end{equation}
and
\begin{equation}\label{der1}
\p_z\xi_0(R_\mu,z)=3(R_\mu-R-1)z^2\xi_1\left( R_\mu, \sqrt{ \frac{R_\mu-R}{z^2} - (R_\mu-R-1)} \right)\ , \quad z\in [0,1)\ ,
\end{equation}
with 
\begin{equation}\label{der2}
\xi_1(R_\mu,t) := \frac{R}{\sqrt{R_\mu}} \sqrt{(1+R) t^2 + R_\mu-R-1} - \eta^2 (1+R) t + R + \eta^2(1+R)
\end{equation}
for $(R_\mu,t)\in (R+1,\infty)\times(1,\infty)$. Note that 
\begin{equation}
\lim_{t\to 1} \xi_1(R_\mu,t)=2R \label{der2.5}
\end{equation}
and
\begin{align}\label{der3}
\lim_{t\to\infty}\xi_1(R_\mu,t) =\left\{
\begin{array}{lll}
\infty&,& R^2>\eta^4R_\mu(1+R),\\
 & & \\
R+\eta^2(1+R)&,& R^2=\eta^4R_\mu(1+R),\\
& & \\
-\infty&,& R^2<\eta^4R_\mu(1+R).
\end{array}
\right.
\end{align}
In addition, 
\begin{align}
\p_t\xi_1(R_\mu,t) =&  \frac{R(1+R)}{\sqrt{R_\mu}} \frac{\zeta(R_\mu,t^2)}{R^2 ((1+R) t^2 + R_\mu-R-1)} \\
& \qquad \times \left( \frac{t}{\sqrt{(1+R) t^2 + R_\mu - R - 1}} + \frac{\eta^2 \sqrt{R_\mu}}{R} \right)^{{-1}}\ ,\label{der4}
\end{align}
with
$$
\zeta(R_\mu,s) := (R^2-\eta^4R_\mu(1+R)) s - \eta^4 R_\mu (R_\mu-R-1)
$$
for $(R_\mu,s)\in (1+R,\infty)\times(1,\infty)$.

\medskip

We handle separately different cases:

\smallskip

\noindent\textbf{Case~1: $R^2>\eta^4R_\mu(1+R)$.} Introducing 
$$
s_0 := \frac{\eta^4 R_\mu (R_\mu-R-1)}{R^2-\eta^4R_\mu(1+R)} >0\ ,
$$
either $s_0\le 1$ and $\zeta(R_\mu,s)> \zeta(R_\mu,s_0) = 0$ for $s>1$ and we deduce that $\partial_t \xi_1(R_\mu,t)>0$ for $t>1$. 
Consequently, $\xi_1(R_\mu,t)>2R$ by \eqref{der2.5}. 
Or $s_0>1$ and $(t-\sqrt{s_0}) \partial_t \xi_1(R_\mu,t) \ge 0$ for $t>1$ with equality only when $t=\sqrt{s_0}$. 
Therefore, 
$$
\xi_1(R_\mu,t) \ge \xi_1(R_\mu,\sqrt{s_0})=\frac{R^2-\eta^4R_\mu(1+R)}{\eta^2R_\mu} \sqrt{s_0} + R + \eta^2(1+R)>0\ ,
$$
and we conclude that $\xi_1(R_\mu,\cdot)$ is positive in $(1,\infty)$ and $\xi_0(R_\mu,\cdot)$ is increasing in $(1,\infty)$ by \eqref{der1}. 
Recalling \eqref{1}, $\xi_0(R_\mu,\cdot)$ is negative in $[0,1)$ and the equation $\xi_0(R_\mu,z)=0$ has no solution in $[0,1)$.

\smallskip

\noindent\textbf{Case~2: $R^2=\eta^4R_\mu(1+R)$.} 
In that case, $\zeta(R_\mu,\cdot)$ is obviously negative in $(1,\infty)$ which, together with \eqref{der3} and \eqref{der4}, entails that $\xi_1(R_\mu,t)>0$ for all $t>1.$ 
Recalling \eqref{1} and \eqref{der1}, we conclude that the equation $\xi_0(R_\mu,z)=0$ has no solution in $[0,1).$

\smallskip

\noindent\textbf{Case~3: $R^2<\eta^4R_\mu(1+R)$.} In that case $\zeta(R_\mu,\cdot)<0$ in $(1,\infty)$ and $\xi_1(R_\mu,\cdot)$ is decreasing from $(1,\infty)$ 
onto $(-\infty,2R)$. 
There is thus a unique $t_1\in (1,\infty)$ such that
$$
\xi_1(R_\mu,t)>0 \;\;\text{ for }\;\; t\in (1,t_1) \;\;\text{ and }\;\; \xi_1(R_\mu,t)<0 \;\;\text{ for }\;\; t\in (t_1,\infty)\ .
$$
Setting 
$$
z_1 := \sqrt{\frac{R_\mu-R}{t_1^2+ R_\mu-R-1}} \in (0,1)\ ,
$$ 
it follows from \eqref{der1} that
$$
\partial_z \xi_0(R_\mu,z)>0 \;\;\text{ for }\;\; z\in (z_1,1) \;\;\text{ and }\;\; \partial_z \xi_0(R_\mu,z)<0 \;\;\text{ for }\;\; z\in (0,z_1)\ .
$$
Recalling \eqref{1}, we realize that $\xi_0(R_\mu,z)<-2 R_\mu<0$ for $z\in [z_1,1)$ so that the function $\xi_0(R_\mu,\cdot)$ vanishes at most once in $[0,1)$ and
necessarily in $[0,z_1)$. 
Clearly this can only happen if $\xi_0(R_\mu,0)\ge 0$.  

\smallskip

Summarizing, we have shown that the equation $\xi_0(R_\mu,z)=0$ has a solution $z\in [0,1)$ if and only if $\xi_0(R_\mu,0)\ge 0,$ this solution being unique.

\smallskip

To see when the inequality $\xi_0(R_\mu,0)\ge 0$ holds, we observe that
\begin{equation}\label{0}
\xi_0(R_\mu, 0)=(1+R)(R_\mu-R) \xi_2(R_\mu-R)\ ,
 \end{equation}
with
$$
\xi_2(t) := t^{1/2} \left(\eta^2-\sqrt{\frac{1+R}{t+R}}\right) - 1 - \eta^2\ , \quad t>1\ .
$$
The function $\xi_2$ satisfies 
$$
\lim_{t\to1}\xi_2(t) = -2 \qquad\text{and}\qquad \lim_{t\to\infty}\xi_2(t) =\infty\ ,
$$
and
$$
\xi_2'(t)= \frac{\eta^2}{2 \sqrt{t} (R+t)^{3/2}} \left[ (R+t)^{3/2} - \frac{R \sqrt{1+R}}{\eta^2} \right]\ , \quad t>1\ .
$$
Therefore,
\begin{align*}
\xi_2'(t) > 0 & \;\;\text{ for }\;\; t> \max\{1,t_2\}\ , \quad t_2 := \frac{R^{2/3} (1+R)^{1/3}}{\eta^{4/3}} - R\ , \\
\xi_2'(t) < 0 & \;\;\text{ for }\;\; t\in (1,\max\{1,t_2\})\ ,
\end{align*}
and $\xi_2$ has a unique zero $t_M\in (1,\infty)$. We set $R_\mu^M(R,\eta) := R + t_M$. Since
$$
\xi_2(R_\mu^+(R,\eta)-R)=- \frac{1+\eta^2}{\eta^2} \sqrt{\frac{1+R}{R_\mu^+(R,\eta)}}<0\ ,
$$
we conclude that $R_\mu^M(R,\eta)>R_\mu^+(R,\eta).$ 
With \eqref{0}, we have thus shown that $\xi_0(R_\mu, 0)<0$ for 
$R_\mu<R_\mu^M(R,\eta)$, $\xi_0(R_\mu^M(R,\eta), 0)=0$, and $\xi_0(R_\mu, 0)>0$ for $R_\mu>R_\mu^M(R,\eta)$.
Returning to the original problem, we have proven that \eqref{d1:V1}-\eqref{d1:V4} has a unique solution $(\beta_1,\alpha,\beta,\gamma)$ 
satisfying \eqref{chopin} if and only if $R_\mu\ge R_\mu^M(R,\eta)$ and the property $\xi_0(R_\mu^M(R,\eta), 0)=0$ entails that $\alpha=0$ 
if $R_\mu=R_\mu^M(R,\eta)$. We finally note that, if $R_\mu\ge R_\mu^M(R,\eta)>R+1$ and $(\beta_1,\alpha,\beta,\gamma)$ denotes the 
corresponding solution to \eqref{d1:V1}-\eqref{d1:V4} satisfying \eqref{chopin}, it follows from \eqref{d1:V4} and \eqref{chopin} that
$$
R_\mu \beta_1^2 > (1+R)(R_\mu-R) \alpha^2 - R (R_\mu-R-1) \alpha^2 = R_\mu \alpha^2\ ,
$$
hence $-\beta_1>\alpha$ by \eqref{chopin}. The proof of Lemma~\ref{P:NSC} is then complete.
\end{proof}

In the next lemma, we study some particular solutions of the system of five algebraic equations \eqref{R1:1}-\eqref{R1:5}.

%%%%%%%%%%%%%%%%%%%%%%%%%%%%%%%%%%%%%%%%
\begin{lemma}\label{L:2.12b}
Let $(R,R_\mu,\eta)$ be three positive real numbers such that $R_\mu>R+1$ and consider a solution
$(\gamma_1,\beta_1,\alpha_1,\alpha,\beta,\gamma)\in\mathbb{R}^6$ of the system of algebraic equations \eqref{R1:1}-\eqref{R1:5} such that
\begin{equation}
\gamma_1 \le \beta_1 \le \alpha_1 \le 0 \le \alpha \le \beta \le \gamma\ . \label{corelli}
\end{equation}
\begin{itemize}
\item[$(i)$] If $\gamma_1=\beta_1$ or $\beta_1=\alpha_1,$ then $\alpha_1=\beta_1=\gamma_1 <0\leq \alpha<\beta<\gamma$ and $(\beta_1,\alpha,\beta,\gamma)$
solves \eqref{d1:V1}-\eqref{d1:V4}.

\item[$(ii)$] If $\gamma=\beta$ or $\beta=\alpha,$ then $\gamma_1<\beta_1<\alpha_1\leq 0<\alpha=\beta=\gamma$ and $(-\alpha,-\alpha_1,-\beta_1,-\gamma_1)$ 
solves \eqref{d1:V1}-\eqref{d1:V4}.

\item[$(iii)$] If $\alpha_1=-\alpha$, then $\gamma_1=-\gamma$, $\beta_1=-\beta$, $0\leq\alpha<\beta<\gamma$, and $(\alpha,\beta,\gamma)$ solves \eqref{eq:41}-\eqref{eq:43}.
\end{itemize}
\end{lemma}
%%%%%%%%%%%%%%%%%%%%%%%%%%%%%%%%%%%%%%%%

\begin{proof}
$(i)$: If $\gamma_1=\beta_1$ or $\beta_1=\alpha_1$ it readily follows from \eqref{R1:1} that
$\gamma_1^2-\alpha_1^2 = (R_\mu-R) (\beta_1^2 - \alpha_1^2)$ and thus $\alpha_1=\beta_1=\gamma_1$ by \eqref{corelli}. 
A similar argument using \eqref{R1:2} and \eqref{corelli} shows that, if $\alpha=\beta$ or $\beta=\gamma$, then $\alpha=\beta=\gamma$. In that case
equation~\eqref{R1:3} yields additionally that $\alpha_1=-\alpha$, which contradicts \eqref{R1:5}. Consequently
$0\leq \alpha<\beta<\gamma$. 
Next, if $\alpha_1=0,$ then we infer from \eqref{R1:2}-\eqref{R1:3} that $\alpha=\beta=\gamma=0$, which also contradicts \eqref{R1:5}.
Finally, we  infer from \eqref{R1:2}-\eqref{R1:5} that $(\beta_1,\alpha,\beta,\gamma)$ solves \eqref{d1:V1}-\eqref{d1:V4}.

\noindent $(ii)$: We simply note that $(-\gamma,-\beta,-\alpha,-\alpha_1,-\beta_1,-\gamma_1)$ also satisfies \eqref{R1:1}-\eqref{R1:5} and \eqref{corelli} 
and deduce~$(ii)$ from~$(i)$.

\noindent $(iii)$: If $\alpha_1=-\alpha$, we infer from \eqref{R1:1}-\eqref{R1:3} that 
$$
\gamma^2-\gamma_1^2 = \beta^2-\beta_1^2 = 0\ .
$$
Combining these identities with \eqref{corelli} entails that $\gamma_1 = -\gamma$ and $\beta_1= -\beta$. 
A further use of \eqref{R1:1}-\eqref{R1:3} shows that $(\alpha,\beta,\gamma)$ solves \eqref{eq:41}-\eqref{eq:43}.
To conclude, we note that if $\alpha=\beta$ or $\beta=\gamma$, equation \eqref{eq:43} implies $\alpha=\beta=\gamma,$ which is not possible by \eqref{eq:42}.
\end{proof}

We investigate now the existence of solutions of the systems \eqref{R1:1}-\eqref{R1:5} and \eqref{R2:1}-\eqref{R2:5} 
which satisfy \eqref{haydn} as well as $\alpha=0$ and $\alpha_1>0$ (or equivalently $\alpha_1=0$ and $\alpha>0$ since 
these systems are invariant with respect to the transformation
$(\gamma_1,\beta_1,\alpha_1,\alpha,\beta,\gamma)\mapsto (-\gamma,-\beta,-\alpha,-\alpha_1,-\beta_1,-\gamma_1)$).
It turns out that, in this case, the solution, if it exists, is uniquely determined by the constants $R_\mu, R,$ and $\eta.$

%%%%%%%%%%%%%%%%%%%%%%%%%%%%%%%%%%%%%%%%%%%
\begin{lemma}\label{L:Nconn}
 Let $R$, $R_\mu$, and $\eta$ be given positive numbers such that $R_\mu>R+1$. Then the system \eqref{R1:1}-\eqref{R1:5} has a
 unique solution $(\gamma_1,\beta_1,\alpha_1,\alpha,\beta,\gamma)$ satisfying 
\begin{equation}
\gamma_1<\beta_1<\alpha_1<0=\alpha<\beta<\gamma\ , \label{monteverdi}
\end{equation} 
for each $R_\mu^+(R,\eta)<R_\mu< R_\mu^M(R,\eta)$ and no solution with this property if $R_\mu\in (R+1, R_\mu^+(R,\eta)]\cup [R_\mu^M(R,\eta), \infty).$
\end{lemma}
%%%%%%%%%%%%%%%%%%%%%%%%%%%%%%%%%%%%%%%%%%%%

\begin{proof} Let $R_\mu>R+1$ be given (recall that we are in \textbf{Case~(II-a)}) and 
pick a solution $(\gamma_1,\beta_1,\alpha_1,\alpha,\beta,\gamma)$ of \eqref{R1:1}-\eqref{R1:5} satisfying \eqref{monteverdi}. It is useful to define the variables
\begin{align}\label{nv3}
x:=\frac{\gamma_1}{\beta}<y:=\frac{\beta_1}{\beta}<z:=\frac{\alpha_1}{\beta}<0<1<t:=\frac{\gamma}{\beta}.
\end{align}
Dividing the equations \eqref{R1:1}-\eqref{R1:3} by $\beta^2$ and the equations \eqref{R1:4}-\eqref{R1:5} by $\beta^3$ we obtain
that $t=(R_\mu-R)^{1/2} $ and $(x,y,z)$ solves the following system:
\begin{eqnarray}
x^2-(R_\mu-R)y^2+ (R_\mu-R-1) z^2&=&0\ ,\label{g3:1}\\
R x^2 + (R_\mu-R) y^2 - (1+R)(R_\mu-R)&=& 0\ ,\label{g3:2}\\
(R_\mu-R)(1-y^3)+\frac{R(R_\mu-R-1)}{1+R}z^3&=&\frac{9\eta^2R_\mu}{\beta^3}\ ,\label{g3:3}\\
 (t^3-x^3)-(R_\mu-R)(1-y^3)-(R_\mu-R-1)z^3&=&\frac{9R_\mu}{\beta^3}\ .\label{g3:4}
\end{eqnarray}
We may extract now $x$ and $z$ from \eqref{g3:1} and \eqref{g3:2} and find
 \begin{align}\label{xy3}
  x=-\sqrt\frac{R_\mu-R}{R}\sqrt{(1+R)-y^2}\qquad\text{and}\qquad z=-\sqrt\frac{(1+R)(R_\mu-R)}{R(R_\mu-R-1)}\sqrt{y^2-1}.
 \end{align}
We note that $x$ and $z$ are well-defined and satisfy $x<y<z$ exactly when  $y\in \left(y_0,-1\right)$ where
$$ 
y_0 := -\sqrt{\frac{(1+R)(R_\mu-R)}{R_\mu}} \ .
$$
Eliminating $\beta^3$ from \eqref{g3:3}-\eqref{g3:4} and using \eqref{xy3} we are left with a single equation $\xi_3(R_\mu,y)=0$ for $y$ 
where $\xi_3:(R+1,\infty)\times (y_0,-1)\to\R$ is defined by
\begin{align*}
 \xi_3(R_\mu,y):=&(1+\eta^2)(R_\mu-R)y^3+\eta^2\left(\frac{R_\mu-R}{R}\right)^{3/2}\left( (1+R)-y^2 \right)^{3/2}\\
 &+ \left( R+\eta^2(1+R) \right) \left(\frac{R_\mu-R}{R}\right)^{3/2} \sqrt{\frac{1+R}{R_\mu-R-1}} \left( y^2-1 \right)^{3/2}\\
 &+\eta^2(R_\mu-R)^{3/2}-(1+\eta^2)(R_\mu-R).
\end{align*}
Consequently any solution $(\gamma_1,\beta_1,\alpha_1,\alpha,\beta,\gamma)$ of \eqref{R1:1}-\eqref{R1:5} satisfying \eqref{monteverdi} 
provides a solution $y\in (y_0,-1)$ of $\xi_3(R_\mu,y)=0.$ 
Conversely, if $y\in (y_0,-1)$ is a solution of the equation $\xi_3(R_\mu,y)=0,$ then we set $t=(R_\mu-R)^{1/2} $ and define $x$ and $z$ by \eqref{xy3},
$\beta$ by \eqref{g3:3}, and $(\gamma_1,\beta_1,\alpha_1,\gamma)$ by \eqref{nv3}.
The sextuplet $(\gamma_1,\beta_1,\alpha_1,\alpha,\beta,\gamma)$ thus constructed is a solution of \eqref{R1:1}-\eqref{R1:5} satisfying \eqref{monteverdi}.
In addition, we infer from \eqref{g3:3} and \eqref{xy3} that
\begin{equation}
\frac{9 \eta^2 R_\mu}{\beta ^3} = B(y) := (R_\mu-R) \left( 1-y^3 \right) - (R_\mu-R)^{3/2} \sqrt{\frac{1+R}{R(R_\mu-R-1)}} \left( y^2 - 1 \right)^{3/2}\ . \label{beethoven}
\end{equation}
Then
$$
B'(y) = - 3 y (R_\mu-R) B_1(y) \;\;\text{ with }\;\; B_1(y) := y + \sqrt{\frac{(1+R)(R_\mu-R)}{R(R_\mu-R-1)}} \left( y^2 - 1 \right)^{1/2}\ .
$$
Since $B_1'(y)<0$ for $y\in (y_0,-1)$, we deduce that 
$$
-1 = B_1(-1) < B_1(y) < B_1(y_0) = 0\ , \quad y\in (y_0,-1)\ ,
$$
so that $B$ is decreasing on $(y_0,1)$. Owing to the monotonicity of the left-hand side of \eqref{beethoven}, we conclude that there is a one-to-one
correspondence between the  solutions $(\gamma_1,\beta_1,\alpha_1,\alpha,\beta,\gamma)$ of \eqref{R1:1}-\eqref{R1:5} satisfying \eqref{monteverdi}
and the solutions $y\in (y_0,-1)$ of $\xi_3(R_\mu,y)=0.$

We now proceed to determine the latter. To this end, notice that
\begin{align*}
 \xi_3(R_\mu,-1) & = 2\eta^2(R_\mu-R)\left(\sqrt{R_\mu-R}-\frac{1+\eta^2}{\eta^2}\right) \\
& = 2\eta^2(R_\mu-R)\left(\sqrt{R_\mu-R}- \sqrt{R_\mu^+(R,\eta)-R}\right),\\
\xi_3(R_\mu,y_0) & = (R_\mu-R)\left[\sqrt{R_\mu-R}\left(\eta^2-\sqrt{\frac{1+R}{R_\mu}}\right)-(1+\eta^2)\right] \\
& = \frac{\xi_0(R_\mu,0)}{1+R},
\end{align*}
where $\xi_0(R_\mu,\cdot)$ is defined in Lemma~\ref{P:NSC}, see also \eqref{0}. We infer from the proof of Lemma~\ref{P:NSC} that
\begin{equation}
\label{Le}
\xi_3(R_\mu^M(R,\eta), y_0)=0\ , \quad \xi_3(R_\mu, y_0) \left( R_\mu - R_\mu^M(R,\eta) \right) >0 \;\text{ for }\; R_\mu\ne R_\mu^M(R,\eta)\ ,
\end{equation}
while it is  easy to see from the above formula that
\begin{equation}
\label{Ra}
\xi_3(R_\mu^+(R,\eta), -1)=0\ , \quad \xi_3(R_\mu, -1)  \left( R_\mu  - R_\mu^+(R,\eta)\right) >0 \;\text{ for }\; R_\mu\ne R_\mu^+(R,\eta)\ .
\end{equation}
Moreover, we have that $\p_y\xi_3(R_\mu,y)=3y^2\xi_4(R_\mu,1/y^2),$ with $\xi_4:(R+1,\infty)\times(1/y_0^2,1)\to\R$ being defined by
 \begin{align*}
\xi_4(R_\mu,r):= & \ (1+\eta^2)(R_\mu-R) +\eta^2\left(\frac{R_\mu-R}{R}\right)^{3/2}\sqrt{(1+R)r-1}\\
 & - \left( R+\eta^2(1+R) \right) \left(\frac{R_\mu-R}{R}\right)^{3/2} \sqrt{\frac{1+R}{R_\mu-R-1}}\ \sqrt{1-r}\ .
 \end{align*}
A simple computation reveals that $\xi_4$ is increasing with $\xi_4(R_\mu,r)\ \to \ 0$ as $r\to 1/y_0^2.$ Consequently, $\xi_3(R_\mu,\cdot)$ is an increasing function which, together with \eqref{Le} and \eqref{Ra}, gives the expected result.
\end{proof}

%%%%%%%%%%%%%%%%%%%%%%%%%%%%%%%%%%%%%%%%
%%%%%%%%%%%%%%%%%%%%%%%%%%%%%%%%%%%%%%%%
\bibliographystyle{siam}
\bibliography{BP}

\begin{thebibliography}{10}

\bibitem{AIJMxx}
{\sc J.~Alkhayal, S.~Issa, M.~Jazar, and R.~Monneau}, {\em {Existence results
  for degenerate cross-diffusion systems with application to seawater
  intrusions}}.
\newblock preprint, 2014.

\bibitem{BPW92}
{\sc F.~Bernis, L.~A. Peletier, and S.~M. Williams}, {\em {Source type
  solutions of a fourth order nonlinear degenerate parabolic equation}},
  Nonlinear Anal., 18 (1992), pp.~217--234.

\bibitem{FF12}
{\sc M.~Bessemoulin-Chatard and F.~Filbet}, {\em {A finite volume scheme for
  nonlinear degenerate parabolic equations}}, SIAM J. Sci. Comput., 34 (2012),
  pp.~B559--B583.

\bibitem{BCKKLLxx}
{\sc A.~Blanchet, J.~A. Carrillo, D.~Kinderlehrer, M.~Kowalczyk, {\relax
  Ph}.~{Lauren\c cot}, and S.~Lisini}, {\em {A hybrid variational principle for
  the {K}eller-{S}egel system in $\mathbb{R}^2$}}.
\newblock submitted.

\bibitem{BL13}
{\sc A.~Blanchet and {\relax Ph}.~Lauren\c{c}ot}, {\em {The parabolic-parabolic
  {K}eller-{S}egel system with critical diffusion as a gradient flow in
  {$\mathbb R^d,\ d\ge3$}}}, Comm. Partial Differential Equations, 38 (2013),
  pp.~658--686.

\bibitem{BMPB10}
{\sc M.~Burger, M.~{Di Francesco}, J.-F. Pietschmann, and B.~Schlake}, {\em
  {Nonlinear cross-diffusion with size exclusion}}, SIAM J. Math. Anal., 42
  (2010), pp.~2842--2871.

\bibitem{CU07}
{\sc E.~A. Carlen and S.~Ulusoy}, {\em {Asymptotic equipartition and long time
  behavior of solutions of a thin-film equation}}, J. Differential Equations,
  241 (2007), pp.~279--292.

\bibitem{CJMTA01}
{\sc J.~A. Carrillo, A.~J\"{u}ngel, P.~A. Markowich, G.~Toscani, and
  A.~Unterreiter}, {\em {Entropy dissipation methods for degenerate parabolic
  problems and generalized {S}obolev inequalities}}, Monatsh. Math., 133
  (2001), pp.~1--82.

\bibitem{CT00}
{\sc J.~A. Carrillo and G.~Toscani}, {\em {Asymptotic {$L^1$}-decay of
  solutions of the porous medium equation to self-similarity}}, Indiana Univ.
  Math. J., 49 (2000), pp.~113--142.

\bibitem{CT02}
\leavevmode\vrule height 2pt depth -1.6pt width 23pt, {\em {Long-time
  asymptotics for strong solutions of the thin film equation}}, Comm. Math.
  Phys., 225 (2002), pp.~551--571.

\bibitem{Const93}
{\sc P.~Constantin, T.~F. Dupont, R.~E. Goldstein, L.~P. Kadanoff, M.~J.
  Scheley, and S.-M. Zhou}, {\em {Droplet breakup in a model of the
  {H}ele-{S}haw cell}}, Physical Review E, 47 (1993), pp.~4169--4181.

\bibitem{CG10}
{\sc D.~C\'{o}rdoba and F.~Gancedo}, {\em {Absence of squirt singularities for
  the multi-phase {M}uskat problem}}, Comm. Math. Phys., 299 (2010),
  pp.~561--575.

\bibitem{ELM11}
{\sc J.~Escher, {\relax Ph}.~{Lauren\c cot}, and B.-V. Matioc}, {\em {Existence
  and stability of weak solutions for a degenerate parabolic system modelling
  two-phase flows in porous media}}, Ann. Inst. H. Poincar\'{e} Anal. Non
  Lin\'{e}aire, 28 (2011), pp.~583--598.

\bibitem{EMM12a}
{\sc J.~Escher, A.-V. Matioc, and B.-V. Matioc}, {\em {A generalized
  {R}ayleigh-{T}aylor condition for the {M}uskat problem}}, Nonlinearity, 25
  (2012), pp.~73--92.

\bibitem{EMM12}
\leavevmode\vrule height 2pt depth -1.6pt width 23pt, {\em {Modelling and
  analysis of the {M}uskat problem for thin fluid layers}}, J. Math. Fluid
  Mech., 14 (2012), pp.~267--277.

\bibitem{EM12x}
{\sc J.~Escher and B.-V. Matioc}, {\em {Existence and stability of solutions
  for a strongly coupled system modelling thin fluid films}}, NoDEA Nonlinear
  Differential Equations Appl., 20 (2013), pp.~539--555.

\bibitem{Esteban12}
{\sc B.~Esteban, J.-R. Riba, G.~Baquero, A.~Rius, and R.~Puig}, {\em
  {Temperature dependence of density and viscosity of vegetable oils}}, Biomass
  and Bioenergy, 42 (2012), pp.~164--171.

\bibitem{GO01}
{\sc L.~Giacomelli and F.~Otto}, {\em {Variational formulation for the
  lubrication approximation of the {H}ele-{S}haw flow}}, Calc. Var. Partial
  Differential Equations, 13 (2001), pp.~377--403.

\bibitem{JM14}
{\sc M.~Jazar and R.~Monneau}, {\em {Derivation of seawater intrusion models by
  formal asymptotics}}, SIAM J. Appl. Math., 74 (2014), pp.~1152--1173.

\bibitem{K76}
{\sc S.~Kamin}, {\em {The asymptotic behavior of the solution of the filtration
  equation}}, Israel J. Math., 14 (1973), pp.~76--87.

\bibitem{K75}
\leavevmode\vrule height 2pt depth -1.6pt width 23pt, {\em {Similar solutions
  and the asymptotics of filtration equations}}, Arch. Rational Mech. Anal., 60
  (1975/76), pp.~171--183.

\bibitem{LM12xx}
{\sc {\relax Ph}.~{Lauren\c cot} and B.-V. Matioc}, {\em {A thin film
  approximation of the Muskat problem with gravity and capillary forces}}, J.
  Math. Soc. Japan,  (2014).
\newblock to appear.

\bibitem{LM12x}
{\sc {\relax Ph}.~Lauren\c{c}ot and B.-V. Matioc}, {\em {A gradient flow
  approach to a thin film approximation of the {M}uskat problem}}, Calc. Var.
  Partial Differential Equations, 47 (2013), pp.~319--341.

\bibitem{BM12}
{\sc B.-V. Matioc}, {\em {Non-negative global weak solutions for a degenerate
  parabolic system modelling thin films driven by capillarity}}, Proc. Roy.
  Soc. Edinburgh Sect. A, 142 (2012), pp.~1071--1085.

\bibitem{MMS09}
{\sc D.~Matthes, R.~J. McCann, and G.~Savar\'{e}}, {\em {A family of nonlinear
  fourth order equations of gradient flow type}}, Comm. Partial Differential
  Equations, 34 (2009), pp.~1352--1397.

\bibitem{Mu34}
{\sc M.~Muskat}, {\em {Two fluid systems in porous media. The encroachment of
  water into an oil sand}}, Physics, 5 (1934), pp.~250--264.

\bibitem{Ne84}
{\sc W.~I. Newman}, {\em {A {L}yapunov functional for the evolution of
  solutions to the porous medium equation to self-similarity. {I}}}, J. Math.
  Phys., 25 (1984), pp.~3120--3123.

\bibitem{Or97}
{\sc A.~Oron, S.~H. Davis, and S.~G. Bankoff}, {\em {Long-scale evolution of
  thin liquid films}}, Rev. Mod. Phys., 69 (1997), pp.~931--980.

\bibitem{Ot98}
{\sc F.~Otto}, {\em {Dynamics of labyrinthine pattern formation in magnetic
  fluids: a mean-field theory}}, Arch. Rational Mech. Anal., 141 (1998),
  pp.~63--103.

\bibitem{Ot01}
\leavevmode\vrule height 2pt depth -1.6pt width 23pt, {\em {The geometry of
  dissipative evolution equations: the porous medium equation}}, Comm. Partial
  Differential Equations, 26 (2001), pp.~101--174.

\bibitem{Ra84}
{\sc J.~Ralston}, {\em {A {L}yapunov functional for the evolution of solutions
  to the porous medium equation to self-similarity. {II}}}, J. Math. Phys., 25
  (1984), pp.~3124--3127.

\bibitem{Si87}
{\sc J.~Simon}, {\em {Compact sets in the space {$L^p(0,T;B)$}}}, Ann. Mat.
  Pura Appl. (4), 146 (1987), pp.~65--96.

\bibitem{Va07}
{\sc J.~L. V\'{a}zquez}, {\em {The porous medium equation. Mathematical
  theory}}, {Oxford Mathematical Monographs}, The Clarendon Press Oxford
  University Press, Oxford, 2007.

\bibitem{ZL99}
{\sc W.~W. Zhang and J.~R. Lister}, {\em {Similarity solutions for van der
  {W}aals rupture of a thin film on a solid substrate}}, Physics of Fluids, 11
  (1999), pp.~2454--2462.

\bibitem{Zixx}
{\sc J.~Zinsl}, {\em {Existence of solutions for a nonlinear system of
  parabolic equations with gradient flow structure}}, Monatsh. Math., 174
  (2014), pp.~653--679.

\end{thebibliography}
%%%%%%%%%%%%%%%%%%%%%%%%%%%%%%%%%%%%%%%%
%%%%%%%%%%%%%%%%%%%%%%%%%%%%%%%%%%%%%%%%

\end{document}